\documentclass[10pt]{amsart} % change to \documentclass[a4paper,11pt]{amsart} if necessary. The current one is suitable for legal size paper.
\usepackage[lmargin=1in,rmargin=1in,tmargin=1in,bmargin=1in]{geometry}
\usepackage[ps,all,arc,rotate]{xy}
\usepackage{graphicx, float, epstopdf}
\usepackage{centernot}
\usepackage{fancyhdr}
\usepackage[utf8]{inputenc}
\usepackage{amsfonts,amssymb,amsmath,amsthm,mathrsfs}
\usepackage{graphics, setspace}
\usepackage[usenames,dvipsnames]{xcolor}
\usepackage{datetime}
\numberwithin{equation}{section}
\numberwithin{figure}{section}
\allowdisplaybreaks[4]                        %allows breaking multiline environments in amsmath commands
 
\newtheorem{lemma}{Lemma}[section]
\newtheorem{theorem}{Theorem}[section]

\newtheorem{conjecture}[lemma]{Conjecture}

\theoremstyle{definition}

\newtheorem{remark}{Remark}[section]
\newcommand{\R}{\mathbb{R}}

\newcommand{\N}{\mathbb{N}}

\newcommand{\e}{\operatorname{e}}
\newcommand{\real}{\operatorname{Re}}
\newcommand{\imag}{\operatorname{Im}}
\newcommand{\sumtwo}{\operatorname*{\sum\sum}}

\newcommand{\sumdots}{\operatorname*{\sum\cdots\sum}}

\newcommand{\ointdots}{\operatorname*{\oint\cdots\oint}}

\newcommand{\barell}{{\bar \ell}}
\newcommand{\barr}{{\bar r}}

\newcommand{\barq}{{\bar q}}

\newcommand{\barL}{{\bar L}}

\begin{document}
\title[More than five-twelfths of the zeros of $\zeta$ are on the critical line]{More than five-twelfths of the zeros of $\zeta$ \\are on the critical line} 
%%%%%%%%%%%%%%%
%%%%%%%%%%%%%%%
\dedicatory{\emph{Dedicated to Brian Conrey on the occasion of the 30th anniversary of his `Two-fifths' paper.}}
\author{Kyle Pratt}
\address{Department of Mathematics, University of Illinois, 1409 West Green Street, Urbana, IL 61801, United States}
\email{kpratt4@illinois.edu}
%%%%%%%%%%%%%%%
\author{Nicolas Robles}
\address{Department of Mathematics, University of Illinois, 1409 West Green Street, Urbana, IL 61801, United States \textnormal{and} Wolfram Research Inc, 100 Trade Center Dr, Champaign, IL 61820, United States}
\email{nirobles@illinois.edu}
\email{nicolasr@wolfram.com}
%%%%%%%%%%%%%%%
\author{Alexandru Zaharescu}
\address{Department of Mathematics, University of Illinois, 1409 West Green Street, Urbana, IL 61801, United States \textnormal{and} Simion Stoilow Institute of Mathematics of the Romanian Academy, P.O. Box 1-764, RO-014700 Bucharest, Romania}
\email{zaharesc@illinois.edu} 
%%%%%%%%%%%%%%%
\author{Dirk Zeindler}
\address{Department of Mathematics and Statistics, Lancaster University, Fylde College, Bailrigg, Lancaster LA1 4YF, United Kingdom}
\email{d.zeindler@lancaster.ac.uk}
%%%%%%%%%%%%%%%
\subjclass[2010]{Primary: 11M26, 11L07; Secondary: 11M06, 05A18. \\ \indent \textit{Keywords and phrases}: Riemann zeta-function, critical line, zeros, mollifier, incomplete Kloosterman sums, autocorrelation ratios, generalized von Mangoldt functions, convolution structure, Bell diagrams.}
\maketitle
%%%%%%%%%%%%%%%%%%%%%%%%%%%%%%%%%%%%%%%%%%%%%%%%%%%%%%%%%%%%%%%%%%%%%%%%%%%%%%%%%%%%%%%%%%%%%%%%%
\begin{abstract}
The second moment of the Riemann zeta-function twisted by a normalized Dirichlet polynomial with coefficients of the form $(\mu \star \Lambda_1^{\star k_1} \star \Lambda_2^{\star k_2} \star \cdots \star \Lambda_d^{\star k_d})$ is computed unconditionally by means of the autocorrelation of ratios of $\zeta$ techniques from Conrey, Farmer, Keating, Rubinstein and Snaith \cite{cfkrs}, Conrey, Farmer and Zirnbauer \cite{cfz} as well as Conrey and Snaith \cite{cs}.
This in turn allows us to describe the combinatorial process behind the mollification of 
\[
\zeta(s) + \lambda_1 \frac{\zeta'(s)}{\log T} + \lambda_2 \frac{\zeta''(s)}{\log^2 T} + \cdots + \lambda_d \frac{\zeta^{(d)}(s)}{\log^d T},
\]
where $\zeta^{(k)}$ stands for the $k$th derivative of the Riemann zeta-function and $\{\lambda_k\}_{k=1}^d$ are real numbers. Improving on recent results on long mollifiers and sums of Kloosterman sums due to Pratt and Robles \cite{pr01}, as an application, we increase the current lower bound of critical zeros of the Riemann zeta-function to slightly over five-twelfths.
\end{abstract}
%%%%%%%%%%%%%%%%%%%%%%%%%%%%%%%%%%%%%%%%%%%%%%%%%%%%%%%%%%%%%%%%%%%%%%%%%%%%%%%%%%%%%%%%%%%%%%%%%
\tableofcontents
\section{Introduction}
\subsection{Integral moments and autocorrelation ratios of $L$-functions}
Although applications of random matrix theory in number theory started with Montgomery's pair correlation conjecture \cite{montgomery} in the mid 1970's, it is during the last two decades that the use of random matrices as a tool in the study of $L$-functions has become indispensable.\\

A local statistic is a figure that involves exclusively correlations between zeros separated on a scale of a few mean spacings. In \cite{montgomery}, Montgomery conjectured that in the limit of a large height $T$ on the critical line, any local statistic is supplied by the associated statistic for eigenvalues from the Gaussian unitary ensemble (GUE). These conjectures were numerically tested by Odlyzko and found to have incredible agreement \cite{odlyzko}. The leading order of statistics involving zeros of $\zeta$ and statistics involving eigenvalues is identical and asymptotically no factors of arithmetical nature appear \cite[p. 594]{cs}. On the other hand, from the work of Bogomolny and Keating \cite{bogokeat}, it is expected that arithmetical contributions will be significant in the lower order terms.\\

The zeros of $\zeta$ are not the only quantities of interest, naturally one is interested in the zeros of other $L$-functions and their assocations to other types of matrices. Katz and Sarnak \cite{katzsarnak, katzsarnakbull} proposed that local statistics of zeros of families of $L$-functions could be understood by the eigenvalues of matrices coming from classical compact groups, see also the work of Rudnick and Sarnak in \cite{rudnicksarnak}. It is thus believed that families of $L$-functions can be modeled by the characteristic polynomials from such groups. These could be unitary, sympletic or orthogonal. The calculation performed by Iwaniec, Luo and Sarnark \cite{iwaniecluosarnak} for the one-level densities of families of $L$-functions with each symmetry type was in agreement with random matrix theory and further showed that there is no arithmetic component in the leading terms.\\

It is only recently that global, rather local, statistics were contrasted with classical compact groups. A distinct feature of global statistics is that an arithmetical factor \textsl{does} appear in the leading order terms.\\

Keating and Snaith argued in groundbreaking papers \cite{keatingsnaith1, keatingsnaith2} (first conjectured by Conrey and Ghosh \cite{conreyghosh}) that the leading terms of the moments of an $L$-function are the product of a characteristic polynomial from a random matrix and an Euler product. Namely, for the general $2k$ moment we expect that
\begin{align*}
\int_0^T |\zeta(\tfrac{1}{2}+it)|^{2k} dt \sim \frac{a_k g_k}{\Gamma(k^2+1)} T \log^{k^2}T.
\end{align*}
where $a_k$ is the Euler product
\begin{align*}
a_k = \prod_p \bigg( \bigg( 1-\frac{1}{p} \bigg)^{k^2} \sum_{r=0}^\infty \frac{d_k^2(p^r)}{p^r} \bigg) = \prod_p \bigg(1-\frac{1}{p}\bigg)^{k^2} \sum_{k=0}^\infty \binom{k+m-1}{m}^2 p^{-m}.
\end{align*}
The value of $g_k$ is the quantity associated with the eigenvalues of random matrices. Prior to \cite{keatingsnaith1, keatingsnaith2}, Conrey and Ghosh \cite{conreyghosh} had conjectured that $g_3=42$ and that in general $g_k$ is an integer. Conrey and Gonek \cite{conreygonek} later conjectured that $g_4 = 24024$. Using the above mentioned techniques from random matrix theory, Keating and Snaith conjectured a closed formula for $g_k$ which is given by the product
\begin{align*}
g_k = k^2! \prod_{j=0}^{k-1} \frac{j!}{(k+j)!}.
\end{align*}
Proving that the above candidate formula for $g_k$ is an integer is not a trivial matter, \cite[p. 196]{farmer3}.\\

This type of result can be generalized by considering averages of ratios of products of $L$-functions (on the number theoretical side) or of characteristic polynomials (on the random matrix side). These results were mostly established and illustrated, sometimes conjecturally when it comes to the number theoretical aspect, by Conrey, Farmer, Keating, Rubinstein and Snaith \cite{cfkrs}, Conrey, Farmer and Zirnbaheur \cite{cfz}, and Conrey and Snaith \cite{cs} among others.\\

These conjectures on the averages of ratios of products of $\zeta$ functions (autocorrelation of ratios) are useful not only for global statistics of zeros, as one would naturally expect, but also for local statistics. It has been suggested in fact that autocorrelation ratios of characteristic polynomials are more fundamental to random matrices than correlation functions (\cite{borodinstrahov} and \cite[p. 595]{cs}). Thus the same thing could be argued for autocorrelation ratios in the $L$-function universe. The reason why autocorrelation ratios are useful is because they provide many local or global statistic ($n$-level correlations, discrete moments, etc...). Moreover, the ratios conjectures imply Montgomery's pair correlation conjecture \cite[p. 594]{cfz}, and they contain additional information that can be utilized to make precise conjectures about the distributions of zeros of $L$-functions. \\

Autocorrelation ratios with usually one or two $L$-functions in the numerator and in the denominator are enough to cover a very wide spectrum of applications, but there is no limit to the size they can accommodate. A salient feature of this paper is that later we will need many zeta functions in both the numerator and in the denominator and this carries a heavy combinatorial price.
On the other hand, Bogomolny and Keating \cite{bogokeat} needed a heavy duty analysis of the Hardy-Littlewood prime pair conjectures to obtain the lower order terms of local statistics where arithmetical components appear. A nice feature of autocorrelations of ratios of $L$-functions is that they bypass those issues \cite[p. 595]{cs}.\\

Mollifiers are fundamentally important objects in the study of the moments of zeta and their arithmetic consequences. There will be plenty to say about this in a moment, but, roughly speaking, they are used to mine information about small values of $L$-functions, specially zeros, as well as to bound the number of zeros either in a vertical strip to the right of $\real(s) = \frac{1}{2}$ or at $\real(s)=\frac{1}{2}$. Mollifiers have also been employed to extract non-vanishing results at the central point for families of $L$-functions (see, among very many examples, \cite{kpv1, kpv2, milinovich, sound}).\\

Unfortunately, none of these results ever comes cheap. Even the simplest examples require sophisticated and very long analysis and regrettably this paper is no exception. Improvements on the underlying technology have somewhat decreased the length and complexity of the calculations. For example, using the autocorrelation ratios technique, Young \cite{youngshort} was able to shorten Levinson's original proof that more than one third of the zeros of zeta are on the critical line from fifty pages to eight\footnote{At one point in Levinson's original paper there are twenty four cancellations going on simultaneously! See \cite[p. 308]{levinsoncollected} for further details.}. Further refinements on mollifiers still require lengthy calculations, however.\\

What is surprising is that unlike other averages of families considered in \cite[p. 596]{cs}, `there does not seem to be a random matrix analogue of mollifying as there is nothing that naturally corresponds to a partial Dirichlet series'.\\

Before we move on to describe the mechanism of autocorrelation ratios, we mention that the ratios can also be used to study moments of $|\zeta'(\rho+a)|$ and allied quantities. Specifically, Conrey and Snaith showed how to obtain all the other lower order terms for these averages in \cite[$\mathsection$7]{cs}.\\

While difficult computations can be simplified with the autocorrelation ratios, it must be stressed that one needs to assume the Riemann hypothesis (RH), or generalized Riemann hypothesis depending on the $L$-function, and therefore there is a limit to how useful they are. In this paper, we provide unconditional results by using the \textsl{underlying} techniques and ideas behind the autocorrelation of ratios of $\zeta$ but without using the conjectures themselves. This is of particular importance because the application we provide is an improvement on the proportion of zeros on the Riemann zeta-function on the critical line, see Theorem \ref{512}. Naturally, RH cannot be assumed for this type of application. Having said this, it is also important to add that, as an illustrative check, Conrey and Snaith obtained the leading terms of the simplest mollified moment in \cite[$\mathsection$5.1]{cs} as well as higher mollified moments \cite[$\mathsection$6]{cs} under the ratios conjecture (and hence under RH).

\subsection{The ratios conjecture}
Since we are only concerned with the Riemann zeta-function, we need not step outside the unitary family. Assume the Riemann hypothesis, set $s=\frac{1}{2}+it$ and let us follow $\mathsection$2.1 of \cite{cs}. Farmer \cite{farmer1, farmer2} was the first to put forward the asymptotic conjecture
\begin{align*}
R(\alpha,\beta,\gamma,\delta) := \int_0^T \frac{\zeta(s+\alpha)\zeta(1-s+\beta)}{\zeta(s+\gamma)\zeta(1-s+\delta)}dt \sim T \frac{(\alpha+\delta)(\beta+\gamma)}{(\alpha+\beta)(\gamma+\delta)} - T^{1-\alpha-\beta} \frac{(\delta-\beta)(\gamma-\alpha)}{(\alpha+\beta)(\gamma+\delta)}
\end{align*}
as $T \to \infty$, provided that $\real(\gamma), \real(\delta)>0$.

The approximate functional equation states that
\begin{align} \label{AFE}
\zeta(s) = \sum_{n \le X} \frac{1}{n^s} + \chi(s) \sum_{n \le Y} \frac{1}{n^{1-s}} + O(R),
\end{align}
where $R$ is a remainder and $XY=t/(2\pi)$. Now, use \eqref{AFE} for the zeta functions that appear in the numerator of the integrand and use ordinary Dirichlet series for the zeta functions in the denominator ($\frac{1}{\zeta(s)} = \sum \mu(n)n^{-s}$). A rule of thumb (the so-called `recipe' \cite[p. 52]{cfkrs}) tells us we only need to be concerned with the pieces for which there is the same number of $\chi(s)$ and $\chi(1-s)$ due to oscillations. The next step is to integrate term-by-term and keep only the diagonal pieces and complete all the sums that we arrive at. For the first term of \eqref{AFE}, this procedure boils down to
\[
\sum_{hm=kn} \frac{\mu(h)\mu(k)}{m^{1/2+\alpha}n^{1/2+\beta}h^{1/2+\gamma}k^{1/2+\delta}} = \prod_p \sum_{h+m=k+n} \frac{\mu(p^h)\mu(p^k)}{p^{(1/2+\alpha)m+(1/2+\beta)n+(1/2+\gamma)h+(1/2+\delta)k}}.
\]
This is the perennial expression that appears, in some way or another, in all calculations involving autocorrelations of ratios of $L$-functions and it is what allowed Young to simplify Levinson's proof from fifty to eight pages. Now, we only have $0$ and $1$ as possibilities for $h$ and $k$. Thus a simple analysis (it will not be this easy again later) shows that the sum on the right-hand side is equal to
\begin{align*}
\frac{1}{1-\frac{1}{p^{1+\alpha+\beta}}} \bigg(1 - \frac{1}{p^{1+\beta+\gamma}} - \frac{1}{p^{1+\alpha+\delta}}+\frac{1}{p^{1+\gamma+\delta}} \bigg).
\end{align*}
This means that the Euler product on the right-hand side is given by the following ratio of products of $\zeta$
\begin{align*}
\frac{\zeta(1+\alpha+\beta)\zeta(1+\gamma+\delta)}{\zeta(1+\alpha+\delta)\zeta(1+\beta+\delta)} A(\alpha,\beta,\gamma,\delta),
\end{align*}
where $A$ is the `arithmetical factor'
\begin{align*}
A(\alpha,\beta,\gamma,\delta) = \prod_p \frac{(1-\frac{1}{p^{1+\gamma+\delta}})(1-\frac{1}{p^{1+\beta+\gamma}}-\frac{1}{p^{1+\alpha+\delta}}-\frac{1}{p^{1+\gamma+\delta}})}{(1-\frac{1}{p^{1+\beta+\gamma}})(1-\frac{1}{p^{1+\alpha+\delta}})}.
\end{align*}
As will become clearer later in our exposition, the piece from the other term coming from \eqref{AFE} is essentially the same except that $\alpha$ is replaced by $-\beta$, $\beta$ is replaced by $-\alpha$ and it is affected by a multiplication by
\[
\chi(s+\alpha)\chi(1-s+\beta) = \bigg(\frac{t}{2\pi}\bigg)^{-\alpha-\beta} \bigg( 1+O\bigg(\frac{1}{|t|}\bigg)\bigg).
\]
%by the use of Stirling's formula. 
These manipulations allowed Conrey, Farmer and Zirnbaeur \cite{cfz} to obtain a more precise ratios conjecture.

\begin{conjecture}[Conrey, Farmer and Zirnbaeur, 2006] \label{cfz2006conjecture}
If $-\frac{1}{4} < \real(\alpha) < \frac{1}{4}$, $\frac{1}{\log T} \ll \real(\delta) < \frac{1}{4}$ and $\imag(\alpha),\imag(\delta) \ll_\varepsilon T^{1-\varepsilon}$ for every $\varepsilon>0$, then
\begin{align*}
R(\alpha,\beta,\gamma,\delta) &= \int_0^T \frac{\zeta(1+\alpha+\beta)\zeta(1+\gamma+\delta)}{\zeta(1+\alpha+\delta)\zeta(1+\beta+\delta)} A(\alpha,\beta,\gamma,\delta) \\
& \quad + \bigg(\frac{t}{2\pi}\bigg)^{-\alpha-\beta} \frac{\zeta(1-\alpha-\beta)\zeta(1+\gamma+\delta)}{\zeta(1-\beta+\delta)\zeta(1-\alpha+\delta)} A(-\beta,-\alpha,\gamma,\delta)dt + O(T^{1/2+\varepsilon}).
\end{align*}
\end{conjecture}

The key to obtaining lower order term in the pair correlations is embedded in the above conjecture. One needs to differentiate with respect to $\alpha$ and $\beta$ and then set $\gamma = \alpha$ and $\delta = \beta$. It is important (and substantially more so later on) to note that $A(\alpha,\beta,\alpha,\beta)=1$. It is also useful to see that
\[
\frac{\partial}{\partial \alpha} \frac{f(\alpha, \gamma)}{\zeta(1-\alpha+\gamma)}\bigg|_{\gamma=\alpha} = -f(\alpha,\alpha).
\]
Unfortunately, we will not have recourse to such neat formulas in our analysis. This differentiation process turns the above into the following.
\begin{theorem}[Conrey and Snaith, 2007]
If Conjecture \textnormal{\ref{cfz2006conjecture}} is true, then
\begin{align*}
\int_0^T &\frac{\zeta'}{\zeta}(s+\alpha)\frac{\zeta'}{\zeta}(1-s+\beta)dt \\
& = \int_0^T \bigg( \bigg(\frac{\zeta'}{\zeta}(1+\alpha+\beta)\bigg)'  + \bigg(\frac{t}{2\pi}\bigg)^{-\alpha-\beta} \zeta(1+\alpha+\beta)\zeta(1-\alpha-\beta) \prod_p \frac{(1-\frac{1}{p^{1+\alpha+\beta}})(1-\frac{2}{p}+\frac{1}{p^{1+\alpha+\beta}})}{(1-\frac{1}{p})^2} \\
& \quad - \sum_p \bigg(\frac{\log p}{p^{1+\alpha+\beta}-1}\bigg)^2 \bigg)dt + O(T^{1/2+\varepsilon}),
\end{align*}
provided that $1/\log T \ll \real(\alpha),\real(\beta)<\frac{1}{4}$.
\end{theorem}
The last sum over $p$ will appear frequently in the latter sections.\\

One example of moments of logarithmic derivatives is taken from \cite[p. 628]{cfz}. It illustrates the presence of the arithmetical factor. Assuming a variant of the ratios conjecture  one has 
\begin{align*}
\frac{1}{T} \int_0^T \bigg|\frac{\zeta'}{\zeta}\bigg(\frac{1}{2}+r+it\bigg)\bigg|^2 dt &= \bigg(\frac{\zeta'}{\zeta}\bigg)' (1+2r) + \bigg(\frac{T}{2\pi}\bigg)^{-2r} A(-r,-r,r,r) \frac{\zeta(1-2r)\zeta(1+2r)}{1-2r} \\
& \quad + c(r) + O(T^{-1/2+\varepsilon}),
\end{align*}
where $c(r)$ is a function of $r$ which is uniformly bounded for $|r| < 1/4-\varepsilon$ and is given by
\begin{align*}
c(r) = \sum_p \bigg(\frac{-p^{1+2r}\log^2 p}{(p^{1+2r}-1)^2} + \int_0^1 \frac{\log^2 p}{(\e(\theta)p^{1/2+r}-1)^2} d\theta \bigg).
\end{align*}

Oddly enough, although we are working in the context of global statistics and we thus expect arithmetical terms to be present, it so happens that only the simplest ones survive after undergoing a certain combinatorial process. Indeed, the arithmetical factors $A$ and their derivatives, which are sums over primes like the one above, conspire to either become zero or to get absorbed in an error term, thus (luckily) leaving us only with terms for which the arithmetical factor is equal to one.\\

It is important to mention the celebrated `Five authors' (Conrey, Farmer, Keating, Rubinstein, and Snaith) conjecture \cite[p. 44]{cfkrs} regarding the $2k$ moments of zeta. First, we recall the Vandermonde
\begin{align*}
\Delta(z_1, \cdots, z_m) = \prod_{1 \le i<j \le m} (z_j - z_i)
\end{align*}
and the notation $\e(z) = e^{2\pi i z}$.

\begin{conjecture}[CFKRS, 2005] \label{CFKRSconjecture}
Suppose $g(t)$ is a suitable weight function. Then
\begin{align*}
\int_{-\infty}^\infty |\zeta(\tfrac{1}{2}+it)|^{2k} g(t) dt = \int_{-\infty}^\infty P_k \bigg( \log \frac{t}{2 \pi} \bigg) (1+O(t^{-1/2+\varepsilon}))g(t)dt,
\end{align*}
where $P_k$ is a polynomial of degree $k^2$ given by the $2k$-fold residue
\begin{align*}
P_k(x) = \frac{(-1)^k}{(k!)^2} \frac{1}{(2 \pi i)^{2k}} \ointdots \frac{G(z_1, \cdots, z_{2k})\Delta^2(z_1, \cdots, z_{2k})}{\prod_{j=1}^{2k} z_j ^{2k}} e^{(x/2) \sum_{j=1}^k z_j - z_{k+j}} dz_1 \cdots z_{2k},
\end{align*}
where one integrates over small circles about $z_i=0$, with
\begin{align*}
G(z_1, \cdots, z_{2k}) = A_k(z_1, \cdots, z_{2k}) \prod_{i=1}^k \prod_{j=1}^k \zeta(1+z_i-z_{k+j}),
\end{align*}
and $A_k$ is the Euler product
\begin{align*}
A_k(z) = \prod_p \prod_{i=1}^k \prod_{j=1}^k \bigg(1-\frac{1}{p^{1+z_i-z_{k+j}}}\bigg) \int_0^1 \prod_{j=1}^k \bigg(1 - \frac{\e(\theta)}{p^{1/2+z_j}} \bigg)^{-1}  \bigg(1 - \frac{\e(-\theta)}{p^{1/2-z_{k+j}}} \bigg)^{-1} d\theta.
\end{align*}
More generally
\begin{align*}
I_{\zeta,\alpha,g} := \int_{-\infty}^\infty \zeta(\tfrac{1}{2}+\alpha_1+it) \cdots \zeta(\tfrac{1}{2}+\alpha_{2k}+it) g(t) dt = \int_{-\infty}^\infty P_k \bigg( \log \frac{t}{2 \pi}, \alpha \bigg) (1+O(t^{-1/2+\varepsilon}))g(t)dt,
\end{align*}
where
\begin{align*}
P_k(x, \alpha) = \frac{(-1)^k}{(k!)^2} \frac{1}{(2 \pi i)^{2k}} \ointdots \frac{G(z_1, \cdots, z_{2k})\Delta^2(z_1, \cdots, z_{2k})}{\prod_{j=1}^{2k} \prod_{i=1}^{2k} (z_j-\alpha_i)} e^{(x/2) \sum_{j=1}^k z_j - z_{k+j}} dz_1 \cdots z_{2k},
\end{align*}
with the path of integration being small circles surrounding the poles $\alpha_i$.
\end{conjecture}

This conjecture displays the rich structure behind the moments of zeta. We shall be needing a special type of moment related to the case $k=1$ for our purposes and we have chosen to write our result (Theorem \ref{theoremmaintermerror47}) in a way that parallels the structure of Conjecture \ref{CFKRSconjecture}.\\

Moreover in \cite{cfkrs, cfz, cs}, a lot of effort is invested in explicating the combinatorial structure of the permutations sums, arithmetical factors as well as double products that appear in certain formulae (notable results in this direction are given by \cite[Lemma 2.5.1, $\mathsection$2.7]{cfkrs}, \cite[$\mathsection$5]{cs} and \cite[$\mathsection$6.4, $\mathsection$7.2]{cfz}). In our findings we also come across formulae and concepts that require a similar effort but for which the existing ideas that have appeared in the literature do not seem to apply directly as far as the enumeration and the combinatorics are concerned.\\

Lastly, in the words of Conrey and Snaith \cite[p. 596]{cs} `before embarking on such a [moment] calculation, it would be useful to know ahead of time what the answer is'. Calling it a `painful calculation', as they do, is nothing short of accurate. The presentation we have decided to adopt follows this philosophy closely. We have started with simple examples where the combinatorics are undemanding so that objects can be counted `by hand' before moving on to the general principles. Even when considering the general principles, we have paused at critical steps to fall back to special cases (which have not have appeared in the literature before) to better illustrate the underlying blueprint of our results.

\subsection{Motivation and choice of the mollifiers}
We set $s = \sigma+it$ with $\sigma, t \in \R$ and denote by $\zeta(s) = \sum_{n\ge 1}n^{-s}$ the Riemann zeta-function for $\sigma>1$, and otherwise by analytic continuation. Now, let
\begin{itemize}
\item $N(T)$ denote the number of zeros $\rho = \beta + i \gamma$, counted with multiplicity, of $\zeta(s)$ inside the rectangle $0 < \beta < 1$ and $0 < \gamma < T$,
\item $N_0(T)$ denote the number of zeros, counted with multiplicity, of $\zeta(s)$ such that $\beta = \tfrac{1}{2}$ and $0 < \gamma < T$.
\end{itemize}
It is well-known \cite[Ch. IX]{titchmarsh} that the asymptotic formula for $N(T)$ is given
\[
N(T) = \frac{T}{2 \pi} \bigg(\log\frac{T}{2\pi}-1\bigg) + \frac{7}{8} + S(T) + O\bigg(\frac{1}{T}\bigg),
\]
where the term $S(T)$ is
\[
S(T) := \frac{1}{\pi} \arg \zeta \bigg(\frac{1}{2}+it\bigg) \ll \log T
\]
as $T \to \infty$. Let $\kappa$ be the proportion of zeros on the critical line, i.e.
\[
\kappa := \liminf_{T \to \infty} \frac{N_0(T)}{N(T)}.
\]
In 1942, Selberg \cite{selberg} showed that $1 \ge \kappa > 0$. This means that a positive proportion of non-trivial zeros of the Riemann zeta-function lies on the critical line.\\

Let $Q(x)$ be a real polynomial satisfying $Q(0)=1$ and $Q'(x)=Q'(1-x)$. Set $d$ to be the degree of $Q$, so that $d = \deg (Q) \ge 1$. We then define the differential operator $V$ by
\[
V(s) := Q \bigg(-\frac{1}{L}\frac{d}{ds}\bigg) \zeta(s),
\]
where, for large $T$, we set
\[
L := \log T.
\]
Using the functional equation of $\zeta(s)$, Littlewood's lemma and the arithmetic and geometric mean inequalities, Conrey \cite{conrey89} (see also \cite{levinson}) showed that
\begin{align} \label{kappa}
\kappa \ge 1 - \frac{1}{R} \log \bigg(\frac{1}{T} \int_1^T |V\psi(\sigma_0+it)|^2 dt \bigg) + o(1).
\end{align}
Here $\sigma_0 = 1/2 - R/L$ where $R$ is a bounded positive real of our choice and $\psi$ is a mollifier.\\

A mollifier is a regular function designed to dampen the large values of $\zeta(s)$ so the product $V\psi$ is expected to be smaller than $V$. To mollify $\zeta$, one uses a Dirichlet polynomial
\[
\psi(s) := \sum_{n \le y} \frac{b(n,y)}{n^s}
\]
with suitable coefficients $b(n,y)$ and an acceptable length $y=T^{\theta}$, where $0 < \theta < 1$.\\

A wide range of coefficients $b(n,y)$ have been studied in the literature. Levinson \cite{levinson} first used
\[
b(n,y) = \mu(n) n^{\sigma_0-1/2} \frac{\log (y/n)}{\log y}
\]
with $\theta = \frac{1}{2}-\varepsilon$. Along with the choice $Q(x)=1-x$ Levinson was able to prove that $\kappa > \frac{1}{3}$ in 1974. \\

In \cite[p. 7]{conreyiwaniec}, a comparison between Selberg's and Levinson's methods is made. Essentially these two methods are `diametrically opposed'. Indeed, Selberg's method is based on counting sign changes of the suitably normalized and mollified Riemann zeta-function and this is a very safe, if not entirely effective, procedure. One cannot get a negative (worse than trivial) bound for the counting number. However, due, among other things, to the fact that the zeros are not evenly spaced, Selberg's method fails to produce significant values of $\kappa$. On the other hand, Levinson's method is a gamble as it could produce negative bounds for the counting number of critical zeros if the pertinent estimates are wasteful. If the mollification is `nearly perfect', then it opens the possibility for $100\%$, or at least substantially higher values of $\kappa$. Therefore, it behooves us to perfect the technique of Levinson's method as much as possible and present it in its greatest flexibility and generality. This is indeed one of the goals of this article.\\

The next refinement is due to Conrey \cite{conrey83a} who further generalized the mollifier to
\begin{align} \label{conreylevinson}
b_C(n,y) = \mu(n) n^{\sigma_0-1/2} P\bigg(\frac{\log (y/n)}{\log y}\bigg)
\end{align}
where $P(x)$ is a real polynomial such that $P(0)=0$ and $P(1)=1$. Combined with further refinements on the polynomial $Q$, such as taking $d= 5$ and keeping $\theta = \frac{1}{2}-\varepsilon$, Conrey showed that $\kappa > 0.36581$ along with other results on the proportion of zeros of derivatives of the Riemann zeta-function on $\real(s) = \frac{1}{2}$.\\

The next improvement would be arithmetical in nature. In \cite{bchb}, Balasubramanian, Conrey and Heath-Brown examined the error terms from the mean value integral in \eqref{kappa} and the resulting exponential sums. In particular they showed using Vaughan's identity \cite{vaughanidentity} on $1/\zeta(s)$ and Weil's bound for Kloosterman sums that one could push the size of the length $\theta$ past the $\frac{1}{2}$ barrier to $\frac{9}{17}$ when the coefficients are given by \eqref{conreylevinson}. The consequences of Hooley's conjecture $R^*$, see \cite{hooley}, are also discussed.\\

Further improvements of this result using \eqref{conreylevinson} are worked out by Conrey in \cite{conrey89}, where he uses results from Deshouillers and Iwaniec \cite{deshouillersIwaniec1, deshouillersIwaniec2} on exponential sums to unconditionally prove that $\theta=\tfrac{4}{7}-\varepsilon$ which results in $\kappa > 0.4088$.\\

Further choices of $b(n,y)$ have been proposed. Following the work of Luo and Yao \cite{luoyao}, Feng \cite{feng} proposed
\[
b_F(n,y_F) = \mu(n) n^{\sigma_0 - 1/2} \sum_{k=2}^K \sum_{p_1 \cdots p_k | n} \frac{\log p_1 \cdots \log p_k}{\log^k y_F} P_k \bigg(\frac{\log (y_F/n)}{\log y_F}\bigg),
\]
where $K=2,3,\cdots$ is an integer of our choice and $y = T^\theta$ for some $0<\theta<1$. The set $\{p_i\}_{i=1}^k$ is composed of distinct primes and $P_k$ are certain polynomials not unlike $P$ above. Working with the two-piece mollifier
\[
\psi(s) = \sum_{n \le y_C} \frac{b_C(n,y)}{n^s} + \sum_{n \le y_F} \frac{b_F(n,y)}{n^s}
\]
where $y_C = T^{\theta_C}$ and $y_F = T^{\theta_F}$ with $\theta_C = \frac{4}{7}-\varepsilon$ and  $\theta_F = \frac{3}{7}-\varepsilon$, Feng proved that $\kappa > 0.4107$. \\

The size of $\theta_F$ was initially taken to be $\frac{4}{7}$ and Feng later reduced it to $\frac{1}{2}$. The proportion of zeros associated to $\theta_F = \frac{1}{2}-\varepsilon$ is $\kappa > 0.4128$. However, in \cite{bui, krz01, pr01, rrz01}, a gap was found in Feng's argument which reduces the size to $\frac{3}{7}$, unless some work is done at the exponential sum level of the error terms. In \cite{pr01}, it is shown, by decomposing the error terms associated to $b_F(n,y)$ into Type I and Type II sums and handling the resulting incomplete Kloosterman sums, that one can take $\theta_F = \frac{6}{11} - \varepsilon$, thereby validating Feng's claim that $\theta_F = \frac{1}{2}-\varepsilon$ and $\kappa > 0.4128$.\\

The best bound for an arbitrary coefficient $a_n$ of a generic Dirichlet series $\sum_{n \le T^{\theta}} a_n n^{-s}$ is $\theta = \frac{17}{33} - \varepsilon$. This is due to Bettin, Chandee and Radziwi\l{}\l{} \cite{bcr}. Its key ingredient is an improvement of a result of Duke, Friedlander and Iwaniec \cite{dfi} on trilinear Kloosterman sums due to Bettin and Chandee \cite{bc}. We also remark that in \cite{rrz02}, Robles and Zaharescu along with Roy proved that the bilinear Kloosterman sums of \cite{dfi} lead to $\theta = \frac{48}{95} - \varepsilon$, but this result was obtained shortly after the publication of \cite{bcr}.\\

Somewhat inspired by \cite{luoyao} and certainly drawing from the autocorrelation of ratios, Bui, Conrey and Young \cite{bcy} introduced a second piece to Conrey's mollifier, namely they worked with
\begin{align} \label{bcy}
\psi(s) = \sum_{n \le y_C} \frac{b_C(n,y)}{n^s} + \chi(s + \tfrac{1}{s}-\sigma_0) \sum_{hk \le y_2} \frac{\mu_2(h)h^{\sigma_0-1/2}k^{1/2-\sigma_0}}{h^sk^{1-s}}P_2 \bigg(\frac{\log (y_2/hk)}{\log y_2}\bigg).
\end{align}
Here $\mu_2$ is given by the Dirichlet convolution $\mu_2(h) = (\mu \star \mu)(h)$ and $\chi(s)$ is such that $\zeta(s) = \chi(s)\zeta(1-s)$, i.e. $\chi(s) = 2^s \pi^{s-1}\sin(\frac{1}{2}\pi s)\Gamma(1-s)$. In this case $y_2 = T^{1/2-\varepsilon}$ and $P_2$ is a polynomial with similar properties to those of $P_1$. They obtained $\kappa > 0.4105$. \\

In \cite[p. 515]{feng} and \cite[p. 36]{bcy} the idea of crossing all mollifiers
\[
\psi(s) = \sum_{n \le y_C} \frac{b_C(n,y)}{n^s} + \sum_{n \le y_F} \frac{b_F(n,y)}{n^s} + \chi(s + \tfrac{1}{s}-\sigma_0) \sum_{hk \le y_2} \frac{\mu_2(h)h^{\sigma_0-1/2}k^{1/2-\sigma_0}}{h^sk^{1-s}}P_2 \bigg(\frac{\log (y_2/hk)}{\log y_2}\bigg)
\]
was remarked and it was hinted at that this would be a `technically difficult' thing to do. This was accomplished in \cite{rrz01}.\\

Lastly, a family of mollifiers that generalizes \eqref{bcy} was studied in \cite{krz02} and independently and almost simultaneously by Sono in \cite{sono}.\\

It was suspected from an argument of Farmer (`$\theta=\infty$ conjecture'), see \cite{farmer1} and \cite[p. 1]{bg}, that mollifiers might be optimal when their size is $1-\varepsilon$ in that they produce $100\%$ of zeros on the critical line. However, another intriguing recent result in this direction is due to Bettin and Gonek \cite{bg}. They prove with a very short and elegant argument involving Mellin transforms and Parseval's formula that if one takes $\theta = \infty$, then RH would follow (not just $100\%$). Of course, we are very far away from such lengths of mollifiers. Nevertheless, theoretically this approach opens the door to a direction towards RH via the moments.\\

In this paper we propose to mollify the whole perturbed Riemann zeta-function. In  other words, we mollify $V(s)$ for a general $d$. As pointed out in the literature, see e.g. \cite[$\mathsection$ 3]{conrey89} and \cite[p. 515]{feng}, the idea behind Selberg's method is to mollify $\zeta(s)$ directly. However, in Levinson's framework, what one needs to mollify is the whole perturbed function $V(s)$. This is not an easy task and one runs into serious combinatorial difficulties. Indeed, as remarked in \cite[Remark (c)]{feng}, `it is too complicated to optimize exactly the coefficients of the mollifier.' It is in fact too complicated to even \textsl{display} the terms of the mollified moment, let alone optimize them.\\

To accomplish this task, we examine the behavior of the inverse $1/V(s)$ as the degree $d$ increases. Namely, we will be studying the expression
\begin{align} \label{Mds}
\mathcal{M}(s,d) := \frac{1}{\zeta(s) + \lambda_1 \frac{\zeta'(s)}{\log T} + \lambda_2 \frac{\zeta''(s)}{\log^2 T} + \cdots + \lambda_d \frac{\zeta^{(d)}(s)}{\log^d T}}
\end{align}
as a function of $d$ and the complex variable $s$. We shall be able to compute the mean value integral appearing in \eqref{kappa} with any desired degree of accuracy in terms of $d$.\\ 

In other words, instead of examining
\begin{align*}
\frac{1}{\zeta(s)} \bigg(\zeta(s) + c_1 \frac{\zeta'(s)}{\log T} + c_2 \frac{\zeta''(s)}{\log^2 T} + \cdots + c_d \frac{\zeta^{(d)}(s)}{\log^d T}\bigg)
\end{align*}
we examine
\begin{align} \label{ratioofzetaderivatives}
\frac{1}{\zeta(s) + \lambda_1 \frac{\zeta'(s)}{\log T} + \lambda_2 \frac{\zeta''(s)}{\log^2 T} + \cdots + \lambda_d \frac{\zeta^{(d)}(s)}{\log^d T}}\bigg(\zeta(s) + c_1 \frac{\zeta'(s)}{\log T} + c_2 \frac{\zeta''(s)}{\log^2 T} + \cdots + c_d \frac{\zeta^{(d)}(s)}{\log^d T}\bigg),
\end{align}
and provide the clarity needed to extract the rich features that this eventual autocorrelation of ratios of products of $\zeta$ functions has to offer. More explicitly, as mentioned in the discussion after Conjecture \ref{CFKRSconjecture} above, the object of study is the $k=1$ case of the autocorrelation functions (sometimes called shifted moments, see \cite[p. 1]{cfkrs2}). In our case, the shifts have their origins in \eqref{kappa}. It will become clearer as we proceed that \eqref{ratioofzetaderivatives} will lead to integrals of a ratios of several products of shifted zeta functions (see \eqref{autocorrelationratio} and \eqref{generaltermsauto} below) such as $G(z_1, \cdots, z_{2k})$ in Conjecture \ref{CFKRSconjecture} and \cite[equation (5.12)]{cfz}.\\

Although computing the moment integrals of the zeta function twisted by a general Dirichlet series associated to the mollification is the main and most difficult target of our research (see Theorem \ref{theoremmaintermerror47}), we can give an immediate application. An interim optimization of the parameters at our disposal yields $\kappa > 0.417293$ and $\kappa^* \ge 0.407511$ where $\kappa^*$ denotes the proportion of simple zeros on the critical line. 

\begin{theorem} \label{512}
More than five-twelfths of the non-trivial zeros of the Riemann zeta-function are on the critical line.
\end{theorem}

One of the most satisfactory features of the result we present for the twisted second moment is the amount of adjustability that it exhibits. Therefore we expect computational enthusiasts to push the boundaries of the terms we present. Moreover, if or when the length of present or future mollifiers is increased, number theorists should be able to use the theoretical and numerical procedures we present in this article to refine $\kappa$ or other arithmetical quantities of interest (see e.g. \cite{cs} for applications of autocorrelation ratios in pair correlations, distributions, discrete moments, connections to random matrix theory, etc). We end our discussion with some problems for future work in $\mathsection$9.

\begin{remark}
Shortly before we presented this paper, Wu \cite{wu} uploaded his result on the twisted mean square and the critical zeros of Dirichlet $L$-functions. Our conclusions partially overlap and the methodologies are independent of each other. Our bounds for the critical zeros (of Riemann or Dirichlet) are in agreement.
\end{remark}

\section{Preliminary tools}
We shall devote this section to presenting the tools we will need throughout the paper. Let $\nu(n)$ denote the number of distinct primes of $n$. We use $n=p_1p_2 \cdots p_r$ to denote the prime factorization of a general square-free number. If $n$ is a square-free number, then $\nu(n)=r$ and $\mu(n)=(-1)^r$.

The generalized von Mangoldt function $\Lambda_k(n)$ is defined as (see \cite{ivicmangoldt, rr02} and \cite{cmrz} for applications to zeros)
\[
\Lambda_k(n) := (\mu \star \log^k)(n)
\]
for $k \in \N$. If $k=1$, then we have $\Lambda_1(n) = \Lambda(n)$, the usual von Mangoldt function. For $\real(s)>1$, its Dirichlet series is given by
\[
\frac{\zeta^{(k)}}{\zeta}(s) = (-1)^k \sum_{n=1}^\infty \frac{\Lambda_k(n)}{n^s},
\]
where $\zeta^{(k)}$ stands for the $k$the derivative of $\zeta(s)$ with respect to $s$. We also note the following identity
\[
\frac{d}{ds} \bigg(\frac{\zeta^{(k)}}{\zeta}(s)\bigg) = \frac{\zeta^{(k+1)}}{\zeta}(s) - \frac{\zeta'}{\zeta}(s)\frac{\zeta^{(k)}}{\zeta}(s).
\]
Arithmetically, this means that
\begin{align} \label{arithmeticLambdaK}
\Lambda_{k+1}(n) = \Lambda_k(n) \log(n) + (\Lambda \star \Lambda_k)(n).
\end{align}
Moreover, for $\real(s)>1$, we can write
\[
\frac{d^{m-1}}{ds^{m-1}} \frac{\zeta'}{\zeta}(s) = (-1)^m \sum_{n=1}^\infty \frac{\Lambda(n)\log^{m-1}n}{n^s} = (-1)^m \sum_{n=1}^\infty \frac{\Lambda_{\mathcal{L},m-1}(n)}{n^s}
\]
where
\begin{equation}
\Lambda_{\mathcal{L},k}(n) := \Lambda(n) \log^k n = 
\begin{cases}
\ell^k \log^{k+1} p, & \mbox{ if $n=p^\ell$  for some prime $p$ and positive integer $\ell$}, \\
0, & \mbox{ otherwise}. \nonumber
\end{cases}
\end{equation}
The advantage of working with $\Lambda_{\mathcal{L},k}$ instead of $\Lambda_k(n)$ is that $\Lambda_{\mathcal{L},k}$ will be zero when $n$ is not a power of a prime, whereas this is certainly not the case for the much more combinatorially complicated arithmetical function $\Lambda_k(n)$.

We denote by $\mathcal{P}(k)$ the representation of unordered partitions of $k$ into positive parts \cite[p. 14]{knafo}, i.e.
\[
\mathcal{P}(k) := \bigg\{ (\vartheta_1, \vartheta_2, \cdots, \vartheta_k) \quad \textnormal{such that} \quad \vartheta_1, \vartheta_2, \cdots, \vartheta_k \ge 0 \quad \textnormal{and} \quad \sum_{i=1}^k i\vartheta_i = k\bigg\},
\]
and $\mathcal{C}(k,n)$ stands for the ordered partitions of the integer $k$ into $n$ nonnegative parts, i.e.
\[
\mathcal{C}(k,n) := \bigg\{ (\lambda_1, \lambda_2, \cdots, \lambda_n) \quad \textnormal{such that} \quad \lambda_1, \lambda_2, \cdots, \lambda_n \ge 0 \quad \textnormal{and} \quad \sum_{i=1}^n \lambda_i = k \bigg\}.
\]
Moreover, we also define
\[
 \mathcal{C}^*(n,m) := \bigg\{ (k_1, k_2, \cdots, k_m) \quad \textnormal{such that} \quad k_1, k_2, \cdots, k_m > 0 \quad \textnormal{and} \quad \sum_{i=1}^m k_i = n \bigg\}.
\]
The multinomial coefficients are given by
\[
\binom{n}{k_1, k_2, \cdots, k_m} = \frac{n!}{k_1! k_2! \cdots k_m!}.
\]
%
%The indicator function is $\mathbf{1}_{\alpha \in A} = 1$ if $\alpha \in A$ and $0$ otherwise.
%\begin{equation}
%\mathbf{1}_{\alpha \in A} := 
%\begin{cases}
%1, & \mbox{ if $\alpha \in A$},\\
%0, & \mbox{ otherwise}, \nonumber
%\end{cases}
%\end{equation}
%has its usual meaning.

An identity involving $\Lambda_k$ that can occasionally become useful is 
\begin{align} \label{Lambdaidentity}
\Lambda_k(n) = \sum_{(i_1, \cdots, i_r) \in \mathcal{C}^*(k,r)} \binom{k}{i_1, \cdots, i_r} \log^{i_1} p_1 \cdots \log^{i_r} p_r
\end{align}
for square-free $n$ as described earlier.\\

The polynomials $B_{n,k}(x_1, x_2, \cdots, x_{n-k+1})$ denote the partial or incomplete exponential Bell polynomials, whereas $B_n(x_1, x_2, \cdots, x_n)$ will denote the $n$th complete exponential Bell polynomials. A good introduction to these combinatorial objects can be found for instance in \cite[$\mathsection$ 3.3]{comtet}. For the sake of completeness, and given the role they will play shortly, we shall define and illustrate their main properties.

The Bell polynomials are defined by
\[
B_{n,k}(x_1, x_2, \cdots, x_{n-k+1}) := \sum_{\divideontimes(k,n)} \frac{n!}{j_1! j_2! \cdots j_{n-k+1}!} \bigg(\frac{x_1}{1!}\bigg)^{j_1}\bigg(\frac{x_2}{2!}\bigg)^{j_2} \cdots \bigg(\frac{x_{n-k+1}}{(n-k+1)!}\bigg)^{j_{n-k+1}}.
\]
where $\divideontimes(k,n)$ indicates that the sum is over $\{j_1, \cdots, j_{n-k+1}\} \in \mathcal{C}(k,n) \cap \mathcal{P}(k)$. The complete exponential Bell polynomials are given by the sum
\[
B_n(x_1, x_2, \cdots, x_n) = \sum_{k=1}^n B_{n,k}(x_1, x_2, \cdots, x_{n-k+1}).
\]
The partial Bell polynomials can be computed efficiently by a recursion  relation
\[
B_{n,k} = \sum_{i=1}^{n-k+1} \binom{n-1}{i-1}x_i B_{n-i,k-i},
\]
where $B_{0,0}=1$, $B_{n,0}=0$ for all $n \ge 1$ and $B_{0,k}=0$ for $k \ge 1$. When it comes to the complete Bell polynomials we have the recursion 
\[
B_{n+1}(x_1, x_2, \cdots, x_{n+1}) = \sum_{i=0}^n \binom{n}{i} B_{n-i}(x_1, x_2, \cdots, x_{n-i})x_{i+1},
\]
with $B_0=1$. Lastly, the generating function is given by
\[
\exp \bigg(\sum_{i=0}^\infty \frac{x_i}{i!} t^i \bigg) = \sum_{n=0}^\infty \frac{1}{n!}B_n(x_1, x_2, \cdots, x_n)t^n.
\]

The exponential Bell polynomial encodes the information related to the ways a set can be partitioned, and given a Bell polynomial $B_n$ we can separate the partial Bell polynomial $B_{n,k}$ by collecting all those monomials with degree $k$. We shall now illustrate some examples.\\

Let us for instance take $B_{3,k}(x_1, x_2, x_3)$. We immediately find
\begin{align*}
B_{3,1}(x_1, x_2, x_3) &= x_3, \quad B_{3,2}(x_1, x_2, x_3) = 3x_1x_2, \quad \textnormal{and} \quad B_{3,3}(x_1, x_2, x_3) = x_1^3,
\end{align*}
which we can represent pictorially as in \cite{dickau, weinsstein}:
\begin{figure}[H]
	\centering
	\includegraphics[scale=0.9]{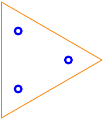}
	\includegraphics[scale=0.9]{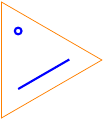}
	\includegraphics[scale=0.9]{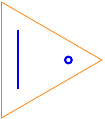}
	\includegraphics[scale=0.9]{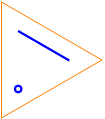}
	\includegraphics[scale=0.9]{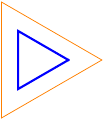}
%	\hspace{5mm}
	\caption{$B_{3,3}(x_1,x_2,x_3)$ (extreme left), $B_{3,2}(x_1,x_2,x_3)$ (3 middle diagrams) and $B_{3,1}(x_1,x_2,x_3)$ (extreme right).}
\end{figure}
In this case, $x_1$ indicates the presence of a block with a single element, $x_2$ the presence of a block with two elements and $x_3$ a block with three elements. Since the coefficient of $B_{3,2}$ is 3, we obtain three different ways of partitioning a block of 3 elements into 2 blocks, one block of 1 element and one block of 2 elements. We also note that $B_3(1,1,1)=B_3=5$, which is the Bell number associated to $3$. This represents the total number of diagrams.\\

Similarly, if we now consider $B_{4,k}(x_1,x_2,x_3,x_4)$ for $k=1,2,3,4$ then we obtain
\begin{align*}
B_{4,1}(x_1, x_2, x_3, x_4) = x_4, &\quad B_{4,2}(x_1, x_2, x_3, x_4) = 3x_2^2 + 4x_1x_3, \\
B_{4,3}(x_1, x_2, x_3, x_4) = 6x_1^2x_2, &\quad B_{4,4}(x_1, x_2, x_3, x_4) = x_1^4.
\end{align*}
\begin{figure}[H]
	\centering
	\includegraphics[scale=0.9]{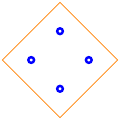}
	\includegraphics[scale=0.9]{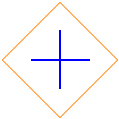}
	\includegraphics[scale=0.9]{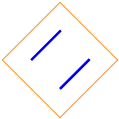}
    \includegraphics[scale=0.9]{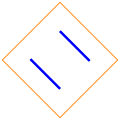}
	\includegraphics[scale=0.9]{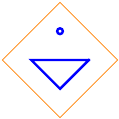}
	\includegraphics[scale=0.9]{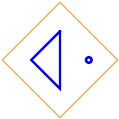}
	\includegraphics[scale=0.9]{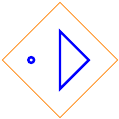}
	\includegraphics[scale=0.9]{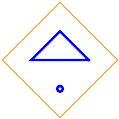}
	\includegraphics[scale=0.9]{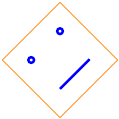}
    \includegraphics[scale=0.9]{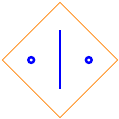}
    \includegraphics[scale=0.9]{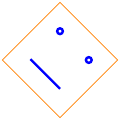}
   	\includegraphics[scale=0.9]{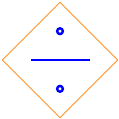}
   	\includegraphics[scale=0.9]{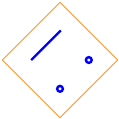}
    \includegraphics[scale=0.9]{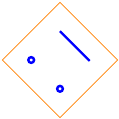}
    \includegraphics[scale=0.9]{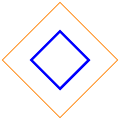}
%	\hspace{5mm}
	\caption{$B_{4,k}(x_1,x_2,x_3,x_4)$ for $k=1,2,3,4$.}
\end{figure}
We remark again that the total number of diagrams, or Bell number, is $B_4=15$. With these tools in mind, let us now proceed.

\section{Constructing a mollifier}
In this section we will present the ideas behind the construction of a mollifier and we will capitalize on how they have been constructed up until now before we explain the approach we have taken.

\subsection{The zeroth order case $d=0$.} In this case, going back to \eqref{Mds} with $d=0$, one simply has
\[
\mathcal{M}(s,0) = \frac{1}{\zeta(s)} = \sum_{n=1}^\infty \frac{\mu(n)}{n^s}
\]
from which we get a mollifier of the form
\begin{align} \label{psi0}
\psi_{d=0}(s) = \sum_{n \le y_0} \frac{\mu(n)n^{\sigma_0-1/2}}{n^s} P_0\bigg(\frac{\log (y_0/n)}{\log y_0}\bigg).
\end{align}
When we impose that the cutoffs on the polynomial $P_0$, i.e. $P_0$ be such that $P_0(0)=0$ and $P_0(1)=1$, we then see that this is in agreement with \eqref{conreylevinson}.

\subsection{The linear case $d=1$.} Now we have to deal with the first derivative. This is the case contemplated by Feng \cite{feng}. Going back to \eqref{Mds} we formally get
\begin{align} \label{truncationd1}
\mathcal{M}(s,1) &= \frac{1}{\zeta(s)+\frac{\zeta'(s)}{\log T}} = \frac{1}{\zeta(s)} \bigg(1+\frac{1}{\log T}\frac{\zeta'}{\zeta}(s)\bigg)^{-1} \nonumber \\
& = \sum_{k=0}^\infty (-1)^k \frac{1}{\log^k T}\frac{1}{\zeta(s)} \bigg(\frac{\zeta'}{\zeta}(s)\bigg)^k  = \sum_{k=0}^\infty \frac{1}{\log^k T}\frac{1}{\zeta(s)} \bigg(-\frac{\zeta'}{\zeta}(s)\bigg)^k \nonumber \\
& = \sum_{k=0}^\infty \frac{1}{\log^k T} \sum_{n=1}^\infty \frac{(\mu \star \Lambda^{\star k})(n)}{n^s},
\end{align}
by the use of the binomial theorem for fractional powers
\[
(1+x)^{-1} = \sum_{k=0}^\infty \binom{-1}{k}x^k = \sum_{k=0}^\infty (-1)^k x^k.
\]
Here $\Lambda^{\star k}$ stands for convolving $\Lambda$ with itself exactly $k$ times. Suppose that $n$ is square-free, then
\begin{align} \label{d1convolution}
(\mu \star \Lambda^{\star k})(n) &= \sum_{d_0 d_1 \cdots d_k = n} \mu(d_0) \Lambda(d_1) \cdots \Lambda(d_k) = \sum_{\operatorname{cyclic}} \Lambda(p_1) \cdots \Lambda(p_k) \mu(p_{k+1} \cdots p_r) \nonumber \\
&= (-1)^{r+k} \sum_{p_1 \cdots p_k | n} \log p_1 \cdots \log p_k = (-1)^k \mu(n) \sum_{p_1 \cdots p_k | n} \log p_1 \cdots \log p_k.
\end{align}
According to Feng's conjecture (\cite[p. 516]{feng}), if $n$ had had a square divisor, then the coefficients coming from \eqref{truncationd1} would contribute a lower order term to the mean value integrals $\int |V\psi(\sigma_0+it)|^2 dt$. This means that we could simply ignore the $n$'s for which $\mu^2(n)=0$. However, since Feng's claim is not substantiated, we must operate by supposing that $n$ is square-free. Otherwise, the computation of the convolution \eqref{d1convolution} becomes very difficult. Therefore, by keeping $n$ square-free, we get
\begin{align*}
\mathcal{M}(s,1) &= \sum_{k=0}^\infty \frac{1}{\log^k T} \sum_{n=1}^\infty \frac{1}{n^s} (-1)^k \mu(n) \sum_{p_1 \cdots p_k} \log p_1 \cdots \log p_k = \sum_{n=1}^\infty \frac{\mu(n)}{n^s} \sum_{k=0}^\infty (-1)^k \sum_{p_1 \cdots p_k | n} \frac{\log p_1 \cdots \log p_k}{\log^k T}.
\end{align*}
Here the $p_i$ denote distinct primes. This suggests a mollifier of the form
\begin{align} \label{fengsmollifierlogp}
\psi_{d=1}(s) = \sum_{n \le y_1} \frac{\mu(n)n^{\sigma_0-1/2}}{n^s} \sum_{k=2}^K \sum_{p_1 \cdots p_k | n} \frac{\log p_1 \cdots \log p_k}{\log^k y_1} P_{1,k} \bigg(\frac{\log(y_1/n)}{\log y_1}\bigg).
\end{align}
The conditions on $P_{1,k}$ are that $P_{1,k}(0)=0$ for all $k$. We note the following remarks.
\begin{enumerate}
\item Feng has set the convention of starting at $K=2$.
\item Here $K=2,3,\cdots$ is an integer of our choice coming from the truncation of the infinite sum over $k$. The higher $K$ is, the more precise the mollification. However, this is achieved at the cost of adding extra terms that require taxing computational resources.
\item The sign alternator $(-1)^k$ has been absorbed into the polyonimals $P_{1,k}$, i.e.
\[
(-1)^k P_{1,k}(x) = (-1)^k \sum_{i=0}^{\deg P} a_{1,k,i} x^i =  \sum_{i=0}^{\deg \tilde{P}} \tilde{a}_{1,k,i} x^i = \tilde{P}_{1,k}(x)
\]
where
\[
\begin{cases} 
\tilde{a}_{1,k,i} = (-1)^k a_{1,k,i}, \\
\deg P = \deg \tilde{P}.
\end{cases}
\]
\item Alternatively, we could have written
\[
\psi_{d=1}(s) = \sum_{n \le y_1} \frac{\mu^2(n)n^{\sigma_0-1/2}}{n^s} \sum_{k=2}^K \frac{1}{\log^k y_1} (\mu \star \Lambda^{\star k})(n) P_{1,k} \bigg(\frac{\log (y_1/n)}{\log y_1}\bigg),
\]
since $\mu^2(n)$ will discriminate square-free numbers.
\end{enumerate}

\subsection{The quadratic case $d=2$.} Before proceeding with the general case, it will be instructive to see how adding the second derivative increases substantially the complexity of the combinatorics associated to this problem. The degree $d$ is small enough that a trick that changes $\zeta''/\zeta$ into derivatives of $\zeta'/\zeta$ is sufficient to obtain a useful mollifier.
The effect of working with $d=2$ is that the expression $\mathcal{M}$ in \eqref{Mds} becomes
\begin{align*}
\mathcal{M}(s,2) &= \frac{1}{\zeta(s) + \frac{\zeta'(s)}{\log T} + \frac{\zeta''(s)}{\log T}} = \frac{1}{\zeta(s)}\bigg(1+\frac{1}{\log T}\frac{\zeta'}{\zeta}(s) + \frac{1}{\log^2 T}\frac{\zeta''}{\zeta}(s) \bigg)^{-1} \\
&= \frac{1}{\zeta(s)} \sum_{k=0}^\infty (-1)^k \bigg(\frac{1}{\log T} \frac{\zeta'}{\zeta}(s) + \frac{1}{\log^2 T} \frac{\zeta''}{\zeta}(s) \bigg)^k \\
&= \frac{1}{\zeta(s)} \sum_{k=0}^\infty (-1)^k \sum_{j=0}^k \binom{k}{j} \frac{1}{(\log T)^{k+j}} \bigg(\frac{\zeta''}{\zeta}(s)\bigg)^j \bigg(\frac{\zeta'}{\zeta}(s)\bigg)^{k-j}.
\end{align*}
Now we use $k=1$ in \eqref{arithmeticLambdaK} to replace $\zeta''/\zeta$ by an expression involving only $\zeta'/\zeta$ and hence
\begin{align*}
\mathcal{M}(s,2) &= \frac{1}{\zeta(s)} \sum_{k=0}^\infty (-1)^k \sum_{j=0}^k \binom{k}{j} \frac{1}{(\log T)^{k+j}} \bigg( \frac{d}{ds} \frac{\zeta'}{\zeta}(s) + \bigg(\frac{\zeta'}{\zeta}(s)\bigg)^2 \bigg)^j \bigg(\frac{\zeta'}{\zeta}(s)\bigg)^{k-j} \\
&= \sum_{k=0}^\infty  \sum_{j=0}^k \binom{k}{j} \frac{1}{(\log T)^{k+j}} \sum_{i=0}^j (-1)^{k+k-j+2i} \binom{j}{i} \sum_{n=1}^\infty \frac{(\mu \star \Lambda^{\star k-j+2i} \star \Lambda_{\mathcal{L},1}^{\star j-i})(n)}{n^s}.
\end{align*}
Let us then look at the Dirichlet convolution a little bit more closely. For $n$ square-free, in a general power setting, we have
\begin{align*}
(\mu \star \Lambda^{\star a} \star \Lambda_{\mathcal{L},1}^{\star b})(n) &= \sum_{d_0d_1 \cdots d_a d_{a+1} \cdots d_{a+b}=n} \mu(d_0) \Lambda(d_1) \cdots \Lambda(d_a) \Lambda_{\mathcal{L},1}(d_{a+1}) \cdots \Lambda_{\mathcal{L},1}(d_{a+b}) \\
&= \sum_{\operatorname{cyclic}} \mu(d_0) \Lambda(p_1) \cdots \Lambda(p_a) \Lambda_{\mathcal{L},1}(p_{a+1}) \cdots \Lambda_{\mathcal{L},1}(p_{a+b}) \\
&= \mu(p_{a+b+1} \cdots p_r) \sum_{p_1 \cdots p_{a+b}|n} \log(p_1) \cdots \log(p_a) \log^2(p_{a+1}) \cdots \log^2(p_{a+b}) \\
&= (-1)^{a+b}\mu(n) \sum_{p_1 \cdots p_{a+b}|n} \log(p_1) \cdots \log(p_a) \log^2(p_{a+1}) \cdots \log^2(p_{a+b}).
\end{align*}
Using $a=k-j+2i$ as well as $b=j-i$ and inserting this into the above expression for $\mathcal{M}(s,2)$ while keeping $n$ square-free yields
\begin{align*}
\mathcal{M}(s,2) &= \sum_{k=0}^\infty \sum_{j=0}^k (-1)^j \binom{k}{j} \frac{1}{(\log T)^{k+j}} \sum_{i=0}^j \binom{j}{i} \sum_{n=1}^\infty (-1)^{k+i} \frac{\mu(n)}{n^s} \\
& \quad \times \sum_{p_1 \cdots p_{k+i}|n} \log p_1 \cdots p_{k-j+2i} \log^2 p_{k-j+2i-1} \cdots \log^2 p_{k+i} \\
&= \sum_{n=1}^\infty \frac{\mu(n)}{n^s} \sum_{k=0}^\infty \sum_{j=0}^k \sum_{i=0}^j (-1)^{i+j+k} \binom{k}{j}\binom{j}{i} \\
& \quad \times \sum_{p_1 \cdots p_{k+i}|n} \frac{\log p_1 \cdots p_{k-j+2i} \log^2 p_{k-j+2i-1} \cdots \log^2 p_{k+i}}{(\log T)^{k+j}} .
\end{align*}
Hence, the mollifier should be of the form
\begin{align*}
\psi_{d=2}(s) &= \sum_{n \le y_2} \frac{\mu(n)n^{\sigma_0-1/2}}{n^s} \sum_{k=0}^K \sum_{j=0}^k \sum_{i=0}^j (-1)^{i+j+k} \binom{k}{j}\binom{j}{i} \\
&\quad \times \sum_{p_1 \cdots p_{k+i}|n} \frac{\log p_1 \cdots p_{k-j+2i} \log^2 p_{k-j+2i-1} \cdots \log^2 p_{k+i}}{(\log T)^{k+j}} P_{2,k,j,i}\bigg(\frac{\log (y_2/n)}{\log y_2}\bigg).
\end{align*}
Again $K=2,3,\cdots$ is a positive integer of our choice. We see that there are as many polynomials are there are primes (in this case $k+i$ polynomials). We could also have written
\[
\psi_{d=2}(s) = \sum_{n \le y_2} \frac{\mu^2(n)n^{\sigma_0-1/2}}{n^s} \sum_{k=1}^K \sum_{j=0}^k \sum_{i=0}^j (-1)^j \binom{k}{j}\binom{j}{i} \frac{(\mu \star \Lambda^{\star k-j+2i} \star \Lambda_{\mathcal{L},1}^{\star j-i})(n)}{(\log y_2)^{k+j}} P_{2,k+i} \bigg(\frac{\log (y_2/n)}{\log y_2} \bigg),
\]
since the term $\mu^2(n)$ will discriminate square-free numbers. 

\subsection{The general $d \ge 0$ case}
We now relax the condition on $n$ and forgo the computation of the Dirichlet convolution. The advantage of operating this way will become clearer in the proof of our main result, see $\mathsection$5.1 and $\mathsection$6. The general $d \ge 0$ mollifier we want to use is given by
\begin{align*}
\mathcal{M}(d,s) &= \frac{1}{\zeta(s) + \tfrac{\zeta'(s)}{\log T} + \tfrac{\zeta''(s)}{\log^2 T} + \cdots + \tfrac{\zeta^{(d)}(s)}{\log^d T}} \\
& = \frac{1}{\zeta(s)} \bigg(1 + \frac{\zeta'}{\zeta}(s) + \frac{1}{\log^2 T} \frac{\zeta''}{\zeta}(s) + \cdots + \frac{1}{\log^d T} \frac{\zeta^{(d)}}{\zeta}(s)\bigg)^{-1} \\
& = \frac{1}{\zeta(s)} \sum_{k=0}^\infty (-1)^k \bigg(\frac{\zeta'}{\zeta}(s) + \frac{1}{\log^2 T} \frac{\zeta''}{\zeta}(s) + \cdots + \frac{1}{\log^d T} \frac{\zeta^{(d)}}{\zeta}(s) \bigg)^k \\
& = \frac{1}{\zeta(s)} \sum_{k=0}^\infty (-1)^k \sum_{k_1 + k_2 + \cdots + k_d = k} \binom{k}{k_1, k_2, \cdots, k_d} \prod_{m=1}^d \bigg(\frac{1}{\log^m T} \frac{\zeta^{(m)}}{\zeta}(s)\bigg)^{k_m} \\
& = \sum_{k=0}^\infty (-1)^k \sum_{k_1 + k_2 + \cdots + k_d = k} \binom{k}{k_1, k_2, \cdots, k_d} \frac{1}{(\log T)^{\sum_{m=1}^d mk_m}} \frac{1}{\zeta(s)} \prod_{m=1}^d \bigg(\frac{\zeta^{(m)}}{\zeta}(s)\bigg)^{k_m}.
\end{align*}
We now use the convolution
\begin{align} \label{generalconvolution}
\frac{1}{\zeta(s)} \prod_{m=1}^d \bigg(\frac{\zeta^{(m)}}{\zeta}(s)\bigg)^{k_m} = (-1)^{1 \times k_1 + 2 \times k_2 + \cdots + d \times k_d}\sum_{n=1}^\infty \frac{(\mu \star \Lambda^{\star k_1} \star \Lambda_2^{\star k_2} \star \cdots \star \Lambda_d^{\star k_d}) (n)}{n^s}
\end{align}
so that we end up with
\begin{align*}
\mathcal{M}(d,s) &= \sum_{k=0}^\infty (-1)^k \sum_{k_1 + k_2 + \cdots + k_d = k} \binom{k}{k_1, k_2, \cdots, k_d} \frac{(-1)^{1 \times k_1 + 2 \times k_2 + \cdots + d \times k_d}}{(\log T)^{\sum_{m=1}^d mk_m}} \\
& \quad \times \sum_{n=1}^\infty \frac{(\mu \star \Lambda^{\star k_1} \star \Lambda_2^{\star k_2} \star \cdots \star \Lambda_d^{\star k_d}) (n)}{n^s}.
\end{align*}
This suggests a mollifier of the form
\begin{align} \label{generalmollifierconvolution}
\psi_d (s) &= \sum_{\ell=0}^K (-1)^\ell \sum_{\ell_1 + \ell_2 + \cdots + \ell_d = \ell} (-1)^{1 \times \ell_1 + 2 \times \ell_2 + \cdots + d \times \ell_d} \binom{\ell}{\ell_1, \ell_2, \cdots, \ell_d} \nonumber \\
& \quad \times \sum_{n \le y_d} \frac{n^{\sigma_0-1/2}}{n^s} \frac{(\mu \star \Lambda^{\star \ell_1} \star \Lambda_2^{\star \ell_2} \star \cdots \star \Lambda_d^{\star \ell_d}) (n)}{(\log y_d)^{\sum_{r=1}^d r \ell_r}} P_{d,\ell} \bigg( \frac{\log(y_d/n)}{\log y_d} \bigg).
\end{align}
Here the polynomials $P$ are such that if $d=0$, then $P_{0,\ell} \equiv P_0$ with $P_0(0)=0$ as well as $P_0(1)=1$; and if $d > 0$, then $P_{d,\ell}(0)=0$ for all $\ell \ge 0$. \\

As a check we observe that $\mu \star \Lambda^{\star 0} = \mu$ and also when $d=0$ the $K$-truncation disappears (i.e. $K=\ell=0$) so we are left with the mollifier
\begin{align}
\psi_0 (s) = \sum_{n \le y_0} \frac{\mu(n) n^{\sigma_0-1/2}}{n^s} P_{0} \bigg( \frac{\log(y_0/n)}{\log y_0} \bigg),
\end{align}
which is the Conrey-Levinson mollifier with $y_0 = T^{\theta_0}$ where $\theta_0 = \frac{4}{7}-\varepsilon$. We shall take $y_d = N = T^{\theta_d}$ with $\theta_d = \frac{4}{7}-\varepsilon$ for all $d \ge 0$. As we shall explain in $\mathsection 6$, there is no need to be concerned with $\mu^2(n)$.

\subsection{Combinatorial interpretation of the mollifier}
Before proceeding with the mean value integral it is worth pausing to see from a different angle the structure of these mollifiers. Another way to interpret the combinatorial meaning behind this mollification is to apply Fa\`{a} di Bruno's formula \cite{comtet} 
\[
\frac{d^n}{dx^n} f(g(x)) = \sum_{k=1}^n f^{(k)}(g(x)) B_{n,k}(g'(x), g''(x), \cdots, g^{(n-k+1)}(x)),
\]
to $\exp(\log \zeta(s))$ so that we can write the following representation
\begin{align}
  \frac{{{\zeta ^{(m)}}}}{\zeta }(s) &= \sum_{k = 1}^m  {B_{m,k}}\left( {\frac{{\zeta '}}{\zeta }(s),\frac{d}{{ds}}\frac{{\zeta '}}{\zeta }(s), \cdots ,\frac{{{d^{m - k}}}}{{d{s^{m - k}}}}\frac{{\zeta '}}{\zeta }(s)} \right) \nonumber \\
%   &= \sum_{k = 1}^m  {B_{m,k}}\left( {\frac{{\zeta '}}{\zeta }(s),\frac{d}{{ds}}\frac{{\zeta '}}{\zeta }(s), \cdots ,\frac{{{d^{m - 1}}}}{{d{s^{m - 1}}}}\frac{{\zeta '}}{\zeta }(s)} \right) \nonumber \\
%   &= \sum_{k = 1}^m  {B_{m,k}}\left( {\frac{{\zeta '}}{\zeta }(s),\frac{d}{{ds}}\frac{{\zeta '}}{\zeta }(s), \cdots ,\frac{{{d^m}}}{{d{s^m}}}\frac{{\zeta '}}{\zeta }(s)} \right) \nonumber \\
   &= {B_m}\left( {\frac{{\zeta '}}{\zeta }(s),\frac{d}{{ds}}\frac{{\zeta '}}{\zeta }(s), \cdots ,\frac{d^{m-1}}{ds^{m-1}}\frac{\zeta '}{\zeta }(s)} \right) \nonumber  
\end{align}
Therefore, picking up from $\mathcal{M}(d,s)$ we arrive at
\begin{align}
  \mathcal{M}(d,s) &= \frac{1}{{\zeta (s)}}\sum_{k = 0}^\infty   {( - 1)^k}\sum_{{k_1} + {k_2} +  \cdots  + {k_d} = k}  \binom{k}{k_1,k_2,\cdots,k_d}\prod_{m = 1}^d  \frac{1}{(\log T)^{m k_m}}\bigg( {\frac{{{\zeta ^{(m)}}}}{\zeta }(s)} \bigg)^{{k_m}} \nonumber \\
   &= \frac{1}{{\zeta (s)}}\sum_{k = 0}^\infty   {( - 1)^k}\sum_{{k_1} + {k_2} +  \cdots  + {k_d} = k}  \binom{k}{k_1,k_2,\cdots,k_d} \nonumber \\
	 & \quad \times \prod_{m = 1}^d  \frac{1}{(\log T)^{m k_m}}{\left( {{B_m}\left( {\frac{{\zeta '}}{\zeta }(s),\frac{d}{{ds}}\frac{{\zeta '}}{\zeta }(s), \cdots ,\frac{{{d^{m-1}}}}{{d{s^{m-1}}}}\frac{{\zeta '}}{\zeta }(s)} \right)} \right)^{{k_m}}}. \nonumber 
\end{align}
Let us set 
\[
x_i = \frac{d^{i}}{ds^{i}} \frac{\zeta'}{\zeta}(s) \quad \textnormal{for} \quad i=0,1,2,\cdots,m.
\] 
We place ourselves in a general setting of  two Bell polynomials and three powers
\begin{align} \label{bellp1}
\sum_{m_1+m_2=3} \bigg(\sum_{k_1=0}^3 B_{1,k_1}(x_1)\bigg)^{m_1} \bigg(\sum_{k_2=0}^3 B_{2,k_2}(x_1, x_2)\bigg)^{m_2} &= x_1^3 + x_1^4 + x_1^5 + x_1^6 \nonumber \\
& \quad + x_1^2x_2 + 2x_1^3x_2 + 3x_1^4x_2 + x_1x_2^2 + 3x_1^2x_2^2 +x_2^3.
\end{align}
This allows us to now plot the exact representations of the Bell diagrams. We first note that
\[
\sum_{m_1+m_2=3} \bigg(\sum_{k_1=0}^3 B_{1,k_1}(1)\bigg)^{m_1} \bigg(\sum_{k_2=0}^3 B_{2,k_2}(1, 1)\bigg)^{m_2} = 15,
\]
which means that we will have 15 diagrams.
\begin{figure}[H]
	\centering
	\includegraphics[scale=0.9]{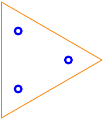}
	\includegraphics[scale=0.9]{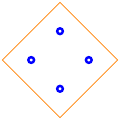}
	\includegraphics[scale=0.9]{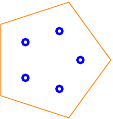}
	\includegraphics[scale=0.9]{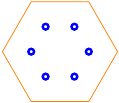}
	\includegraphics[scale=0.9]{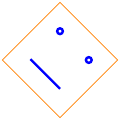}
	\includegraphics[scale=0.9]{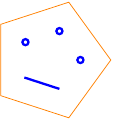}
	\includegraphics[scale=0.9]{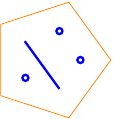}
	\includegraphics[scale=0.9]{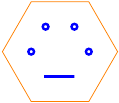}
	\includegraphics[scale=0.9]{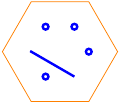}
	\includegraphics[scale=0.9]{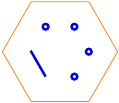}
	\includegraphics[scale=0.9]{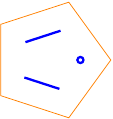}
	\includegraphics[scale=0.9]{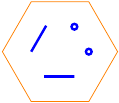}
	\includegraphics[scale=0.9]{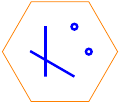}
	\includegraphics[scale=0.9]{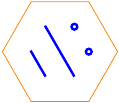}
	\includegraphics[scale=0.9]{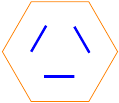}					
	\caption{Pictorial representation of \eqref{bellp1}.}
\end{figure}
As additional examples we will now increase the precision of the truncation by keeping two Bell polynomials but taking more powers. That means 
\begin{align} \label{bellp2}
\sum_{m_1+m_2=4} \bigg(\sum_{k_1=0}^3 B_{1,k_1}(x_1)\bigg)^{m_1} \bigg(\sum_{k_2=0}^3 B_{2,k_2}(x_1, x_2)\bigg)^{m_2} &= x_1^4 + x_1^5 + x_1^6 + x_1^7 + x_1^8 + x_1^3 x_2 + 
 2 x_1^4 x_2 + 3x_1^5 x_2   \nonumber \\
&\quad + 4 x_1^6 x_2 + x_1^2 x_2^2 + 
 3 x_1^3 x_2^2 + 6 x_1^4 x_2^2 + x_1 x_2^3  \nonumber \\
 &\quad + 4x_1x_2^3 + x_2^4.
\end{align}
This produces
\begin{align*}
\sum_{m_1+m_2=4} \bigg(\sum_{k_1=0}^3 B_{1,k_1}(1)\bigg)^{m_1} \bigg(\sum_{k_2=0}^3 B_{2,k_2}(1, 1)\bigg)^{m_2} = 31
\end{align*}
diagrams:
\begin{figure}[H]
	\centering
	\includegraphics[scale=0.9]{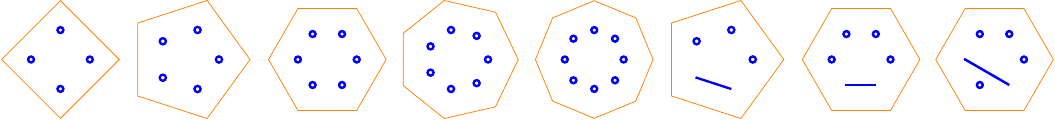}
	\includegraphics[scale=0.9]{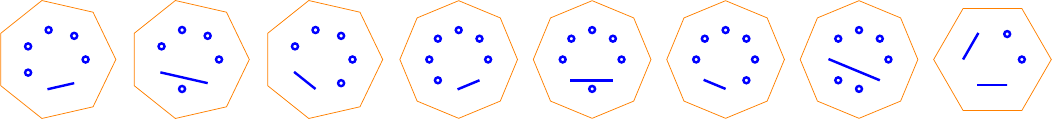}
	\includegraphics[scale=0.9]{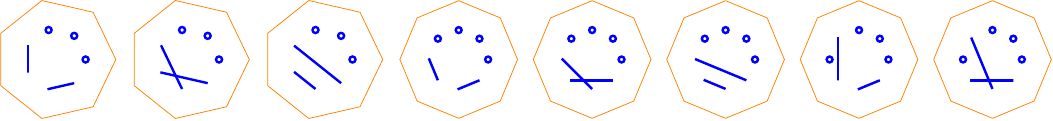}
	\includegraphics[scale=0.9]{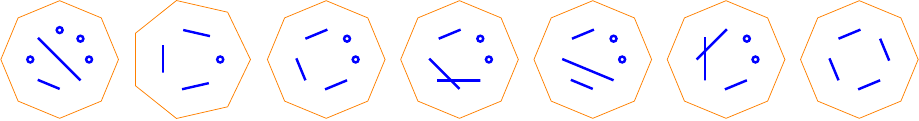}		
	\caption{Pictorial representation of \eqref{bellp2}.}
\end{figure}
Furthermore, if we instead increase the number of polynomials and keep three powers, then 
\begin{align}  \label{bellp3}
\sum_{m_1+m_2+m_3=3}&\bigg(\sum_{k_1=0}^3 B_{1,k_1}(x_1)\bigg)^{m_1} \bigg(\sum_{k_2=0}^3 B_{2,k_2}(x_1, x_2)\bigg)^{m_2}  \bigg(\sum_{k_3=0}^3 B_{3,k_3}(x_1, x_2, x_3)\bigg)^{m_3} \nonumber \\
&=
x_1^3 + x_1^4 + 2 x_1^5 + 2 x_1^6 + 2 x_1^7 + x_1^8 + x_1^9 + x_1^2 x_2 + 5 x_1^3 x_2 + 7 x_1^4 x_2 +  11 x_1^5 x_2 + 7 x_1^6 x_2 \nonumber \\
&\quad +  
 9 x_1^7 x_2 + x_1 x_2^2 + 6 x_1^2 x_2^2 + 16 x_1^3 x_2^2 + 15 x_1^4 x_2^2 +  
 27 x_1^5 x_2^2 + x_2^3 + 3 x_1 x_2^3 + 9 x_1^2 x_2^3 \nonumber \\
&\quad + 
 27 x_1^3 x_2^3 + x_1^2 x_3 +  x_1^3 x_3 + 3 x_1^4 x_3 + 2 x_1^5 x_3 + 3 x_1^6 x_3 + x_1 x_2 x_3 \nonumber \\
&\quad +  
 8 x_1^2 x_2 x_3 + 8 x_1^3 x_2 x_3 + 18 x_1^4 x_2 x_3 + x_2^2 x_3 + 6 x_1 x_2^2 x_3 + 27 x_1^2 x_2^2 x_3 \nonumber \\
&\quad +  
 x_1 x_3^2 + x_1^2 x_3^2 + 3 x_1^3 x_3^2 + x_2 x_3^2 +  
 9 x_1 x_2 x_3^2 + x_3^3.
\end{align}
Setting $x_1=x_2=x_3=1$ in \eqref{bellp3} yields 250 diagrams
%This set-up yields
%\[
%\sum_{m_1+m_2+m_3=3} \bigg(\sum_{k_1=0}^3 B_{1,k_1}(1)\bigg)^{m_1} \bigg(\sum_{k_2=0}^3 B_{2,k_2}(1, 1)\bigg)^{m_2}  \bigg(\sum_{k_3=0}^3 B_{3,k_3}(1, 1, 1)\bigg)^{m_3} = 250
%\]
%diagrams:
\begin{figure}[H] %0.28 max
	\centering
	\includegraphics[scale=0.28]{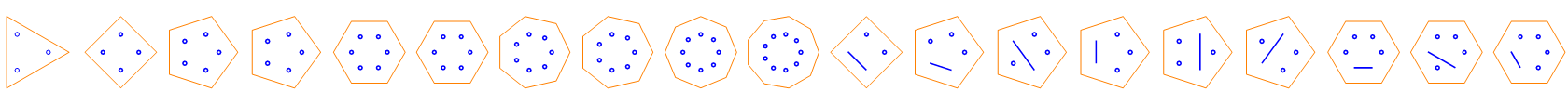}
	\includegraphics[scale=0.28]{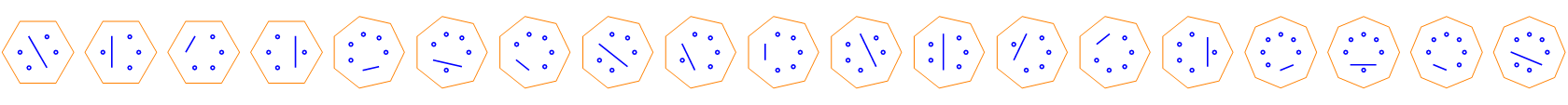}
	\includegraphics[scale=0.28]{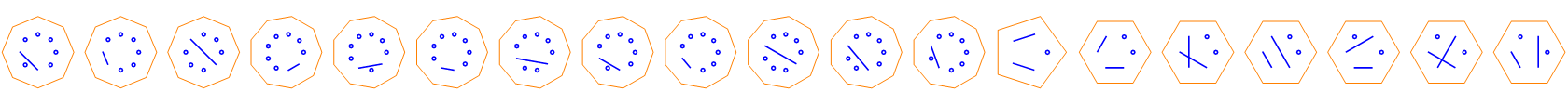}
	\includegraphics[scale=0.28]{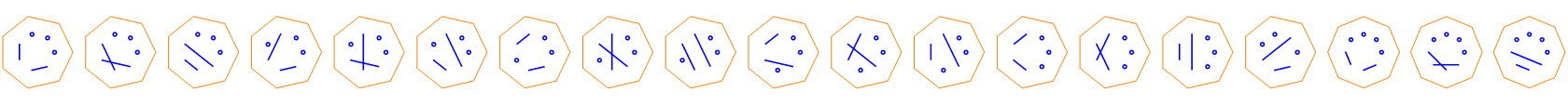}
	\includegraphics[scale=0.28]{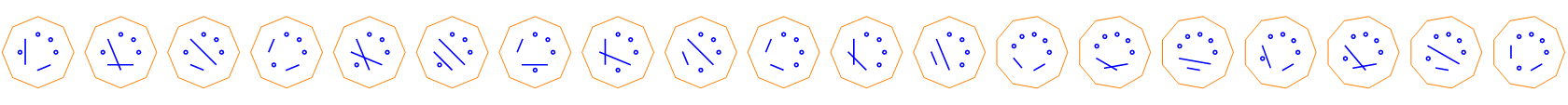}
	\includegraphics[scale=0.28]{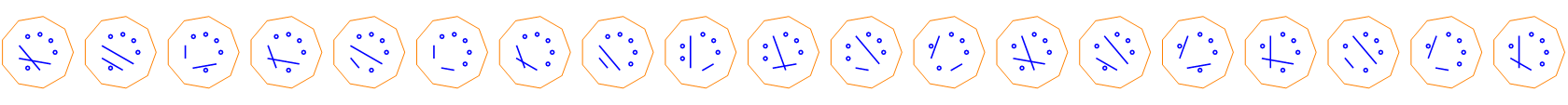}
	\includegraphics[scale=0.28]{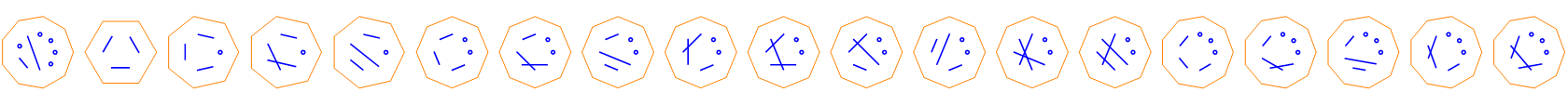}
	\includegraphics[scale=0.28]{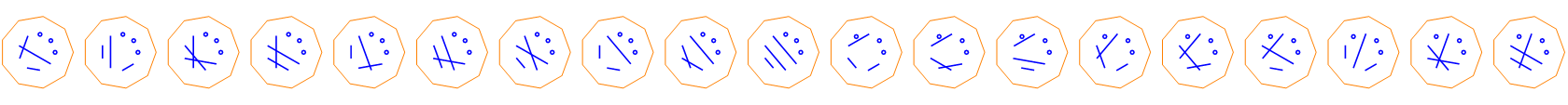}
	\includegraphics[scale=0.28]{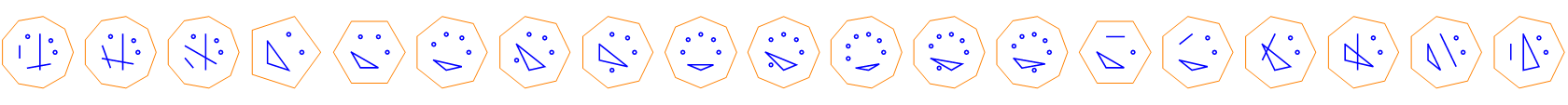}
	\includegraphics[scale=0.28]{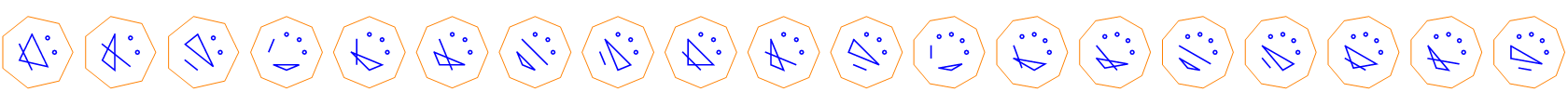}
	\includegraphics[scale=0.28]{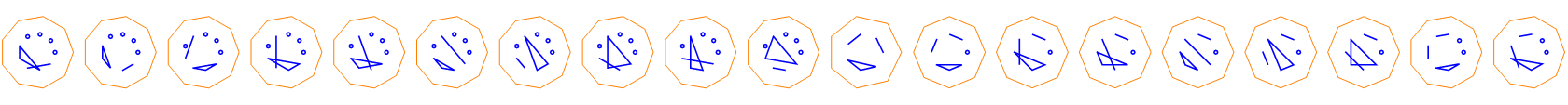}
	\includegraphics[scale=0.28]{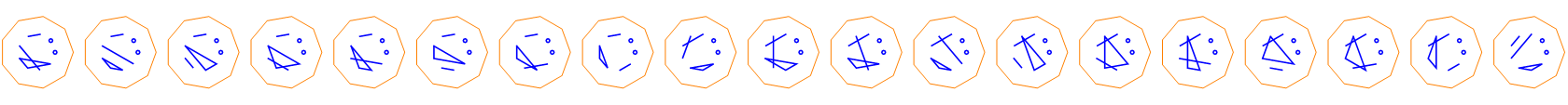}
	\includegraphics[scale=0.28]{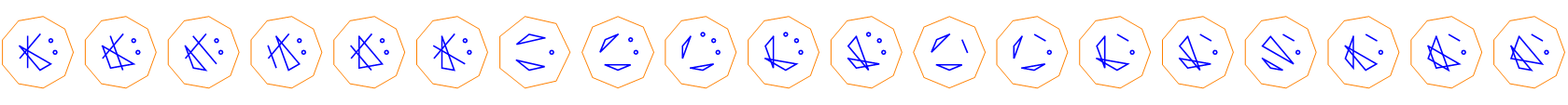}
	\includegraphics[scale=0.28]{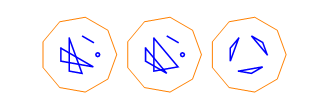}
	\caption{Pictorial representation of \eqref{bellp3}.}
\end{figure}
%\foreach \x in {1,2,3,4,5,6,7,8,9,10,11,12}
%{ %\lipsum[\x]
%    \includegraphics[scale=1]{linesof8_gr\x.eps}
%  %  \clearpage
%}
%
%We \textcolor{red}{need to code a small program that allows us to select (or returns as output) the diagrams such that
%\[
%\sumdotss_{\substack{0 \le k_1, \cdots ,k_{d-1},k_d \le K \\ k_1 + \cdots + k_{d-1} + k_d = K}} B_1(x_1)^{k_1} \cdots B_{k_{d-1}}(x_1,x_2,\cdots,x_{k_{d-1}})^{k_{d-1}}B_{k_d}(x_1,x_2,\cdots,x_{k_d})^{k_d},
%\]
%where $K$ and $d$ are parameters of our choice.}
%%%%%%%%%%%%%%%%%%%%%%%%%%%%%%%%%%%%%%%%%%%%%%%%%%%%%%%%%%%%%%%%%%%%%%%%%%%%%%%%%%%

\section{Main result for the moment integral}
Recall that $L := \log T$ and let
\[
\psi_1(s) := \sum_{n \le N} \frac{a_n}{n^s}, \quad \psi_2(s) := \sum_{n \le N} \frac{b_n}{n^s}, \quad \textnormal{with} \quad a_n, b_n \ll_\varepsilon n^\varepsilon, \quad N:= T^{\theta} \quad \textnormal{and} \quad \theta<1.
\]
Moreover, we shall denote the twisted second moment by
\begin{align*}
I(\alpha,\beta) := \int_{-\infty}^\infty \zeta(\tfrac{1}{2}+\alpha+it)\zeta(\tfrac{1}{2}+\beta-it) \psi_1 \overline{\psi_2} (\tfrac{1}{2}+it) \Phi\bigg(\frac{t}{T}\bigg)dt,
\end{align*}
where $\Phi$ is a smooth function supported on $[1,2]$ and satisfying $\Phi^{(j)}(x) \ll_j \log^j T$. The starting point is the following improvement of \cite[Theorem 1.2]{pr01}.
\begin{theorem} \label{meanvalueintegral}
Let $\alpha, \beta \ll L^{-1}$. Then one has
\begin{align*}
I(\alpha,\beta) = \sumtwo_{1 \le d,e \le N}\frac{a_d \overline{b_e}}{[d,e]} \frac{(d,e)^{\alpha+\beta}}{d^\alpha e^\beta} \int_{-\infty}^\infty \bigg(\zeta(1+\alpha+\beta)+\zeta(1-\alpha-\beta)\bigg(\frac{2\pi de}{t(d,e)^2}\bigg)^{\alpha+\beta} \bigg)\Phi\bigg(\frac{t}{T}\bigg)dt + O(\mathcal{E}),
\end{align*}
with $\mathcal{E}$ given by the following choices
\begin{equation}
\mathcal{E} = 
\begin{cases}
T^{\frac{3}{20}}N^{\frac{33}{20}}+N^{\frac{1}{2}}T^{\frac{1}{2}+\varepsilon}, & \mbox{if} \quad a_n \ll n^{\varepsilon}, \\
T^{\varepsilon}(N^{\frac{11}{6}}+N^{\frac{11}{12}}T^{\frac{1}{2}}), & \mbox{if} \quad a_n = \mu^2(n) (\mu \star \Lambda_1^{\star k_1}\star \Lambda_2^{\star k_2} \star \cdots \star \Lambda_D^{\star k_D})(n), \\
T^{\varepsilon}(N^{\frac{7}{4}}+N^{\frac{7}{8}}T^{\frac{1}{2}}), & \mbox{if} \quad a_n =  (\mu \star \Lambda_1^{\star k_1}\star \Lambda_2^{\star k_2} \star \cdots \star \Lambda_D^{\star k_D})(n). \nonumber
\end{cases}
\end{equation}
%with $f \in \mathcal{F}$.
\end{theorem}

We note that the second case was only proved for $D=1$ in \cite{pr01}, but the proof below shows that it can be adapted to $D \ge 0$. This second case will no longer be needed as we can `improve' it to the third case by relaxing the condition that discriminates square-free numbers\footnote{It is not exactly an improvement but rather a different problem altogether.}. However, we leave it in the theorem for chronological accuracy or in case it becomes useful in another moment integral problem of this type. The first case (when $\alpha=\beta=0$) is due to Bettin, Chandee and Radziwi\l{}\l{} \cite{bcr}. As mentioned earlier, the key to that result is the recent improvement of bilinear Kloosterman sums of Duke, Friedlander and Iwaniec \cite{dfi} due to Bettin and Chandee \cite{bc}.

Effectively, this means that if one considers a Dirichlet polynomial whose coefficients are given by the third case, then one can `push' the size of $\theta$ from $\frac{6}{11}$ to $\frac{4}{7}$. Also note that if $D=0$ in the third case, then one recovers $\mu \star \Lambda^{\star 0} = \mu$, that is the Conrey-Levinson mollifier. This means that the Feng mollifier \cite{feng, krz01, rrz01} and all its generalizations can be taken to have size $\theta_d = \frac{4}{7} - \varepsilon$ for $D \ge 0$ just as in Conrey's mollifier.

\begin{proof}[Proof of Theorem \textnormal{\ref{meanvalueintegral}}]
We need to adapt the proof appearing in \cite[$\mathsection$3.2.5]{pr01} since only the error terms are affected and the main terms remain exactly the same.

%Essentially, the relaxation on square-free numbers allows us to complete the incomplete Kloosterman sums appearing in \cite{pr01}. Once the term $\mu^2$ is removed, it was shown in \cite[$\mathsection$3.2.5]{pr01} that the size the error would be precisely the third case above. 

We assume familiarity with \cite{pr01} and its notation. Following through the proof of \cite[$\mathsection$3.2.5]{pr01}, we must bound the quantity
\begin{align*}
\sum_{0 < |a| < A} \nu_{x,y}(a) \mathop{\sum \sum}_{(n_1,n_2)=1} \frac{a_{dn_1}F_{N_1}(dn_1)r(n_2)}{n_1^{\alpha + w}}\e \left(-a \frac{\overline{n_1}}{n_2}\right),
\end{align*}
where $|r(n)| \ll n^\epsilon$ is some function. Recall that $n_i \asymp N_i/d$ and $n_1 \leq N$. The coefficients $a_m$ are a finite  linear combination of functions, so using linearity and the definition of the $a_m$ we see it suffices to bound
\begin{align*}
\sum_{0 < |a| < A} \nu_{x,y}(a) \mathop{\sum \sum}_{(n_1,n_2)=1} \frac{n_1^{\sigma_0 - \frac{1}{2}} (\mu \star \Lambda^{\star\ell_1} \star \cdots \star \Lambda_d^{\star\ell_d})(dn_1) F_{N_1}(dn_1)P \left(\frac{\log(N/dn_1)}{\log N} \right) r(n_2)}{n_1^{\alpha + w}}\e \left(-a \frac{\overline{n_1}}{n_2}\right).
\end{align*}
Using the binomial theorem and the additivity of the logarithm we may separate $d$ and $n_1$ from one another in the polynomial $P$. We see it suffices to bound
\begin{align*}
\sum_{0 < |a| < A} \nu_{x,y}(a) \mathop{\sum \sum}_{(n_1,n_2)=1} \frac{n_1^{\sigma_0 - \frac{1}{2}} (\log n_1)^j (\mu \star \Lambda^{\star\ell_1} \star \cdots \star \Lambda_d^{\star\ell_d})(dn_1) F_{N_1}(dn_1) r(n_2)}{n_1^{\alpha + w}}\e \left(-a \frac{\overline{n_1}}{n_2}\right),
\end{align*}
where $j$ is some fixed, nonnegative integer.

We are now faced with the task of separating $d$ and $n_1$ in the arithmetic factor 
\[
(\mu \star \Lambda^{\star\ell_1} \star \cdots \star \Lambda_d^{\star\ell_d})(dn_1) F_{N_1}(dn_1).
\] 
In the second case of the theorem we had a $\mu^2$ factor in our coefficients which meant we could automatically take $d$ and $n_1$ to be coprime to one another, and this simplified things somewhat.

We factor $n_1 \rightarrow h n_1$, where $h \mid d^\infty$ and $n_1$ is coprime to $d$. The quantity to bound therefore becomes
\begin{align*}
&\sum_{\substack{h \mid d^\infty \\ h \ll N_1/d}} \frac{h^{\sigma_0 - \frac{1}{2}}}{h^{\alpha + w}} \sum_{0 < |a| < A} \nu_{x,y}(a) \\ 
& \quad \times \mathop{\sum \sum}_{\substack{n_1 \leq N \\ n_1 \asymp N_1/dh \\ N_2 \asymp N_2/d \\ (n_1,d n_2)=1 \\ (n_2,h)=1}} \frac{n_1^{\sigma_0 - \frac{1}{2}} (\log hn_1)^j (\mu \star \Lambda^{\star\ell_1} \star \cdots \star \Lambda_d^{\star\ell_d})(dhn_1) F_{N_1}(dhn_1) r(n_2)}{n_1^{\alpha + w}}\e \left(-a \frac{\overline{hn_1}}{n_2}\right).
\end{align*}
By the additivity of the logarithm and the binomial theorem we may separate $h$ and $n_1$ in $(\log hn_1)^j$, so that we must bound
\begin{align*}
&\sum_{\substack{h \mid d^\infty \\ h \ll N_1/d}} \frac{h^{\sigma_0 - \frac{1}{2}}(\log h)^k}{h^{\alpha + w}} \sum_{0 < |a| < A} \nu_{x,y}(a) \\ 
& \quad \times \mathop{\sum \sum}_{\substack{n_1 \leq N \\ n_1 \asymp N_1/dh \\ N_2 \asymp N_2/d \\ (n_1,d n_2)=1 \\ (n_2,h)=1}} \frac{n_1^{\sigma_0 - \frac{1}{2}} (\log n_1)^j (\mu \star \Lambda^{\star\ell_1} \star \cdots \star \Lambda_d^{\star\ell_d})(dhn_1) F_{N_1}(dhn_1) r(n_2)}{n_1^{\alpha + w}}\e \left(-a \frac{\overline{hn_1}}{n_2}\right)
\end{align*}
for some nonnegative integers $j$ and $k$ ($j$ is not necessarily the same as before). 

We claim that for all $k \geq 1$ we may write $\Lambda_k(n)$ as a finite linear combination of functions of the form
\begin{align*}
(\log^{j_1}\Lambda \star\log^{j_2}\Lambda \star \cdots \star \log^{j_R}\Lambda)(n),
\end{align*}
where the $j_i$ are nonnegative integers. We proceed by induction. The base case $k = 1$ is trivial. Now assume it is true for $k$. The recurrence formula gives
\begin{align*}
\Lambda_{k+1}(n) = \log(n) \Lambda_k(n) + (\Lambda \star\Lambda_k)(n).
\end{align*}
By the induction hypothesis $\Lambda_k$ is a linear combination of functions of the desired form, and therefore so is $\Lambda \star\Lambda_k$. To see that $\log(n) \Lambda_k(n)$ is also of the desired form, it suffices to apply the induction hypothesis and note that for any arithmetic functions $f_1, \cdots, f_J$ we have
\begin{align*}
\log(n) (f_1 \star \cdots \star f_J)(n) = \sum_{i=1}^J (g_{1,i} \star \cdots \star g_{J,i})(n),
\end{align*}
where $g_{j,i} = f_j$ if $i \neq j$, and $g_{i,i} = f_i \log$.

The quantity to bound is therefore a finite linear combination of quantities of the form
\begin{align*}
&\sum_{\substack{h \mid d^\infty \\ h \ll N_1/d}} \frac{h^{\sigma_0 - \frac{1}{2}} (\log h)^k}{h^{\alpha + w}} \sum_{0 < |a| < A} \nu_{x,y}(a) \\ 
& \quad \times \mathop{\sum \sum}_{\substack{n_1 \leq N \\ n_1 \asymp N_1/dh \\ N_2 \asymp N_2/d \\ (n_1,d n_2)=1 \\ (n_2,h)=1}} \frac{n_1^{\sigma_0 - \frac{1}{2}} (\log n_1)^j (\mu \star \log^{j_1}\Lambda \star \cdots \star \log^{j_R}\Lambda)(dhn_1) F_{N_1}(dhn_1) r(n_2)}{n_1^{\alpha + w}}\e \left(-a \frac{\overline{hn_1}}{n_2}\right).
\end{align*}
We now give the argument that shows how to separate $n_1$ from $dh$ in expressions of this form. We have
\begin{align*}
(\mu \star \log^{j_1}\Lambda \star \cdots \star \log^{j_R}\Lambda)(dhn_1) &= \mathop{\sum \cdots \sum}_{m\ell_1 \cdots \ell_R = dhn_1} \mu(m) \log^{j_1}(\ell_1)\Lambda(\ell_1) \cdots \log^{j_R}(\ell_R) \Lambda(\ell_R).
\end{align*}
Since $\log(\ell_i)^{j_i}\Lambda(\ell_i)$ is supported on prime powers and $n_1$ is coprime to $dh$, we see that $(\mu \star\log^{j_1}\Lambda \star \cdots \star \log^{j_R}\Lambda)(dhn_1)$ is the sum of a bounded number of functions of the form
\begin{align*}
\mathop{\sum \cdots \sum}_{\substack{m\ell_1 \cdots \ell_R = dhn_1 \\ \ell_{i_1}, \cdots, \ell_{i_s} \mid dh \\ \ell_{i_{s+1}},\cdots,\ell_{i_R} \mid n_1}} \mu(m) \log^{j_1}(\ell_1)\Lambda(\ell_1) \cdots \log^{j_R}(\ell_R) \Lambda(\ell_R).
\end{align*}
We write $m = m'm''$, where $m' \mid dh$ and $m'' \mid n_1$ to see that this last quantity is
\begin{align*}
(\mu \star \log^{j_{i_1}}\Lambda \star \cdots \star \log^{j_{i_s}}\Lambda)(dh)(\mu \star \log^{j_{i_{s+1}}}\Lambda \star \cdots \star \log^{j_{i_R}}\Lambda)(n_1),
\end{align*}
which gives the desired separation. It follows that the quantity in question is a finite linear combination of sums of the form
\begin{align*}
&\sum_{\substack{h \mid d^\infty \\ h \ll N_1/d}} \frac{G(h)}{h^{\alpha + w}} \sum_{0 < |a| < A} \nu(a) \\ 
& \quad \times \mathop{\sum \sum}_{\substack{n_1 \leq N \\ n_1 \asymp N_1/dh \\ n_2 \asymp N_2/d \\ (n_1,d n_2)=1 \\ (n_2,h)=1}} \frac{n_1^{\sigma_0 - \frac{1}{2}} (\log n_1)^j (\mu \star\log^{j_1}\Lambda \star \cdots \star \log^{j_R}\Lambda)(n_1) F_{N_1}(dhn_1) r(n_2)}{n_1^{\alpha + w}}\e \left(-a \frac{\overline{hn_1}}{n_2}\right),
\end{align*}
where $G$ is some function satisfying $|G(h)| \ll \tau(dh)^{O(1)} (\log N)^{O(1)}$, and the $j_i$ are nonnegative integers that are not necessarily the same as before.

Now that we have separated the variables, we perform a combinatorial decomposition on the function $(\mu \star\log^{j_1}\Lambda \star \cdots \star \log^{j_R}\Lambda)(n_1)$ to reduce to Type I and Type II exponential sums. For $n \leq x$ and $K \geq 1$ a fixed integer, the multinomial theorem and Heath-Brown's identity \cite{heathbrown} imply
\begin{align*}
(\log n)^J \Lambda(n) &= \sum_{j=1}^K (-1)^{j-1}{K \choose j} \mathop{\sum \cdots \sum}_{k_1 + \cdots + k_{2j} = J} {J \choose {k_1,\ldots,k_{2j}}} \\
& \quad \times \mathop{\sum \cdots \sum}_{\substack{m_1 \cdots m_{2j} = n \\ m_i \leq x^{1/K}, \ i \leq j}} \mu(m_1) \cdots \mu(m_j) \log(m_{2j})\prod_{t=1}^{2j} \log^{k_t}(m_t).
\end{align*}
Performing the usual dyadic decompositions and combinatorial maneuvers (see \cite{pr01} for more details), we see that the sum in question is a linear combination of $O((N_1/dh)^\epsilon)$ sums of the form
\begin{align*}
\sum_{0 < |a| < A} \nu(a) \sum_{\substack{v \asymp V \\ (v,h)=1}}r(v)\sum_{\substack{u \leq N \\ u \asymp U \\ (u,dv)=1}} \frac{u^{\sigma_0 - \frac{1}{2}}(\log u)^k}{u^{\alpha + w}} (\alpha \star \beta)(u) F_{N_1}(dhu) \e \left(-a \frac{\overline{hu}}{v} \right),
\end{align*}
where $\alpha$ and $\beta$ are arithmetic functions supported on dyadic intervals, $U \asymp N_1/dh, V \asymp N_2/d$. We have dropped the $h$ summation, temporarily, but we shall return to it later. In the Type I case we have that $\alpha$ is supported on integers $n \leq W$ and $\beta = \log^S$ for some nonnegative integer $S$. Here $W \ll (N_1/dh)^{1/3}$ is a parameter at our disposal. In the Type II case we have that $\alpha$ and $\beta$ are supported on integers in $[W,(N_1/dh)W^{-1}]$.

Let us consider first the case of a Type I sum. Using the binomial theorem to separate variables, we have
\begin{align*}
\sum_{0 < |a| < A} \nu(a) \sum_{\substack{v \asymp V \\ (v,h)=1}}r(v) \sum_{\substack{e \asymp E \\ (e,dv)=1}} \alpha(e) \sum_{\substack{f \leq N/e \\ f \asymp U/E \\ (f,dv)=1}} \frac{f^{\sigma_0 - \frac{1}{2}}(\log f)^j}{f^{\alpha + w}}F_{N_1}(dhef) \e \left(-a \frac{\overline{hef}}{v} \right),
\end{align*}
where $\alpha$ is not necessarily the same as before, and $E \ll W$. We use summation by parts and the usual Weil bound for incomplete Kloosterman sums \cite{weilbound} to obtain that the sum on $f$ is
\begin{align*}
\ll (1+|w|)T^\epsilon V^{1/2} (a,v) \left(1 + \frac{U}{E V}\right).
\end{align*}
Summing over all our variables and recalling $E \ll W$, we obtain a Type I bound of
\begin{align*}
(1+|w|)T^\epsilon (AV^{3/2}W + AUV^{1/2}).
\end{align*}

We now turn to a Type II sum. We use the binomial theorem to separate variables in $\log^j$ and the Mellin transform to separate variables in $F_{N_1}$. We obtain
\begin{align*}
\sum_{0 < |a| < A} \nu(a) \sum_{\substack{v \asymp V \\ (v,h)=1}}r(v) \mathop{\sum \sum}_{\substack{e \asymp E \\ f \asymp F \\ (ef,v)=1}} \alpha(e) \beta(f)\e \left(-a \frac{\overline{hef}}{v} \right),
\end{align*}
where $EF \asymp U$ and $W \ll E\ll F \ll UW^{-1}$. By a result of Deshouillers and Iwaniec (\cite[Lemma 1]{deshouillersIwaniec2}, the variable $h$ corresponds to $\rho$ and bounds on $h$ arise from that lemma) we obtain
\begin{align*}
\ll \ &E^{-1/2}A^{1/2}UV + h^{1/4} A UV^{1/2} + h^{1/4}E^{1/4} A^{3/4}UV^{1/2} + h^{1/4} E^{-1/4}A U^{3/4} V \\ 
&+ h^{1/2}E^{1/4}A U^{3/4}V^{3/4} + h^{1/4} E^{1/4}A^{1/2}U^{3/4}V + h^{1/2}E^{3/4}A^{1/2}U^{3/4}V^{3/4}.
\end{align*}
We now set $W = U^{1/4}$ to balance the Type I and Type II bounds. We use the bound $E^{-a} \ll W^{-a} = U^{-a/4}$ for $a < 0$, and $E^b \ll U^{b/2}$ for $b > 0$, and note that in each term the power of $h$ is smaller than the power of $U$. The variable $U$ is of size $N_1/dh$, thus when $U$ appears there is a hidden $h$. The significance of the power of $h$ being less than the power of $U$ is that we need some positive power of $h$ in the denominator at the end of the day so we can sum over $h$ and not have it blow up. Since
\begin{align*}
\sum_{h \mid d^\infty} \frac{G(h)}{h^\epsilon} \ll (dT)^\epsilon,
\end{align*}
we find that, in the notation of \cite{pr01}, we have
\begin{align*}
\mathcal{A}_{M,N_1,N_2}^* \ll \frac{T^\epsilon}{d} (T^{1/2}N^{7/8} + N^{7/4} ).
\end{align*}
Thus, taking $N = T^{\frac{4}{7} - \varepsilon}$ is permissible.
\end{proof}

\section{The square-free terms and Feng's conjecture}

We pause at this point to explain what happened to Feng's conjecture (unproved claim, rather) that only square-free terms contribute to the integral appearing in Theorem \ref{meanvalueintegral}. It is not entirely clear how Feng guessed that numbers containing a square would `contribute a lower order term for the mean value integral', \cite[p. 516]{feng}.\\

It is plausible that the approach Feng followed, which is a mixture of elementary inductions juxtaposed with Mertens's formulas $\sum_{p \le y} \frac{\log p}{p} = \log y + O(1)$ and $\sum_{p|n} \frac{\log p}{p} \ll \log \log n$, only works reasonably well if $n$ is square-free. Feng re-writes the mollifier by removing the Dirichlet convolutions and explicitly writing their result. He uses
\begin{align} \label{convolution11}
(\mu \star \Lambda^{\star k})(n) = (-1)^k \mu(n) \sum_{p_1 p_2 \cdots p_k | n} \log p_1 \log p_2 \cdots \log p_k,
\end{align}
which is only valid when $n$ is square-free. Attempting to obtain a general formula for $(\mu \star \Lambda^{\star k})(n)$ when $n$ contains a square runs into combinatorial difficulties (see also \cite[p. 309]{levinsoncollected} for a similar commentary on mollifying $\frac{1}{\zeta'}$ instead of $\frac{\zeta'}{\zeta}$). In fact, a key result of this is Lemma 9. This lemma explicitly demands that we deal only with square-free numbers when multiplying mollifiers. Again, adapting such a result for arbitrary $n$ would not be easy.\\

For $\ell = 1$, $(\mu \, \star \, \Lambda^{\star \ell}) = \mu^2(n)(\mu\, \star \,\Lambda^{\star \ell})$,the plots below illustrate the differences between the unrestricted convolution $(\mu \star \Lambda^{\star \ell})$ and the `square-free convolution' $\mu^2(n)(\mu \star \Lambda^{\star \ell})$ for $\ell=2$ and $\ell=3$, respectively.   

\begin{figure}[H]
\centering
\includegraphics[scale=0.89]{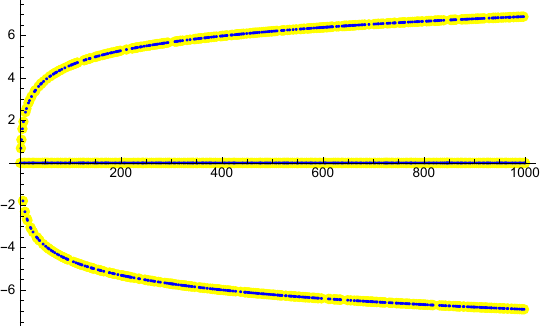}
\includegraphics[scale=0.89]{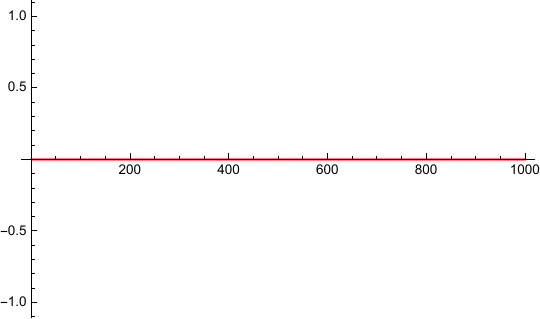}
\caption{\underline{Left}: Plots of $\mu^2(n)(\mu^{\star 2} \star \log n)(n)$ in yellow, $(\mu^{\star 2} \star \log n)(n)$ in blue. \underline{Right}: Plot of $(\mu^{\star 2} \star \log n)(n)-\mu^2(n)(\mu^{\star 2} \star \log n)(n)$, for $1 \le n \le 1000$.}
\end{figure}
Next
\begin{figure}[H]
\centering
\includegraphics[scale=0.89]{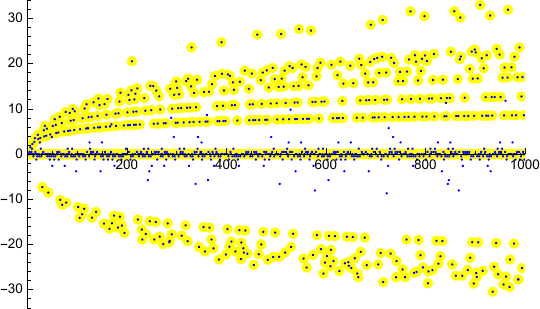}
\includegraphics[scale=0.89]{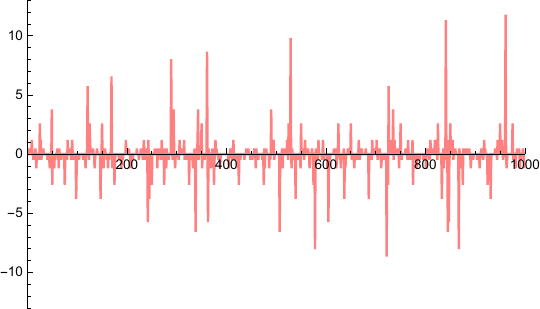}
\caption{\underline{Left}: Plots of $\mu^2(n)(\mu^{\star 3} \star \log^{\star 2} n)(n)$ in yellow, $(\mu^{\star 3} \star \log^{\star 2} n)(n)$ in blue. \underline{Right}: Plot of $(\mu^{\star 3} \star \log^{\star 2} n)(n)-\mu^2(n)(\mu^{\star 3} \star \log^{\star 2} n)(n)$, for $1 \le n \le 1000$.}
\end{figure}
Moreover
\begin{figure}[H]
\centering
\includegraphics[scale=0.89]{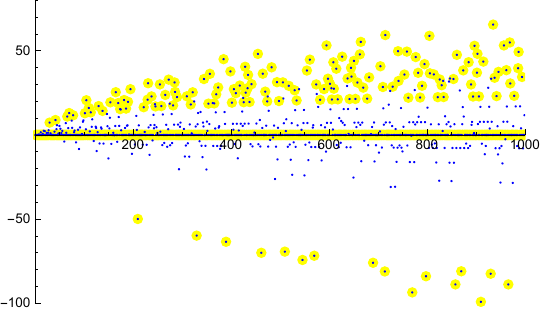}
\includegraphics[scale=0.89]{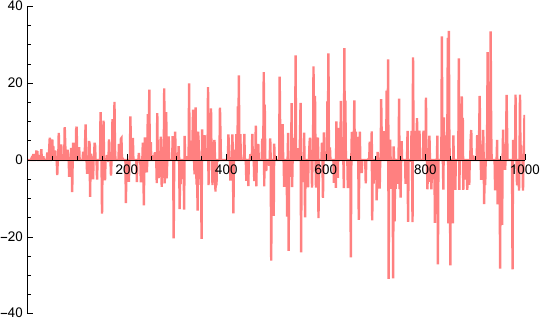}
\caption{\underline{Left}: Plots of $\mu^2(n)(\mu^{\star 4} \star \log^{\star 3} n)(n)$ in yellow, $(\mu^{\star 4} \star \log^{\star 3} n)(n)$ in blue. \underline{Right}: Plot of $(\mu^{\star 4} \star \log^{\star 3} n)(n)-\mu^2(n)(\mu^{\star 4} \star \log^{\star 3} n)(n)$, for $1 \le n \le 1000$.}
\end{figure}
The plots below illustrate the difference between the partial sums of the restricted and unrestricted convolutions.
\begin{figure}[H]
\centering
\includegraphics[scale=0.89]{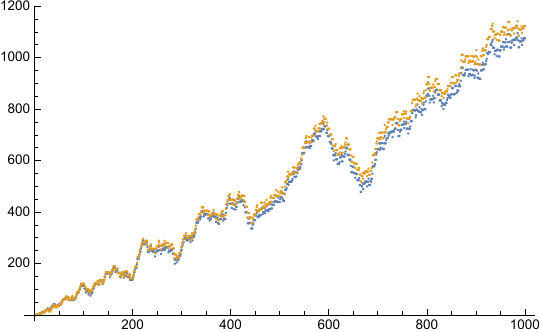}
\includegraphics[scale=0.89]{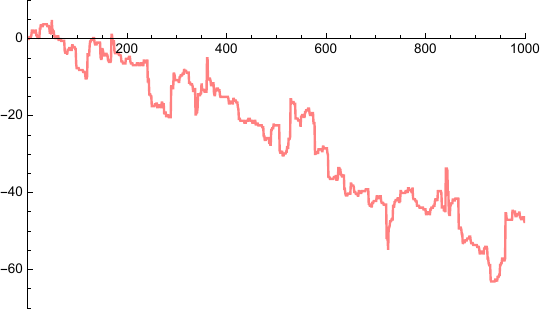}
\caption{\underline{Left}: Plots of $\sum_{n \le x} (\mu \star \Lambda^{\star 2})(n)$ in blue and $\sum_{n \le x} \mu^2(n)(\mu \star \Lambda^{\star 2})(n)$ in orange for $1 \le x \le 1000$. \underline{Right}: Plots of $\sum_{n \le x} (\mu \star \Lambda^{\star 2})(n)-\sum_{n \le x} \mu^2(n)(\mu \star \Lambda^{\star 2})(n)$ for $1 \le x \le 1000$}
\end{figure}
\begin{figure}[H]
\centering
\includegraphics[scale=0.89]{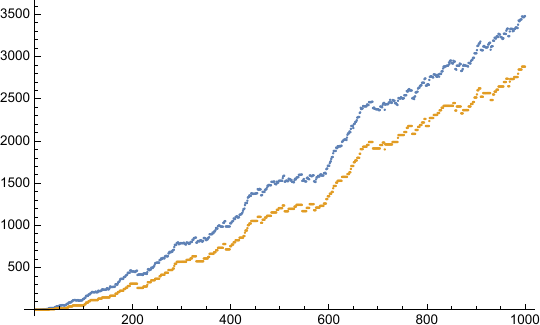}
\includegraphics[scale=0.89]{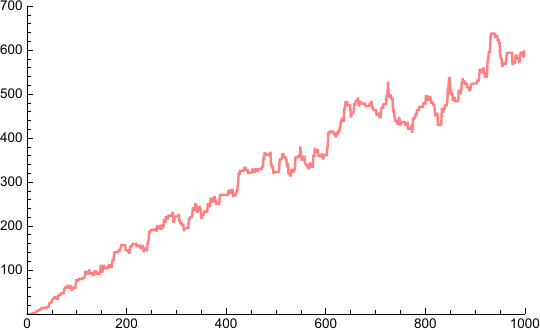}
\caption{\underline{Left}: Plots of $\sum_{n \le x} (\mu \star \Lambda^{\star 3})(n)$ in blue and $\sum_{n \le x} \mu^2(n)(\mu \star \Lambda^{\star 3})(n)$ in orange for $1 \le x \le 1000$. \underline{Right}: Plots of $\sum_{n \le x} (\mu \star \Lambda^{\star 3})(n)-\sum_{n \le x} \mu^2(n)(\mu \star \Lambda^{\star 3})(n)$ for $1 \le x \le 1000$}
\end{figure}
%In the $\mathsection$5.1, we will prove Feng's claim.

\subsection{Proof of Feng's conjecture}

We examine here only the case $K = 2$; the case of higher values of $K$ works out similarly. We examine the main term
\begin{align}\label{eq: main term expression}
T \widehat{\Phi}(0) \zeta(1+\alpha + \beta) \mathop{\sum \sum}_{d,e \leq N} \frac{a_d a_e}{[d,e]} \frac{(d,e)^{\alpha + \beta}}{d^\alpha e^\beta},
\end{align}
where $a_d$ is our mollifier coefficient for $K = 2$, i.e.
\begin{align*}
a_n = (\mu \star \Lambda^{\star 2})(n)P\left( \frac{\log N/n}{\log N}\right),
\end{align*}
where $P$ is a polynomial satisfying $P(0) = 0$. We write $a_d = b_d + c_d$, where $b_d$ is the Feng coefficient
\begin{align*}
b_d = \mu(d)P\left( \frac{\log N/d}{\log N}\right)\mathop{\sum \sum}_{q_1 \neq q_2 \mid d} \frac{(\log q_1)(\log q_2)}{(\log N)^2} 
\end{align*}
and $c_d$ is the difference of $a_d$ and $b_d$. The reader may like to prove, as an exercise, that $c_d$ vanishes unless $d$ is of the form $d = p^k m$, where $p$ is a prime, $k \geq 2$, and $m$ is a squarefree integer coprime to $p$, in which case one has
\begin{align*}
c_{d} = \frac{1}{(\log N)^2} (\log p)^2 \mu(m).
\end{align*}
Putting $a_d = b_d + c_d$ into \eqref{eq: main term expression} we get the ``Feng'' part, which has $b_d$ and $b_e$, and then three other terms all of which involve at least one of $c_d$ or $c_e$. We expect that these latter three terms are negligible.

We investigate the crossed case with $b_d$ and $c_e$. Write $h(d)$ for the sum over $q_1$ and $q_2$. Using what we know about $c_e$, and ignoring the factor of $T \widehat{\Phi}(0)$, we see we may write this expression as
\begin{align*}
\frac{\zeta(1+\alpha + \beta)}{ (\log N)^{2}}\sum_{d \leq N} \frac{\mu(d)h(d)P\left( \frac{\log N/d}{\log N}\right)}{d^\alpha} \mathop{\sum \sum}_{\substack{p^k m \leq N \\ k \geq 2 \\ (m,p)=1}} \frac{(\log p)^2 \mu(m) (d,p^k m)^{\alpha + \beta}}{[d,p^k m]p^{k \beta}m^\beta}P\left( \frac{\log N/p^k m}{\log N}\right).
\end{align*}
Now we put the sum on $p^k$ as the outermost sum. There are two cases to consider: $p \mid d$ and $p \nmid d$. We consider here only the ``generic'' case $p \nmid d$. We therefore have
\begin{align*}
\frac{\zeta(1+\alpha + \beta)}{(\log N)^2}\sum_{\substack{p^k \leq N \\ k \geq 2}} \frac{(\log p)^2}{p^{(1+\beta)k}}\sum_{\substack{d \leq N \\ (d,p)=1}} \frac{\mu(d)h(d)P\left( \frac{\log N/d}{\log N}\right)}{d^\alpha} \sum_{\substack{m \leq N/p^k \\ (m,p)=1}} \frac{ \mu(m) (d, m)^{\alpha + \beta}}{[d,m]m^\beta}P\left( \frac{\log N/p^k m}{\log N}\right).
\end{align*}

We now change variables $m \rightarrow em$, where $e \mid d$ and $(m,d)=1$, and arrive at
\begin{align*}
\frac{\zeta(1+\alpha + \beta)}{(\log N)^2}&\sum_{\substack{p^k \leq N \\ k \geq 2}} \frac{(\log p)^2}{p^{(1+\beta)k}}\sum_{\substack{d \leq N \\ (d,p)=1}} \frac{\mu(d)h(d)P\left( \frac{\log N/d}{\log N}\right)}{d^{1+\alpha}} \\
&\times\sum_{\substack{e \mid d \\ e \leq N/p^k}} \mu(e) e^\alpha \sum_{\substack{m \leq N/ep^k \\ (m,dp)=1}} \frac{ \mu(m)}{m^{1+\beta}}P\left( \frac{\log N/ep^k m}{\log N}\right).
\end{align*}
Next, we interchange the order of summation so the $e$-sum is the next sum after the $p^k$-sum. We change variables $d \rightarrow ed$ to get
\begin{align*}
\frac{\zeta(1+\alpha + \beta)}{(\log N)^2}&\sum_{\substack{p^k \leq N \\ k \geq 2}} \frac{(\log p)^2}{p^{(1+\beta)k}}\sum_{\substack{e \leq N/p^k \\ (e,p)=1}} \frac{\mu^2(e)}{e}\sum_{\substack{d \leq N/e \\ (d,ep)=1}} \frac{\mu(d) h(de)P\left( \frac{\log N/de}{\log N}\right)}{d^{1+\alpha}} \\
&\times \sum_{\substack{m \leq N/ep^k \\ (m,dep)=1}} \frac{ \mu(m)}{m^{1+\beta}}P\left( \frac{\log N/ep^k m}{\log N}\right).
\end{align*}
To make the $d$- and $m$-sums independent from one another, we apply M\"obius inversion to remove the condition that $m$ and $d$ are coprime. Interchanging orders of summation and changing variables yet again, we obtain
\begin{align*}
\frac{\zeta(1+\alpha + \beta)}{(\log N)^2}&\sum_{\substack{p^k \leq N \\ k \geq 2}} \frac{(\log p)^2}{p^{(1+\beta)k}}\sum_{\substack{e \leq N/p^k \\ (e,p)=1}} \frac{\mu^2(e)}{e} \sum_{\substack{r \leq N/ep^k \\ (r,ep)=1}} \frac{\mu^2(r)}{p^{2+\alpha + \beta}} \sum_{\substack{d \leq N/er \\ (d,epr)=1}} \frac{\mu(d) h(der)P\left( \frac{\log N/der}{\log N}\right)}{d^{1+\alpha}} \\
&\times \sum_{\substack{m \leq N/ep^kr \\ (m,epr)=1}} \frac{ \mu(m)}{m^{1+\beta}}P\left( \frac{\log N/emp^kr}{\log N}\right).
\end{align*}

Observe that for coprime squarefree integers $m,n$ we have
\begin{align*}
h(mn) = h(m) + h(n) + 2\frac{(\log m)(\log n)}{(\log N)^2}.
\end{align*}
By two applications of this rule and exploiting the additivity of the logarithm we find
\begin{align*}
h(der) = h(d) + h(e) + h(r) + 2 \frac{(\log d)(\log e)}{(\log N)^2} + 2 \frac{(\log d)(\log r)}{(\log N)^2} + 2 \frac{(\log e)(\log r)}{(\log N)^2}.
\end{align*}
Thus, we have six different cases to consider.

\subsubsection{The $m$-sum}

Before we consider these cases, let us estimate the sum over $m$, which does not depend on which of the six cases we are in. For technical reasons we do not proceed immediately to writing the sum as an integral. Instead, it is advantageous to use M\"obius inversion to deal with the coprimality condition $(m,er)= 1$. Our sum becomes
\begin{align*}
\sum_{\substack{f \mid er \\ f \leq N/ep^kr}} \frac{\mu^2(f)}{f^{1+\beta}}\sum_{\substack{m \leq N/efp^kr \\ (m,fp)=1}} \frac{ \mu(m)}{m^{1+\beta}}P\left( \frac{\log N/efmp^kr}{\log N}\right).
\end{align*}
It turns out to be helpful to have this averaging over $f$ at our disposal.

We write
\begin{align*}
P(x) = \sum_{1 \leq j \leq \text{deg}(P)} c_j x^j,
\end{align*}
so our sum becomes
\begin{align*}
\sum_{\substack{f \mid er \\ f \leq N/ep^kr}} \frac{\mu^2(f)}{f^{1+\beta}} \sum_{1 \leq j \leq \text{deg}(P)} \frac{c_j}{(\log N)^j}\sum_{\substack{m \leq N/efp^kr \\ (m,fp)=1}} \frac{ \mu(m)}{m^{1+\beta}}\log (N/efmp^kr)^j.
\end{align*}
Writing the logarithm as an integral and interchanging the order of summation and integration, we have
\begin{align*}
\frac{1}{(\log N)^j}\sum_{\substack{m \leq N/efp^kr \\ (m,fp)=1}} \frac{ \mu(m)}{m^{1+\beta}}\log (N/efmp^kr)^j &= \frac{j!}{(\log N)^j} \frac{1}{2\pi i} \int_{(c)} \left( \frac{N}{efp^k r}\right)^s \bigg(\sum_{(m,fp)=1} \frac{\mu(m)}{m^{1+s+\beta}} \bigg) \frac{ds}{s^{j+1}}.
\end{align*}
Here $c = \frac{1}{\log N}$, say. By an Euler product computation we see that the integrand (apart from the $s^{-j-1}$) is equal to
\begin{align*}
\left( \frac{N}{efp^k r}\right)^s\left(1 - \frac{1}{p^{1+s + \beta}} \right)^{-1}\prod_{q \mid f} \left(1 - \frac{1}{q^{1+s+\beta}}\right)^{-1} \frac{1}{\zeta(1+s+\beta)}.
\end{align*}
We use the classical zero-free region for $\zeta$ and some bounds for $\zeta^{-1}$ close to the $1$-line, see \cite[Ch. III]{titchmarsh}. We can show that the $m$-sum is equal to
\begin{align*}
O \left(\frac{\tau(f)}{\log N} \exp \left(-c \sqrt{\log (N/efp^kr)} \right) \right)
\end{align*}
plus
\begin{align*}
\frac{j!}{(\log N)^j}\operatorname{res}_{s=0}\bigg(\left( \frac{N}{efp^k r}\right)^s\left(1 - \frac{1}{p^{1+s + \beta}} \right)^{-1}\prod_{q \mid f} \left(1 - \frac{1}{q^{1+s+\beta}}\right)^{-1} \frac{1}{\zeta(1+s+\beta)} \frac{1}{s^{j+1}} \bigg).
\end{align*}
We calculate this residue by taking $j$ derivatives in $s$ and then letting $s \rightarrow 0$. One can bound this term by
\begin{align*}
O \left(\frac{(\log \log 3f)^{O_P(1)}}{\log N}\right);
\end{align*}
here we have used the fact that
\begin{align*}
\sum_{p \mid n} \frac{(\log p)^k}{p} \ll_k (\log \log 3n)^k.
\end{align*}
It is also important that $|\beta| \asymp \frac{1}{\log N}$. We therefore obtain a total bound of
\begin{align*}
\sum_{\substack{m \leq N/efp^kr \\ (m,fp)=1}} \frac{ \mu(m)}{m^{1+\beta}}P\left( \frac{\log N/efmp^kr}{\log N}\right) \ll \frac{\tau(f) (\log \log 3f)^{O(1)}}{\log N},
\end{align*}
and upon summing over $f$ we obtain
\begin{align*}
\sum_{\substack{m \leq N/ep^kr \\ (m,epr)=1}} \frac{ \mu(m)}{m^{1+\beta}}P\left( \frac{\log N/emp^kr}{\log N}\right) \ll \frac{(\log \log 3er)^{O(1)}}{\log N}.
\end{align*}

\subsubsection{The $d$-sums}

Let us now turn to the various cases for the sum over $d$. Due to symmetry, there are actually only two $d$-sums we need to consider:
\begin{align*}
\sum_{\substack{d \leq N/er \\ (d,epr)=1}} \frac{\mu(d) h(d)P\left( \frac{\log N/der}{\log N}\right)}{d^{1+\alpha}}
\end{align*}
as well as
\begin{align*}
\frac{1}{\log N}\sum_{\substack{d \leq N/er \\ (d,epr)=1}} \frac{\mu(d) (\log d)P\left( \frac{\log N/der}{\log N}\right)}{d^{1+\alpha}}.
\end{align*}

We start with the sum involving $h(d)$. Opening up $h(d)$ and interchanging the order of summation, we have
\begin{align*}
\frac{1}{(\log N)^2}\mathop{\sum \sum}_{\substack{q_1q_2 \leq N/er \\ q_1 \neq q_2 \\ (q_1q_2,epr)=1}} \frac{(\log q_1)(\log q_2)}{q_1^{1+\alpha}q_2^{1+\alpha}} \sum_{\substack{d \leq N/erq_1q_2 \\ (d,epq_1q_2r)=1}} \frac{\mu(d) P\left( \frac{\log N/derq_1q_2}{\log N}\right)}{d^{1+\alpha}}.
\end{align*}
This inner sum on $d$ is almost identical to the $m$-sum we estimated earlier. We similarly use M\"obius inversion to handle the condition $(d,er) = 1$, and then write the sum as an integral and use the zero-free region for $\zeta$. We find the sum on $d$ is
\begin{align*}
\ll \frac{(\log \log 3er)^{O(1)}}{\log N},
\end{align*}
and we obtain the sum bound upon summing over $q_1$ and $q_2$.

Let us now turn to the other type of $d$-sum. We wish to estimate
\begin{align*}
\frac{1}{\log N}\sum_{\substack{d \leq N/er \\ (d,epr)=1}} \frac{\mu(d) (\log d)P\left( \frac{\log N/der}{\log N}\right)}{d^{1+\alpha}}.
\end{align*}
Since $d$ is squarefree, we may write
\begin{align*}
\log (d) = \sum_{\mathfrak{p} \mid d} \log \mathfrak{p},
\end{align*}
where $\mathfrak{p}$ is a prime (we have used all the Roman letters traditionally associated with primes, so we have to branch out a bit). We interchange the order of the $\mathfrak{p}$- and $d$-summations, and then argue as before. The bounds we obtain are of a similar shape.

\subsubsection{Clean-up}

It is now a routine matter to sum up all of the bounds, obtaining a final bound of
\begin{align*}
T \widehat{\Phi}(0) \zeta(1+\alpha + \beta) \mathop{\sum \sum}_{d,e \leq N} \frac{b_d c_e}{[d,e]} \frac{(d,e)^{\alpha + \beta}}{d^\alpha e^\beta} \ll T \frac{(\log \log N)^{O(1)}}{(\log N)^2}.
\end{align*}

Observe that we could have afforded to lose one more logarithm, and we still would have had an acceptable bound. The reason for this is that since we are looking at Feng's $K = 2$, we get to divide by two logarithms in order to make the coefficients bounded. However, the difference between Feng's $K=2$ and our $K=2$ is like a Conrey mollifier on average, which already has bounded coefficients. This accounts for our bound being one logarithm smaller than we actually need it to be.

\section{Specializing the coefficients} \label{sectionspecial}
Let us now take
\[
a_n = \mathfrak{a}_n P_a\bigg(\frac{\log(N/n)}{\log N}\bigg) \quad \textnormal{and} \quad b_n = \mathfrak{b}_n P_b\bigg(\frac{\log(N/n)}{\log N}\bigg), 
\]
where $P_{(\cdot)}$ is a polynomial associated to the coefficients $a$ or $b$ satisfying some conditions such as $P_{(\cdot)}(0)=0$ and $P_{(\cdot)}(1)=1$. For the moment $\mathfrak{a}_n$ and $\mathfrak{b}_n$ will be arbitrary coefficients that will later have certain parameters that will also affect $P_{(\cdot)}$. By the Mellin representation of $P$ we may write
\[
P_a \bigg(\frac{\log(N/n)}{\log N}\bigg) = \sum_{i=0}^{\operatorname{deg} P_a} \frac{p_{a,i}}{\log^i N} (\log(N/n))^i = \sum_i \frac{p_{a,i}i!}{\log^i N} \frac{1}{2 \pi i} \int_{(1)} \bigg( \frac{N}{n} \bigg)^{s} \frac{ds}{s^{i+1}},
\]
as well as
\[
P_b \bigg(\frac{\log(N/n)}{\log N}\bigg) = \sum_{j=0}^{\operatorname{deg} P_b} \frac{p_{b,j}}{\log^j N} (\log(N/n))^j = \sum_j \frac{p_{b,j}j!}{\log^j N} \frac{1}{2 \pi i} \int_{(1)} \bigg( \frac{N}{n} \bigg)^{u} \frac{du}{u^{j+1}}.
\]
We go back to the right-hand side of Theorem \ref{meanvalueintegral} and write
\begin{align} \label{presumS}
I(\alpha,\beta) &= \sumtwo_{i,j} \frac{p_{a,i}p_{b,j}i!j!}{\log^{i+j}N}\frac{1}{(2\pi i)^2}\int_{(1)}\int_{(1)} \sumtwo_{1 \le d,e \le \infty}\frac{\mathfrak{a}_d \overline{\mathfrak{b}_e}}{[d,e]} \frac{(d,e)^{\alpha+\beta}}{d^{\alpha+s} e^{\beta+u}} \nonumber \\
& \quad \times \int_{-\infty}^\infty \bigg(\zeta(1+\alpha+\beta)+\zeta(1-\alpha-\beta)\bigg(\frac{2\pi de}{t(d,e)^2}\bigg)^{\alpha+\beta} \bigg)\Phi\bigg(\frac{t}{T}\bigg)dt N^{s+u} \frac{ds}{s^{i+1}}\frac{du}{u^{j+1}} + O(\mathcal{E}).
\end{align}
The double sum over $d$ and $e$, which we call $\mathcal{S}$, requires a closer look. We will handle the first term of the $t$-integral, that is the term involving $\zeta(1+\alpha+\beta)$, as the second term will have the symmetries $\alpha \to -\beta$ and $\beta \to -\alpha$ as well as a premultiplication by $T^{-\alpha-\beta}$. 

We have reached the point where we need some structure on $\mathfrak{a}$ and $\mathfrak{b}$ that we may exploit to our advantage. 

\subsection{The linear case $d=1$}

To get a taste of things to come, let us first assume that
\[
\mathfrak{a}_n := \frac{(\mu \star \Lambda_1^{\star \ell_1})(n)}{\log^{\ell_1} N} = \frac{(\mu \star (\mu \star \log) \star \cdots \star (\mu \star \log))(n)}{\log^{\ell_1} N} = \frac{(\mu^{\star \ell_1 + 1} \star \log^{\star \ell_1})(n)}{\log^{\ell_1} N},
\]
as well as
\[
\mathfrak{b}_n := \frac{(\mu \star \Lambda_1^{\star \ell_2})(n)}{\log^{\ell_2} N} = \frac{(\mu \star (\mu \star \log) \star \cdots \star (\mu \star \log))(n)}{\log^{\ell_2} N} = \frac{(\mu^{\star \ell_2 + 1} \star \log^{\star \ell_2})(n)}{\log^{\ell_2} N}.
\]
This means that (leaving out the denominator $(\log N)^{\ell_1+\ell_2}$ for $I_{1}$)
\begin{align*}
\mathcal{S}_1 :&= \sumtwo_{1 \le d,e \le \infty}\frac{\mathfrak{a}_d \overline{\mathfrak{b}_e}}{[d,e]} \frac{(d,e)^{\alpha+\beta}}{d^{\alpha+s} e^{\beta+u}} = \sumtwo_{d,e }\frac{(\mu^{\star \ell_1 + 1} \star \log^{\star \ell_1})(d) (\mu^{\star \ell_2 + 1} \star \log^{\star \ell_2})(e)}{[d,e]} \frac{(d,e)^{\alpha+\beta}}{d^{\alpha+s} e^{\beta+u}}.
\end{align*}
Write $d=d_0 d_1 \cdots d_{\ell_1}$ and $e=e_0 e_1 \cdots e_{\ell_2}$ so that
\begin{align*}
\mathcal{S}_1 = \sumtwo_{\substack{d_0, d_1, \cdots, d_{\ell_1} \\ e_0, e_1, \cdots, e_{\ell_2}}}
\frac{\mu^{\star \ell_1+1}(d_0)\log d_1 \cdots \log d_{\ell_1}\mu^{\star \ell_2+1}(e_0)\log e_1 \cdots \log e_{\ell_2}}{[d_0 d_1 \cdots d_{\ell_1}, e_0 e_1 \cdots e_{\ell_2}]} \frac{(d_0 d_1 \cdots d_{\ell_1}, e_0 e_1 \cdots e_{\ell_2})^{\alpha+\beta}}{(d_0 d_1 \cdots d_{\ell_1})^{\alpha+s} (e_0 e_1 \cdots e_{\ell_2})^{\beta+u}}.
\end{align*}
We now employ the incredibly useful formula
\begin{align} \label{cauchyd1}
\log x = - \frac{\partial}{\partial \gamma} \frac{1}{x^\gamma} \bigg|_{\gamma=0} = - \frac{1}{2 \pi i} \oint \frac{1}{x^z}\frac{dz}{z^2},
\end{align}
where the contour of integration is a small circle around the origin. This leads us to
\begin{align*}
\mathcal{S}_1 &= (-1)^{\ell_1+\ell_2} \frac{1}{(2\pi i)^{\ell_1}}\ointdots \frac{1}{(2\pi i)^{\ell_2}}\ointdots \sumtwo_{\substack{d_0, d_1, d_2, \cdots, d_{\ell_1} \\ e_0, e_1, e_2, \cdots, e_{\ell_2}}}
\frac{\mu^{\star \ell_1+1}(d_0)\mu^{\star \ell_2+1}(e_0)}{[d_0 d_1 d_2 \cdots d_{\ell_1}, e_0 e_1 e_2 \cdots e_{\ell_2}]} \\
& \times \frac{(d_0 d_1 d_2 \cdots d_{\ell_1}, e_0 e_1 e_2 \cdots e_{\ell_2})^{\alpha+\beta}}{d_0^{\alpha+s} d_1^{\alpha+s+z_1} d_2^{\alpha+s+z_2} \cdots d_{\ell_1}^{\alpha+s+z_{\ell_1}} e_0^{\beta+u} e_1^{\beta+u+w_1} e_2^{\beta+u+w_2} \cdots e_{\ell_2}^{\beta+u+w_{\ell_2}}} \frac{dz_1}{z_1^2} \cdots \frac{dz_{\ell_1}}{z_{\ell_1}^2} \frac{dw_1}{w_1^2} \cdots \frac{dw_{\ell_2}}{w_{\ell_2}^2}.
\end{align*}
The advantage of this formula is that in the inner sum of $\mathcal{S}_1$ we now have products of multiplicative functions (the $\mu$'s) and completely multiplicative functions (the $d^x$'s and the $e^y$'s) instead of $\log$'s and $\Lambda$'s. The next step is to write this as an Euler product so that
\begin{align*}
\mathcal{S}_1 &= \frac{(-1)^{\ell_1}}{(2\pi i)^{\ell_1}}\ointdots \frac{(-1)^{\ell_2}}{(2\pi i)^{\ell_2}}\ointdots \prod_p \sumtwo_{\substack{p^{d_0}, p^{d_1}, p^{d_2}, \cdots, p^{d_{\ell_1}} \\ p^{e_0}, p^{e_1}, p^{e_2}, \cdots, p^{e_{\ell_2}}}}
\frac{\mu^{\star \ell_1+1}(p^{d_0})\mu^{\star \ell_2+1}(p^{e_0})}{[p^{d_0} p^{d_1} p^{d_2} \cdots p^{d_{\ell_1}}, p^{e_0} p^{e_1} p^{e_2} \cdots p^{e_{\ell_2}}]} \\
& \quad \times \frac{(p^{d_0} p^{d_1} p^{d_2} \cdots p^{d_{\ell_1}}, p^{e_0} p^{e_1} p^{e_2} \cdots p^{e_{\ell_2}})^{\alpha+\beta}}{(p^{d_0})^{\alpha+s} (p^{d_1})^{\alpha+s+z_1} \cdots (p^{d_{\ell_1}})^{\alpha+s+z_{\ell_1}} (p^{e_0})^{\beta+u} (p^{e_1})^{\beta+u+w_1} \cdots (p^{e_{\ell_2}})^{\beta+u+w_{\ell_2}}} \frac{dz_1}{z_1^2} \cdots \frac{dz_{\ell_1}}{z_{\ell_1}^2} \frac{dw_1}{w_1^2} \cdots \frac{dw_{\ell_2}}{w_{\ell_2}^2}.
\end{align*}
To recover the constant ($p^0$) and linear terms ($p^1$), we restrict the choices of the $d$'s and the $e$'s to the following nine possibilities. To enumerate them in a simple and fast manner, we shall denote the choices by the labels $\{d_0,d_i,e_0,e_j\}$ where $i=1,2,\cdots,\ell_1$ and $j=1,2,\cdots,\ell_2$.
\begin{enumerate}
\item[(1)] For $\{0,0,0,0\}$ we get 1 in the Euler product. This is our constant term.
\item[(2)] The case $\{1,0,0,0\}$ yields
\begin{align*}
%\frac{\mu^{\star \ell_1+1}(p^{d_0})}{[p^{d_0}, 1]} 
%\frac{(p^{d_0}, 1)^{\alpha+\beta}}{(p^{d_0})^{\alpha+s}}
%= 
\frac{\mu^{\star \ell_1+1}(p)}{[p, 1]} 
\frac{(p, 1)^{\alpha+\beta}}{p^{\alpha+s}}
= -\frac{\ell_1+1}{p^{1+\alpha+s}}.
\end{align*}
\item[(3)] The case $\{1,0,1,0\}$ yields
\begin{align*}
%\frac{\mu^{\star \ell_1+1}(p^{d_0})\mu^{\star \ell_2+1}(p^{e_0})}{[p^{d_0}, p^{e_0}]} 
%\frac{(p^{d_0}, p^{e_0})^{\alpha+\beta}}{(p^{d_0})^{\alpha+s}(p^{e_0})^{\beta+u}}
%= 
\frac{\mu^{\star \ell_1+1}(p)\mu^{\star \ell_2+1}(p)}{[p, p]} 
\frac{(p, p)^{\alpha+\beta}}{p^{\alpha+s}p^{\beta+u}}
= 
\frac{(\ell_1+1)(\ell_2+1)}{p^{1+s+u}}.
\end{align*}
\item[(4)] The case $\{1,0,0,1\}$ yields
\begin{align*}
%\frac{\mu^{\star \ell_1+1}(p^{d_0})}{[p^{d_0}, p^{e_j}]} 
%\frac{(p^{d_0}, p^{e_j})^{\alpha+\beta}}{(p^{d_0})^{\alpha+s} (p^{e_j})^{\beta+u+w_j}}
%=
\frac{\mu^{\star \ell_1+1}(p)}{[p, p]} 
\frac{(p, p)^{\alpha+\beta}}{p^{\alpha+s} p^{\beta+u+w_j}}
=
-\frac{\ell_1+1}{p^{1+s+u+w_j}}.
\end{align*}
\item[(5)] The case $\{0,1,0,0\}$ yields
\begin{align*}
%\frac{1}{[p^{d_i}, 1]} 
%\frac{(p^{d_i}, 1)^{\alpha+\beta}}{((p^{d_i})^{\alpha+s+z_i})}
%=
\frac{1}{[p, 1]} 
\frac{(p, 1)^{\alpha+\beta}}{p^{\alpha+s+z_i}}
=
\frac{1}{p^{1+\alpha+s+z_i}}.
\end{align*}
\item[(6)] By symmetry with $\{1,0,0,1\}$, the case $\{0,1,1,0\}$ yields
\begin{align*}
-\frac{\ell_2+1}{p^{1+s+u+z_i}}.
\end{align*}
\item[(7)] The case $\{0,1,0,1\}$ is the most difficult as it mixes the variables $z$ and $w$. We have
\begin{align*}
%\frac{1}{[p^{d_i}, p^{e_j}]} 
%\frac{(p^{d_i}, p^{e_j})^{\alpha+\beta}}{((p^{d_i})^{\alpha+s+z_i}) ((p^{e_j})^{\beta+u+w_j})}
%=
\frac{1}{[p, p]} 
\frac{(p, p)^{\alpha+\beta}}{p^{\alpha+s+z_i} p^{\beta+u+w_j}}
=
\frac{1}{p^{1+s+u+z_i+w_j}}.
\end{align*}
\item[(8)] By symmetry with $\{1,0,0,0\}$, the case $\{0,0,1,0\}$ yields
\begin{align*}
-\frac{\ell_2+1}{p^{1+\beta+u}}.
\end{align*}
\item[(9)] Lastly, the case $\{0,0,0,1\}$ is symmetric with respect to $\{0,1,0,0\}$ and hence we get
\begin{align*}
\frac{1}{p^{1+\beta+u+w_j}}.
\end{align*}
\end{enumerate}
If we now insert these terms into $\mathcal{S}_1$, then we arrive at
\begin{align*}
\mathcal{S}_1 = \frac{(-1)^{\ell_1}}{(2\pi i)^{\ell_1}} & \ointdots \frac{(-1)^{\ell_2}}{(2\pi i)^{\ell_2}}\ointdots \prod_p \bigg(1-\frac{\ell_1+1}{p^{1+\alpha+s}}+\frac{(\ell_1+1)(\ell_2+1)}{p^{1+s+u}} -\sum_{j=1}^{\ell_2} \frac{\ell_1+1}{p^{1+s+u+w_j}} \\
& + \sum_{i=1}^{\ell_1} \frac{1}{p^{1+\alpha+s+z_i}} - \sum_{i=1}^{\ell_1} \frac{\ell_2+1}{p^{1+s+u+z_i}} + \sum_{i=1}^{\ell_1} \sum_{j=1}^{\ell_2} \frac{1}{p^{1+s+u+z_i+w_j}} -\frac{\ell_2+1}{p^{1+\beta+u}} + \sum_{j=1}^{\ell_2} \frac{1}{p^{1+\beta+u+w_j}} \bigg)\\ 
& \quad \times A_{\alpha,\beta}(z_1,\cdots,z_{\ell_1},w_1,\cdots,w_{\ell_2};s,u) \frac{dz_1}{z_1^2} \cdots \frac{dz_{\ell_1}}{z_{\ell_1}^2} \frac{dw_1}{w_1^2} \cdots \frac{dw_{\ell_2}}{w_{\ell_2}^2}.
\end{align*}
We can write this more concisely with an autocorrelation-type ratio of zeta functions \cite{cfkrs, cfz, cs} as 
\begin{align} \label{autocorrelationratio}
\mathcal{S}_1 &= (-1)^{\ell_1}(-1)^{\ell_2}\frac{1}{(2\pi i)^{\ell_1}}\ointdots \frac{1}{(2\pi i)^{\ell_2}}\ointdots  \frac{\zeta(1+s+u)^{(\ell_1+1)(\ell_2+1)}}{\zeta(1+\alpha+s)^{\ell_1+1}\zeta(1+\beta+u)^{\ell_2+1}} \nonumber \\
& \quad  \times \frac{ (\prod_{i=1}^{\ell_1} \zeta(1+\alpha+s+z_i)) (\prod_{j=1}^{\ell_2} \zeta(1+\beta+u+w_j)) (\prod_{i=1}^{\ell_1} \prod_{j=1}^{\ell_2} \zeta(1+s+u+z_i+w_j))}{(\prod_{j=1}^{\ell_2} \zeta(1+s+u+w_j)^{\ell_1+1}) (\prod_{i=1}^{\ell_1} \zeta(1+s+u+z_i)^{\ell_2+1})}  \nonumber \\
& \quad  \times A_{\alpha,\beta}(z_1,\cdots,z_{\ell_1},w_1,\cdots,w_{\ell_2};s,u) \frac{dz_1}{z_1^2} \cdots \frac{dz_{\ell_1}}{z_{\ell_1}^2} \frac{dw_1}{w_1^2} \cdots \frac{dw_{\ell_2}}{w_{\ell_2}^2}.
\end{align}
Here $A = A_{\alpha,\beta}(z_1,\cdots,z_{\ell_1},w_1,\cdots,w_{\ell_2};s,u)$ is an arithmetical factor that is absolutely convergent in some half-plane containing the origin. Examination of the nine cases above indicates that it is given by
\begin{align*}
A &= \prod_p \bigg(1-\frac{1}{p^{1+s+u}}\bigg)^{(\ell_1+1)(\ell_2+1)} \bigg(1- \frac{1}{p^{1+\alpha+s}} \bigg)^{-(\ell_1+1)}\bigg(1- \frac{1}{p^{1+\beta+u}} \bigg)^{-(\ell_2+1)}  \\
& \quad  \times 
\prod_{i=1}^{\ell_1}\prod_{j=1}^{\ell_2} \bigg(1-\frac{1}{p^{1+s+u+z_i+w_j}}\bigg)
\prod_{j=1}^{\ell_2} \bigg(1-\frac{1}{p^{1+s+u+w_j}}\bigg)^{-(\ell_1+1)}
\prod_{i=1}^{\ell_1} \bigg(1-\frac{1}{p^{1+s+u+z_i}}\bigg)^{-(\ell_2+1)}
\\
&  \quad \times \prod_{i=1}^{\ell_1} \bigg(1-\frac{1}{p^{1+\alpha+s+z_i}} \bigg) \prod_{j=1}^{\ell_2} \bigg(1-\frac{1}{p^{1+\beta+u+w_j}} \bigg)  \bigg\{1-\frac{\ell_1+1}{p^{1+\alpha+s}}+\frac{(\ell_1+1)(\ell_2+1)}{p^{1+s+u}} -\sum_{j=1}^{\ell_2} \frac{\ell_1+1}{p^{1+s+u+w_j}} \\
& + \sum_{i=1}^{\ell_1} \frac{1}{p^{1+\alpha+s+z_i}} - \sum_{i=1}^{\ell_1} \frac{\ell_2+1}{p^{1+s+u+z_i}} + \sum_{i=1}^{\ell_1} \sum_{j=1}^{\ell_2} \frac{1}{p^{1+s+u+z_i+w_j}} -\frac{\ell_2+1}{p^{1+\beta+u}} + \sum_{j=1}^{\ell_2} \frac{1}{p^{1+\beta+u+w_j}} \bigg\}.
\end{align*}

We may now go back to $I_1$, the first half of $I$, and tidy up a bit so that
\begin{align} \label{I1clean}
I_1(\alpha,\beta) &= T\widehat{\Phi}(0)\sumtwo_{i,j} \frac{p_{\ell_1,i}p_{\ell_2,j}i!j!}{\log^{i+j}N}\frac{1}{(2\pi i)^2}\int_{(1)}\int_{(1)} \frac{(-1)^{\ell_1}(-1)^{\ell_2}}{\log^{\ell_1+\ell_2} N} \nonumber \\
& \quad  \times \frac{1}{(2\pi i)^{\ell_1}}\ointdots \frac{1}{(2\pi i)^{\ell_2}}\ointdots  \frac{\zeta(1+s+u)^{(\ell_1+1)(\ell_2+1)} \zeta(1+\alpha+\beta) }{\zeta(1+\alpha+s)^{\ell_1+1}\zeta(1+\beta+u)^{\ell_2+1}} \nonumber \\
& \quad  \times  \bigg(\prod_{i=1}^{\ell_1} \prod_{j=1}^{\ell_2} \zeta(1+s+u+z_i+w_j) \bigg) \bigg(\prod_{i=1}^{\ell_1} \frac{\zeta(1+\alpha+s+z_i)}{\zeta(1+s+u+z_i)^{\ell_2+1}}\bigg) \bigg(\prod_{j=1}^{\ell_2} \frac{\zeta(1+\beta+u+w_j)}{\zeta(1+s+u+w_j)^{\ell_1+1}}\bigg) \nonumber \\
& \quad  \times A_{\alpha,\beta}(z_1,\cdots,z_{\ell_1},w_1,\cdots,w_{\ell_2};s,u) \frac{dz_1}{z_1^2} \cdots \frac{dz_{\ell_1}}{z_{\ell_1}^2} \frac{dw_1}{w_1^2} \cdots \frac{dw_{\ell_2}}{w_{\ell_2}^2} N^{s+u} \frac{ds}{s^{i+1}}\frac{du}{u^{j+1}} + O(\mathcal{E}_3),
\end{align}
where $\mathcal{E}_3$ is the third case of Theorem \ref{meanvalueintegral} and where
\begin{align} \label{Arithemticclean}
A_{\alpha,\beta}(\mathbf{z},\mathbf{w};s,u) = \prod_p &\bigg\{ \frac{(1-\tfrac{1}{p^{1+s+u}})^{(\ell_1+1)(\ell_2+1)}}{(1- \tfrac{1}{p^{1+\alpha+s}})^{(\ell_1+1)}(1- \tfrac{1}{p^{1+\beta+u}})^{(\ell_2+1)}}  \bigg( \prod_{i=1}^{\ell_1}\prod_{j=1}^{\ell_2} \bigg( 1-\frac{1}{p^{1+s+u+z_i+w_j}}\bigg) \bigg) \nonumber \\
& \times \bigg(\prod_{i=1}^{\ell_1} \frac{1-\tfrac{1}{p^{1+\alpha+s+z_i}}}{(1-\tfrac{1}{p^{1+s+u+z_i}})^{(\ell_2+1)}}\bigg) \bigg(\prod_{j=1}^{\ell_2} \frac{1-\tfrac{1}{p^{1+\beta+u+w_j}}}{(1-\tfrac{1}{p^{1+s+u+w_j}})^{(\ell_1+1)}} \bigg) \nonumber \\
& \times \bigg[ 1-\frac{\ell_1+1}{p^{1+\alpha+s}} -\frac{\ell_2+1}{p^{1+\beta+u}} + \frac{(\ell_1+1)(\ell_2+1)}{p^{1+s+u}} -\sum_{j=1}^{\ell_2} \frac{\ell_1+1}{p^{1+s+u+w_j}} - \sum_{i=1}^{\ell_1} \frac{\ell_2+1}{p^{1+s+u+z_i}} \nonumber \\
&  \quad \quad \, + \sum_{i=1}^{\ell_1} \frac{1}{p^{1+\alpha+s+z_i}} + \sum_{j=1}^{\ell_2} \frac{1}{p^{1+\beta+u+w_j}} + \sum_{i=1}^{\ell_1} \sum_{j=1}^{\ell_2} \frac{1}{p^{1+s+u+z_i+w_j}} \bigg] \bigg\}.
\end{align}

Let us define the integrand to be
\begin{align*}
\mathfrak{M}_{\alpha,\beta}(\mathbf{z},\mathbf{w};s,u) &= \frac{\zeta(1+s+u)^{(\ell_1+1)(\ell_2+1)} \zeta(1+\alpha+\beta) }{\zeta(1+\alpha+s)^{\ell_1+1}\zeta(1+\beta+u)^{\ell_2+1}}A_{\alpha,\beta}(z_1,\cdots,z_{\ell_1},w_1,\cdots,w_{\ell_2};s,u)\\
& \quad  \times \bigg(\prod_{i=1}^{\ell_1} \prod_{j=1}^{\ell_2} \zeta(1+s+u+z_i+w_j) \bigg) \bigg(\prod_{i=1}^{\ell_1} \frac{\zeta(1+\alpha+s+z_i)}{\zeta(1+s+u+z_i)^{\ell_2+1}}\bigg) \bigg(\prod_{j=1}^{\ell_2} \frac{\zeta(1+\beta+u+w_j)}{\zeta(1+s+u+w_j)^{\ell_1+1}}\bigg).
\end{align*}
Our first observation is that
\begin{align*}
\mathfrak{M}_{\alpha,\beta}(\mathbf{z},\mathbf{w};\beta,\alpha) &= \zeta(1+\alpha+\beta)^{\ell_1 \ell_2}A_{\alpha,\beta}(z_1,\cdots,z_{\ell_1},w_1,\cdots,w_{\ell_2};\beta,\alpha)\\
& \quad  \times \bigg(\prod_{i=1}^{\ell_1} \prod_{j=1}^{\ell_2} \zeta(1+\alpha+\beta+z_i+w_j) \bigg) \bigg(\prod_{i=1}^{\ell_1} \frac{1}{\zeta(1+\alpha+\beta+z_i)^{\ell_2}}\bigg) \bigg(\prod_{j=1}^{\ell_2} \frac{1}{\zeta(1+\alpha+\beta+w_j)^{\ell_1}}\bigg),
\end{align*}
and hence at $\mathbf{z}=\mathbf{w}=\mathbf{0}$ we end up with
\begin{align*}
\mathfrak{M}_{\alpha,\beta}(\mathbf{0},\mathbf{0};\beta,\alpha) &= \zeta(1+\alpha+\beta)^{\ell_1 \ell_2}A_{\alpha,\beta}(0,\cdots,0,0,\cdots,0;\beta,\alpha)\\
&  \quad \times \bigg(\prod_{i=1}^{\ell_1} \prod_{j=1}^{\ell_2} \zeta(1+\alpha+\beta) \bigg) \bigg(\prod_{i=1}^{\ell_1} \frac{1}{\zeta(1+\alpha+\beta)^{\ell_2}}\bigg) \bigg(\prod_{j=1}^{\ell_2} \frac{1}{\zeta(1+\alpha+\beta)^{\ell_1}}\bigg) \\
&= A_{\alpha,\beta}(\mathbf{0},\mathbf{0};\beta,\alpha).
\end{align*}
On the other hand, for the arithmetical factor $A$ we find that
\begin{align*}
A_{\alpha,\beta}(\mathbf{z},\mathbf{w};\beta,\alpha) = \prod_p &\bigg\{\bigg(1-\frac{1}{p^{1+\alpha+\beta}}\bigg)^{\ell_1 \ell_2 - 1} \bigg( \prod_{i=1}^{\ell_1}\prod_{j=1}^{\ell_2} \bigg( 1-\frac{1}{p^{1+\alpha+\beta+z_i+w_j}}\bigg) \bigg) \\
& \times \bigg(\prod_{i=1}^{\ell_1} \frac{1}{(1-\tfrac{1}{p^{1+\alpha+\beta+z_i}})^{\ell_2}}\bigg) \bigg(\prod_{j=1}^{\ell_2} \frac{1}{(1-\tfrac{1}{p^{1+\alpha+\beta+w_j}})^{\ell_1}} \bigg)  \\
& \times \bigg[ 1-\frac{\ell_1+1}{p^{1+\alpha+\beta}} -\frac{\ell_2+1}{p^{1+\alpha+\beta}} + \frac{(\ell_1+1)(\ell_2+1)}{p^{1+\alpha+\beta}} -\sum_{j=1}^{\ell_2} \frac{\ell_1+1}{p^{1+\alpha+\beta+w_j}} - \sum_{i=1}^{\ell_1} \frac{\ell_2+1}{p^{1+\alpha+\beta+z_i}} \\
&  \quad \quad \, + \sum_{i=1}^{\ell_1} \frac{1}{p^{1+\alpha+\beta+z_i}} + \sum_{j=1}^{\ell_2} \frac{1}{p^{1+\alpha+\beta+w_j}} + \sum_{i=1}^{\ell_1} \sum_{j=1}^{\ell_2} \frac{1}{p^{1+\alpha+\beta+z_i+w_j}} \bigg] \bigg\},
\end{align*}
and hence at $\mathbf{z}=\mathbf{w}=\mathbf{0}$ we arrive at
\begin{align} \label{ArithmeticAt0}
A_{\alpha,\beta}(\mathbf{0},\mathbf{0};\beta,\alpha) = \prod_p &\bigg\{\bigg(1-\frac{1}{p^{1+\alpha+\beta}}\bigg)^{-1} \bigg[ 1 -\frac{1}{p^{1+\alpha+\beta}} \bigg] \bigg\} = 1.
\end{align}
The property that $\mathfrak{M}_{\alpha,\beta}(\mathbf{0},\mathbf{0};\beta,\alpha) = A_{\alpha,\beta}(\mathbf{0},\mathbf{0};\beta,\alpha) = 1$ will become useful shortly.

We will find it expedient to examine the logarithm of $A$ in order to turn the products into sums. One has
\begin{align} \label{logderivativeA}
\log A_{\alpha,\beta}(\mathbf{z},\mathbf{w};s,u) = \sum_p &\bigg\{ \log \frac{(1-\tfrac{1}{p^{1+s+u}})^{(\ell_1+1)(\ell_2+1)}}{(1- \tfrac{1}{p^{1+\alpha+s}})^{(\ell_1+1)}(1- \tfrac{1}{p^{1+\beta+u}})^{(\ell_2+1)}} + \sum_{i=1}^{\ell_1}\sum_{j=1}^{\ell_2} \log \bigg( 1-\frac{1}{p^{1+s+u+z_i+w_j}}\bigg) \nonumber \\
& + \sum_{i=1}^{\ell_1} \log \frac{1-\tfrac{1}{p^{1+\alpha+s+z_i}}}{(1-\tfrac{1}{p^{1+s+u+z_i}})^{(\ell_2+1)}} + \sum_{j=1}^{\ell_2} \log \frac{1-\tfrac{1}{p^{1+\beta+u+w_j}}}{(1-\tfrac{1}{p^{1+s+u+w_j}})^{(\ell_1+1)}} \nonumber \\
& + \log \bigg[ 1-\frac{\ell_1+1}{p^{1+\alpha+s}} -\frac{\ell_2+1}{p^{1+\beta+u}} + \frac{(\ell_1+1)(\ell_2+1)}{p^{1+s+u}} -\sum_{j=1}^{\ell_2} \frac{\ell_1+1}{p^{1+s+u+w_j}} - \sum_{i=1}^{\ell_1} \frac{\ell_2+1}{p^{1+s+u+z_i}} \nonumber \\
&  \quad \quad \, + \sum_{i=1}^{\ell_1} \frac{1}{p^{1+\alpha+s+z_i}} + \sum_{j=1}^{\ell_2} \frac{1}{p^{1+\beta+u+w_j}} + \sum_{i=1}^{\ell_1} \sum_{j=1}^{\ell_2} \frac{1}{p^{1+s+u+z_i+w_j}} \bigg] \bigg\}.
\end{align}

For all cosmetic purposes, \eqref{I1clean} combined with \eqref{Arithemticclean} and \eqref{ArithmeticAt0} and supplemented with \eqref{logderivativeA} is as far as we can go before things take an ugly turn due to the combinatorics involved in this situation. Indeed, this is the price to pay for having used \eqref{cauchyd1}. However, it is customary \cite{bcy, krz02, rrz01, youngshort} to decouple the complex variables $s$ and $u$ and then perform the sums over the indices $i$ and $j$ to give cleaner terms.  \\

We will start with $\ell_1 = \ell_2 =1$, and $\ell_1 = \ell_2 =2$ in $d=1$, then move on to $d=2$.

\subsection{Choosing the truncation $\ell$}
\subsubsection{The case $\ell_1=\ell_2=1$}
In order to keep things readable, we shall specialize to $\ell_1 = \ell_2 = 1$ and then delineate the path to general $\ell_1$ and $\ell_2$. In this simpler case, a direct residue calculus computation shows that
\begin{align*}
\frac{1}{(2\pi i)^2}  \oint\oint & \mathfrak{M}_{\alpha,\beta}(z_1,w_1;s,u) \frac{dz_1}{z_1^2} \frac{dw_1}{w_1^2} = \frac{\zeta(1+s+u)\zeta(1+\alpha+\beta)}{\zeta(1+\alpha+s)\zeta(1+\beta+u)} \\
& \times \bigg[ A_{\alpha,\beta}(0,0;s,u)\bigg\{\frac{\zeta''}{\zeta}(1+s+u) + \frac{\zeta'}{\zeta}(1+\alpha+s)\frac{\zeta'}{\zeta}(1+\beta+u) \\
& \quad \quad \quad \quad - \frac{\zeta'}{\zeta}(1+s+u)\frac{\zeta'}{\zeta}(1+\beta+u) - \frac{\zeta'}{\zeta}(1+s+u)\frac{\zeta'}{\zeta}(1+\alpha+s) \bigg\} \\
& \; + A_{\alpha,\beta}^{(1,0)}(0,0;s,u) \bigg\{\frac{\zeta'}{\zeta}(1+\beta+u) - \frac{\zeta'}{\zeta}(1+s+u) \bigg\} \\
& \; + A_{\alpha,\beta}^{(0,1)}(0,0;s,u) \bigg\{\frac{\zeta'}{\zeta}(1+\alpha+s) - \frac{\zeta'}{\zeta}(1+s+u) \bigg\} + A_{\alpha,\beta}^{(1,1)}(0,0;s,u) \bigg].
\end{align*}
Here the notation $f^{(a,b)}(x,y)$ indicates the $a$-th derivative with respect to $x$ and the $b$-th derivative with respect to $y$. Therefore
\begin{align*}
I_1(\alpha,\beta) &= \frac{T\widehat{\Phi}(0)}{\log^{2} N}\sumtwo_{i,j} \frac{p_{\ell_1,i}p_{\ell_2,j}i!j!}{(\log N)^{i+j}} J_{1} + O(\mathcal{E}_3),
\end{align*}
where
\begin{align*}
J_{1} = \frac{1}{(2\pi i)^2} \int_{(\delta)} & \int_{(\delta)} \frac{\zeta(1+s+u)\zeta(1+\alpha+\beta)}{\zeta(1+\alpha+s)\zeta(1+\beta+u)} \\
& \times \bigg[ A_{\alpha,\beta}(0,0;s,u)\bigg\{\frac{\zeta''}{\zeta}(1+s+u) + \frac{\zeta'}{\zeta}(1+\alpha+s)\frac{\zeta'}{\zeta}(1+\beta+u) \\
& \quad \quad \quad \quad - \frac{\zeta'}{\zeta}(1+s+u)\frac{\zeta'}{\zeta}(1+\beta+u) - \frac{\zeta'}{\zeta}(1+s+u)\frac{\zeta'}{\zeta}(1+\alpha+s) \bigg\} \\
& \; + A_{\alpha,\beta}^{(1,0)}(0,0;s,u) \bigg\{\frac{\zeta'}{\zeta}(1+\beta+u) - \frac{\zeta'}{\zeta}(1+s+u) \bigg\} \\
& \; + A_{\alpha,\beta}^{(0,1)}(0,0;s,u) \bigg\{\frac{\zeta'}{\zeta}(1+\alpha+s) - \frac{\zeta'}{\zeta}(1+s+u) \bigg\} + A_{\alpha,\beta}^{(1,1)}(0,0;s,u) \bigg] N^{s+u} \frac{ds}{s^{i+1}}\frac{du}{u^{j+1}},
\end{align*}
by deforming the path of integration of the $s,u$-integrals to $\real(s)=\real(u)=\delta$ with $\delta>0$. Next use Dirichlet series to see that
\begin{align*}
J_{1} = \frac{1}{(2\pi i)^2} \int_{(\delta)} & \int_{(\delta)} \frac{\zeta(1+\alpha+\beta)}{\zeta(1+\alpha+s)\zeta(1+\beta+u)} \\
& \times \bigg[ A_{\alpha,\beta}(0,0;s,u)\sum_{n \le N} \bigg\{\frac{(\mathbf{1} \star \Lambda_2)(n)}{n^{1+s+u}} + \frac{1}{n^{1+s+u}}\frac{\zeta'}{\zeta}(1+\alpha+s)\frac{\zeta'}{\zeta}(1+\beta+u) \\
& \quad \quad \quad \quad + \frac{(\mathbf{1} \star \Lambda)(n)}{n^{1+s+u}} \frac{\zeta'}{\zeta}(1+\beta+u) + \frac{(\mathbf{1} \star \Lambda)(n)}{n^{1+s+u}}\frac{\zeta'}{\zeta}(1+\alpha+s) \bigg\} \\
& \; + \zeta(1+s+u) A_{\alpha,\beta}^{(1,0)}(0,0;s,u) \bigg\{\frac{\zeta'}{\zeta}(1+\beta+u) - \frac{\zeta'}{\zeta}(1+s+u) \bigg\} \\
& \; + \zeta(1+s+u) A_{\alpha,\beta}^{(0,1)}(0,0;s,u) \bigg\{\frac{\zeta'}{\zeta}(1+\alpha+s) - \frac{\zeta'}{\zeta}(1+s+u) \bigg\} \\ 
& \; + A_{\alpha,\beta}^{(1,1)}(0,0;s,u)\sum_{n \le N} \frac{1}{n^{1+s+u}}  \bigg] N^{s+u} \frac{ds}{s^{i+1}}\frac{du}{u^{j+1}}.
\end{align*}
We have used the Dirichlet convolution
\[
\zeta(1+s+u) \frac{\zeta''}{\zeta}(1+s+u) = \sum_{n=1}^\infty \frac{(\mathbf{1} \star \Lambda_2)(n)}{n^{1+s+u}}, \quad \real(s+u)>0, \quad \textnormal{etc}.
\]
Now we take $\delta \asymp L^{-1}$ and bound integrals trivially to get $J_{1} \ll L^{i+j+2}$. We then use a Taylor expansion so that $A_{\alpha,\beta}^{(m,n)}(0,0;s,u) = A_{\alpha,\beta}^{(m,n)}(0,0;\beta,\alpha) + O(|s-\beta|+|u-\alpha|)$ for $m,n \in \{0,1\}$. Each factor of $s$ and $u$ saves a factor of $\log T$ and $\alpha,\beta \ll 1/\log T$. We know that $ A_{\alpha,\beta}^{(0,0)}(0,0;\beta,\alpha) = 1$ and we use \eqref{logderivativeA} to find the other three values that we are missing in the expression for $J_{1}$. 

By logarithmic differentiation we have that
\begin{align*}
\frac{\partial}{\partial z_1} \log A_{\alpha,\beta}(z_1,w_1;s,u) = \frac{A_{\alpha,\beta}^{(1,0)}(z_1,w_1;s,u)}{A_{\alpha,\beta}(z_1,w_1;s,u)}.
\end{align*}
We concentrate on the left-hand side and observe from \eqref{logderivativeA} that
\begin{align*}
\frac{\partial}{\partial z_1} \log A_{\alpha,\beta}(z_1,w_1;s,u) \bigg|_{w_1=z_1=0} = \sum_p \bigg(&-\frac{\log p}{-1+p^{1+s+u}} \\
&+ \frac{p^{\alpha+s}(-1+p^{1+\alpha+\beta})\log p}{(-1+p^{1+\alpha+s})(p^{\alpha+\beta}-p^{\alpha+s}+p^u(-p^\beta+p^{1+\alpha+\beta+s}))}\bigg).
%\frac{\partial}{\partial z_1} \log A_{\alpha,\beta}(z_1,w_1;s,u) \bigg|_{w_1=z_1=0} & = \sum_p \bigg( \frac{\log p}{1-p^{1+s+u}} %- \frac{\log p}{1-p^{1+\alpha+s}} \\
%& \quad \quad + \frac{\log p}{p^{1+s+u}-p^{u-\beta}-p^{s-\alpha}+1} - \frac{\log p}{p^{1+\alpha+s}-p^{\alpha+s-\alpha-u}-p^{\beta-u}+1}\bigg).
\end{align*}
Consequently
\begin{align*}
\frac{\partial}{\partial z_1} \log A_{\alpha,\beta}(z_1,w_1;\beta,\alpha) \bigg|_{w_1=z_1=0} & = \sum_p \bigg(-\frac{\log p}{-1+p^{1+\beta+\alpha}} + \frac{\log p}{-1+p^{1+\alpha+\beta}}\bigg) = 0.
\end{align*}
Therefore
\begin{align*}
A_{\alpha,\beta}^{(1,0)}(0,0;\beta,\alpha) = A_{\alpha,\beta}(0,0;\beta,\alpha)\frac{\partial}{\partial z_1} \log A_{\alpha,\beta}(z_1,w_1;\beta,\alpha) \bigg|_{w_1=z_1=0} = 1 \times 0 = 0.
\end{align*}
By symmetry one also gets that $A_{\alpha,\beta}^{(0,1)}(0,0;\beta,\alpha) = 0$. We next move on to the derivative with respect to $z_1$ and with respect to $w_1$. Unfortunately, the trick with the logarithmic derivative will not work as cleanly as it has hitherto. To find $A_{\alpha,\beta}^{(1,1)}$ we make use of Fa\`{a} di Bruno's formula
\begin{align} \label{faadibruno}
\frac{\partial}{\partial z_1}\cdots\frac{\partial}{\partial z_{\ell_1}}\frac{\partial}{\partial z_{\ell_1+1}}\cdots\frac{\partial}{\partial z_{\ell_2}} \log A_{\alpha,\beta} (\mathbf{z};s,u) = \sum_{\pi \in \Pi(\ell_1+\ell_2)} \frac{(-1)^{\pi-1}(\pi-1)!}{A_{\alpha,\beta} (\mathbf{z};s,u)^\pi} \prod_{B \in \pi} \frac{\partial^{|B|}}{\prod_{k \in B}\partial z_k}A_{\alpha,\beta} (\mathbf{z};s,u),
\end{align} 
where we have found it useful to temporarily consolidate the notation from $\mathbf{z}=(z_1, \cdots, z_{\ell_1})$ and $\mathbf{w} = (w_1, \cdots, w_{\ell_2})$ to simply $\mathbf{z}=(z_1, \cdots, z_{\ell_1}, z_{\ell_1+1}, \cdots, z_{\ell_2})$. Here $\pi$ runs through the set $\Pi(\ell_1+\ell_2)$ of all partitions of $\{1,2,\cdots,\ell_1+\ell_2\}$ and $B \in \pi$ indicates that the variable $B$ runs through the list of all blocks of the partition $\pi$. For $\ell_1 = \ell_2 = 1$, the left-hand side of \eqref{faadibruno} is given by the expression
\begin{align*}
\frac{\partial^2}{\partial z_1 \partial w_1} \log A_{\alpha,\beta}(z_1,w_1;s,u) \bigg|_{w_1=z_1=0} & = \sum_p \bigg( -\frac{p^{1+s+u}\log^2 p}{(p^{1+s+u}-1)^2} + \frac{p^{\alpha+\alpha+s+u}(p^{1+\alpha+\beta}-1)\log^2 p}{(p^{\alpha+\beta}-p^{-\alpha+s}+p^{\beta+u}+p^{1+\alpha+\alpha+s+u})^2} \bigg).
\end{align*}
Consequently 
\begin{align*}
\frac{\partial^2}{\partial z_1 \partial w_1} \log A_{\alpha,\beta}(z_1,w_1;\beta,\alpha) \bigg|_{w_1=z_1=0} & = \sum_p \bigg( -\frac{p^{1+\alpha+\beta}\log^2 p}{(p^{1+\alpha+\beta}-1)^2} + \frac{\log^2 p}{p^{1+\alpha+\beta}-1} \bigg) = - \sum_p \frac{\log^2 p}{(p^{1+\alpha+\beta}-1)^2}. 
\end{align*}
For the right-hand side of \eqref{faadibruno} we get
\begin{align*}
& \frac{1}{A_{\alpha,\beta}(z_1,w_1;s,u)} \frac{\partial^2 A_{\alpha,\beta}(z_1,w_1;s,u)}{\partial z_1 \partial w_1}  - \frac{1}{A_{\alpha,\beta}(z_1,w_1;s,u)^2} \bigg(\frac{\partial A_{\alpha,\beta}(z_1,w_1;s,u)}{\partial z_1}\frac{\partial A_{\alpha,\beta}(z_1,w_1;s,u)}{\partial w_1}\bigg).
\end{align*}
Since $A_{\alpha,\beta}(0,0;\beta,\alpha)=1$ and $A_{\alpha,\beta}^{(1,0)}(0,0;\beta,\alpha) = A_{\alpha,\beta}^{(0,1)}(0,0;\beta,\alpha) = 0$, we finally see that at $w_1 = z_1 = 0$ and $s=\beta$ and $u=\alpha$ the above expression reduces to
\begin{align*}
& \frac{\partial^2 A_{\alpha,\beta}(z_1,w_1;\beta,\alpha)}{\partial z_1 \partial w_1}\bigg|_{z_1=w_1=0} = A_{\alpha,\beta}^{(1,1)}(0,0;\beta,\alpha).
\end{align*}
Hence comparing the left- and right-hand sides of \eqref{faadibruno} we obtain
\begin{align*}
A_{\alpha,\beta}^{(1,1)}(0,0;\beta,\alpha) = - \sum_p \bigg(\frac{\log p}{p^{1+\alpha+\beta}-1} \bigg)^2.
\end{align*}
Note that this is the term appearing in \cite[Theorem 2.5]{cs}. We remark that for $|\alpha + \beta| \le \varepsilon < \tfrac{1}{4}$ we get
\[
|A_{\alpha,\beta}^{(1,1)}(0,0;\beta,\alpha)| < \zeta(2) \quad \textnormal{and} \quad A_{0,0}^{(1,1)}(0,0;0,0) \approx 1.385603705.
\]

We can assemble our findings back in $J_{1}$ to obtain
\begin{align*}
J_{1} = \zeta(1+\alpha+\beta)\frac{1}{(2\pi i)^2} & \int_{(\delta)} \int_{(\delta)} \frac{1}{\zeta(1+\alpha+s)\zeta(1+\beta+u)} \\
& \times \bigg[ A_{\alpha,\beta}(0,0;\beta,\alpha)\sum_{n \le N} \bigg\{\frac{(\mathbf{1} \star \Lambda_2)(n)}{n^{1+s+u}} + \frac{1}{n^{1+s+u}}\frac{\zeta'}{\zeta}(1+\alpha+s)\frac{\zeta'}{\zeta}(1+\beta+u) \\
& \quad \quad \quad \quad + \frac{(\mathbf{1} \star \Lambda)(n)}{n^{1+s+u}} \frac{\zeta'}{\zeta}(1+\beta+u) + \frac{(\mathbf{1} \star \Lambda)(n)}{n^{1+s+u}}\frac{\zeta'}{\zeta}(1+\alpha+s) \bigg\} \\
& \quad \quad \quad \quad + A_{\alpha,\beta}^{(1,1)}(0,0;\beta,\alpha)\sum_{n \le N}\frac{1}{n^{1+s+u}}\bigg] N^{s+u} \frac{ds}{s^{i+1}}\frac{du}{u^{j+1}} \\
& =: J_{1,1}+J_{1,2}+J_{1,3}+J_{1,4}+J_{1,5}.
\end{align*}
We quickly explain how to get the main terms in rough strokes before moving on the general idea. For $J_{1,1}$ we may now write
\[
J_{1,1} = \frac{1}{\alpha+\beta}\sum_{n \le N} \frac{(\mathbf{1} \star \Lambda_2)(n)}{n} L_{1,1} L_{1,2},
\]
where
\[
L_{1,1} = \frac{1}{2 \pi i} \int_{\asymp (L^{-1})} \bigg(\frac{N}{n}\bigg)^s  \frac{1}{\zeta(1+\alpha+s)}\frac{ds}{s^{i+1}} \quad \textnormal{and} \quad L_{1,2} = \frac{1}{2 \pi i} \int_{\asymp (L^{-1})} \bigg(\frac{N}{n}\bigg)^u  \frac{1}{\zeta(1+\beta+u)}\frac{du}{u^{j+1}}.
\]
The critical thing to notice is that the variables $s$ and $u$ have been separated. These integrals are special cases of Lemma \ref{contourlemma} that will proved later in the wider narrative of the general case. This will require a careful use of the standard zero-free region of the Riemann zeta-function and the Laurent series around $s=0$
\[
\zeta(1+s) = \frac{1}{s}+C_0+O(s),
\]
where $C_0$ is the Euler constant. At any rate, these integrals are seen to be equal to \cite[p. 546]{youngshort}
\[
J_{1,1} = \frac{1}{\alpha+\beta}\sum_{n \le N} \frac{(\mathbf{1} \star \Lambda_2)(n)}{n} \frac{1}{(2\pi i)^2} \oint\oint \bigg(\frac{N}{n}\bigg)^s  (\alpha+s)
 \bigg(\frac{N}{n}\bigg)^u  (\beta+u)
\frac{ds}{s^{i+1}}  \frac{du}{u^{j+1}} + O(L^{i+j}),
\]
where the contours are small circles of radii one centered around the origin. Some re-arrangements lead to
\begin{align*}
J_{1,1} &= \frac{d^2}{dxdy}e^{\alpha x + \beta y}\frac{1}{\alpha+\beta}\sum_{n \le N} \frac{(\mathbf{1} \star \Lambda_2)(n)}{n} \frac{1}{(2\pi i)^2} \oint\oint \bigg(\frac{N}{n}e^x\bigg)^s 
 \bigg(\frac{N}{n}e^y\bigg)^u 
\frac{ds}{s^{i+1}}  \frac{du}{u^{j+1}} + O(L^{i+j}) \bigg|_{x=y=0} \\
&= \frac{1}{i!j!}\frac{d^2}{dxdy}e^{\alpha x + \beta y}\frac{1}{\alpha+\beta}\sum_{n \le N} \frac{(\mathbf{1} \star \Lambda_2)(n)}{n} \bigg(x + \log \frac{N}{n}\bigg)^i \bigg(y + \log \frac{N}{n}\bigg)^j + O(L^{i+j}) \bigg|_{x=y=0} \\
&= \frac{1}{\alpha+\beta}\frac{1}{i!j!}\frac{(\log N)^{i+j}}{\log^2 N}\frac{d^2}{dxdy}N^{\alpha x + \beta y}\sum_{n \le N} \frac{(\mathbf{1} \star \Lambda_2)(n)}{n}  \bigg( x + \frac{\log (N/n)}{\log N}\bigg)^i \bigg(y + \frac{\log(N/n)}{\log N}\bigg)^j \bigg|_{x=y=0} \\
&  \quad + O(L^{i+j}),
\end{align*}
by the change of variables $x \to x \log N$ and $y \to y \log N$. Going back to $I_1$ we can perform the sums over $i$ and $j$ so that
\begin{align*}
I_{1,1}(\alpha,\beta) &= \frac{T\widehat{\Phi}(0)}{\log^{4} N}\frac{1}{\alpha+\beta} \frac{d^2}{dxdy}N^{\alpha x + \beta y}\sum_{n \le N} \frac{(\mathbf{1} \star \Lambda_2)(n)}{n} \\
& \quad \times \sumtwo_{i,j} p_{\ell_1,i}p_{\ell_2,j} \bigg( x + \frac{\log (N/n)}{\log N}\bigg)^i \bigg(y + \frac{\log(N/n)}{\log N}\bigg)^j  \bigg|_{x=y=0} + O(T/L) \\
&= \frac{T\widehat{\Phi}(0)}{\log^{4} N}\frac{1}{\alpha+\beta} \frac{d^2}{dxdy}N^{\alpha x + \beta y}\sum_{n \le N} \frac{(\mathbf{1} \star \Lambda_2)(n)}{n} P_{1} \bigg( x + \frac{\log (N/n)}{\log N}\bigg) P_{2} \bigg(y + \frac{\log(N/n)}{\log N}\bigg) \bigg|_{x=y=0} \\
& \quad + O(T/L).
\end{align*}
Using Euler-Maclaurin summation (see \cite[Corollary 4.5]{bcy}, \cite[Lemma 2.7]{krz01} and \cite[Lemma 3.6]{rrz01}) the above can be turned into
\begin{align*}
I_{1,1}(\alpha,\beta) = \frac{T\widehat{\Phi}(0)}{\log N}\frac{1}{\alpha+\beta}\frac{d^2}{dxdy}N^{\alpha x + \beta y} \int_0^{1} (1-u)^2 P_{1} (x+u) P_{2} (y+u)du \bigg|_{x=y=0} + O(T/L).
\end{align*}

Next, we look at $J_{1,2}$. The procedure parallels the one above. We see that
\begin{align*}
J_{1,2} = \frac{1}{\alpha+\beta}\sum_{n \le N} \frac{1}{n} L_{2,1}L_{2,2},
\end{align*}
where
\begin{align*}
L_{2,1} = \frac{1}{2 \pi i} \int_{\asymp (L^{-1})} \bigg(\frac{N}{n}\bigg)^s \frac{1}{\zeta(1+\alpha+s)}\frac{\zeta'}{\zeta}(1+\alpha+s) \frac{ds}{s^{i+1}},
\end{align*} 
as well as
\begin{align*}
\quad L_{2,2} = \frac{1}{2 \pi i} \int_{\asymp (L^{-1})} \bigg(\frac{N}{n}\bigg)^u \frac{1}{\zeta(1+\beta+u)}\frac{\zeta'}{\zeta}(1+\beta+u) \frac{du}{u^{j+1}}.
\end{align*}
%Here the argument becomes more subtle, but it essentially necessitates the standard zero-free region of $\zeta$ as well as 
Now use the standard zero-free region and bounds of $\zeta$, as well as the Laurent series around $s=0$
\[
\frac{\zeta'}{\zeta}(1+s) = -\frac{1}{s} + C_0 + O(s),
\]
(more on this later in Lemma 7.1) so that
\begin{align*}
J_{1,2}(\alpha,\beta) &= \frac{1}{\alpha+\beta}\sum_{n \le N} \frac{1}{n} \frac{1}{(2 \pi i)^2} \oint \oint \bigg(\frac{N}{n}\bigg)^{s+u} \frac{ds}{s^{i+1}}\frac{du}{u^{j+1}} + O(L^{i+j+1}) \\
&= \frac{1}{\alpha+\beta}\sum_{n \le N} \frac{1}{n} \frac{1}{i!j!}\bigg(\log\frac{N}{n}\bigg)^{i+j} + O(L^{i+j+1}).
\end{align*}
Again we sum over $i$ and $j$ and obtain
\begin{align*}
I_{1,2}(\alpha,\beta) &= \frac{T\widehat{\Phi}(0)}{\log^2 N}\frac{1}{\alpha+\beta}\sum_{n \le N} \frac{1}{n} \sumtwo_{i,j}\frac{p_{\ell_1,i}p_{\ell_2,j}i!j!}{\log^{i+j} N}  \frac{1}{i!j!}\bigg(\log\frac{N}{n}\bigg)^{i+j} + O(T/L) \\
&= \frac{T\widehat{\Phi}(0)}{\log^2 N} \frac{1}{\alpha+\beta}\sum_{n \le N} \frac{1}{n} P_{1}\bigg(\frac{\log(N/n)}{\log N}\bigg)P_{2}\bigg(\frac{\log(N/n)}{\log N}\bigg) + O(T/L).
\end{align*}
Euler-Maclaurin summation turns this into
\begin{align*}
I_{1,2}(\alpha,\beta) = \frac{T\widehat{\Phi}(0)}{\log N}\frac{1}{\alpha+\beta}\int_0^{1} P_{1}(u)P_{2}(u)du + O(T/L).
\end{align*}

The cases $J_{1,3}$ and $J_{1,4}$ are identical, up to the obvious symmetries. We do $J_{1,3}$. The same strategy as before leads us to
\begin{align*}
J_{1,3} &= \frac{1}{\alpha+\beta}\sum_{n \le N} \frac{(\mathbf{1} \star \Lambda_1)(n)}{n} L_{1,1}L_{2,2} \\
&= - \frac{1}{\alpha+\beta}\sum_{n \le N} \frac{(\mathbf{1} \star \Lambda_1)(n)}{n} \frac{1}{i!}\frac{\log^i N}{\log N} \frac{d}{dx} N^{\alpha x} \bigg(x+\frac{\log(N/n)}{\log N} \bigg)^i \frac{1}{j!} \bigg(\log\frac{N}{n}\bigg)^j \bigg|_{x=0} + O(L^{i+j}).
\end{align*}
Summing over $i$ and $j$ yields
\begin{align*}
I_{1,3}(\alpha,\beta) %&=  -\frac{1}{\alpha+\beta}\sum_{n \le N} \frac{(\mathbf{1} \star \Lambda_1)(n)}{n} \sumtwo_{i,j}\frac{p_{\ell_1,i}p_{\ell_2,j}i!j!}{\log^{i+j}N} \frac{1}{i!}\frac{\log^i N}{\log N} \frac{d}{dx} N^{\alpha x} \bigg(x+\frac{\log(N/n)}{\log N} \bigg)^i \frac{1}{j!} \bigg(\log\frac{N}{n}\bigg)^j\bigg|_{x=0} + O(T/L) \\
&=  -\frac{T\widehat{\Phi}(0)}{\log^2 N}\frac{(\log N)^{-1}}{\alpha+\beta}\frac{d}{dx} N^{\alpha x} \sum_{n \le N} \frac{(\mathbf{1} \star \Lambda_1)(n)}{n} \sumtwo_{i,j} p_{\ell_1,i}p_{\ell_2,j}  \bigg(x+\frac{\log(N/n)}{\log N} \bigg)^i  \bigg(\frac{\log(N/n)}{\log N}\bigg)^j\bigg|_{x=0} \\
& \quad  + O(T/L) \\
& =  -\frac{T\widehat{\Phi}(0)}{\log^3 N}\frac{1}{\alpha+\beta}\frac{d}{dx} N^{\alpha x} \sum_{n \le N} \frac{(\mathbf{1} \star \Lambda_1)(n)}{n} P_{1}\bigg(x+\frac{\log(N/n)}{\log N} \bigg) P_{2}\bigg(\frac{\log(N/n)}{\log N}\bigg) \bigg|_{x=0} + O(T/L).
\end{align*}
Lastly, Euler-Maclaurin transforms this into
\begin{align*}
I_{1,3}(\alpha,\beta) =  -\frac{T\widehat{\Phi}(0)}{\log N}\frac{1}{\alpha+\beta}\frac{d}{dx} N^{\alpha x} \int_0^1  (1-u) P_{1}(x+u) P_{2}(u)du \bigg|_{x=0} + O(T/L).
\end{align*}
By symmetry
\begin{align*}
I_{1,4}(\alpha,\beta) =  -\frac{T\widehat{\Phi}(0)}{\log N}\frac{1}{\alpha+\beta}\frac{d}{dy} N^{\beta y} \int_0^1  (1-u) P_{1}(u)P_{2}(y+u)du \bigg|_{y=0} + O(T/L).
\end{align*}
Finally, the last term is very similar to the first one since
\begin{align*}
J_{1,5} &= - \frac{1}{\alpha+\beta}\sum_p \bigg(\frac{\log p}{p^{1+\alpha+\beta}-1}\bigg)^2 \sum_{n \le N} \frac{1}{n} L_{1,1}L_{1,2},
\end{align*}
so that the sum over $i$ and $j$ takes us to
\begin{align*}
I_{1,5} & = - \frac{T\widehat{\Phi}(0)}{\log^3 N}\sum_p \bigg(\frac{\log p}{p^{1+\alpha+\beta}-1}\bigg)^2 \frac{1}{\alpha+\beta}\frac{d^2}{dxdy}N^{\alpha x + \beta y} \int_0^{1} P_{1} (x+u) P_{2} (y+u)du \bigg|_{x=y=0} + O(T/L).
\end{align*}
We remark that with respect to $\log N$ this last term is of smaller order of magnitude smaller than the other terms $I_{1,i}$ for $i \in \{1,2,3,4\}$.

The final step is to assemble back $I_1(\alpha,\beta)$ with $I_{2}(\alpha,\beta)$, where $I_2$ was given by the second part of Theorem \ref{meanvalueintegral} or of equation \eqref{presumS}. This is an easy step that produces a lot of attrition. We start with $I_{1,1}$, $I_{1,2}$ and move on chronologically. For the very first term we see that
\begin{align*}
I_{1}(\alpha,\beta) &= I_{1,1}(\alpha,\beta)+I_{2,1}(\alpha,\beta) = I_{1,1}(\alpha,\beta)+T^{-\alpha-\beta}I_{1,1}(-\beta,-\alpha) +O(T/L) \\
& = \frac{\widehat{\Phi}(0)}{\log N} \frac{d^2}{dxdy} \frac{N^{\alpha x + \beta y}-T^{-\alpha-\beta}N^{-\beta x - \alpha y}}{\alpha+\beta} \int_0^{1} (1-u)^2 P_{1} (x+u) P_{2} (y+u)du \bigg|_{x=y=0} \\
& \quad + O(T/L).
\end{align*}
However, if we focus on the part involving $\alpha$ and $\beta$ only, then
\begin{align*}
\frac{N^{\alpha x + \beta y}-T^{-\alpha-\beta}N^{-\beta x - \alpha y}}{\alpha+\beta} &= N^{\alpha x + \beta y}\frac{1-N^{-\alpha x - \beta y}T^{-\alpha-\beta}N^{-\beta x - \alpha y}}{\alpha+\beta} \\
&= N^{\alpha x + \beta y}\frac{1-(N^{x + y}T)^{-\alpha-\beta}}{\alpha+\beta} \\
&= N^{\alpha x + \beta y} \log(N^{x+y}T) \int_0^1 (N^{x+y}T)^{-t(\alpha+\beta)}dt.
\end{align*}
Next, we bring in the differential operators $Q$ to get
\begin{align*}
& Q \bigg(\frac{-1}{\log T}\frac{\partial}{\partial \alpha}\bigg) Q \bigg(\frac{-1}{\log T}\frac{\partial}{\partial \beta}\bigg) \frac{N^{\alpha x + \beta y}-T^{-\alpha-\beta}N^{-\beta x - \alpha y}}{\alpha+\beta} \\
& =  \log(N^{x+y}T) Q \bigg(\frac{-1}{\log T}\frac{\partial}{\partial \alpha}\bigg)Q \bigg(\frac{-1}{\log T}\frac{\partial}{\partial \beta}\bigg) \int_0^1 (N^{t(x+y)-y}T^t)^{-\alpha} (N^{t(x+y)-x}T^t)^{-\beta}dt \\
& =  \log(N^{x+y}T) \int_0^1 Q\bigg(\frac{\log N^{t(x+y)-y}T^t}{\log T}\bigg)Q\bigg(\frac{\log N^{t(x+y)-x}T^t}{\log T}\bigg) (N^{t(x+y)-y}T^t)^{-\alpha} (N^{t(x+y)-x}T^t)^{-\beta} dt \\
& =  \log(T^{\theta(x+y)+1}) \int_0^1 Q(\theta t(x+y)-\theta y+t)Q(\theta t(x+y)-\theta x+t) (T^{\theta t(x+y)-\theta y+t})^{-\alpha} (T^{\theta t(x+y)-\theta x+t})^{-\beta} dt.
\end{align*}
At $\alpha = \beta = -R/L$ the above reduces to 
\begin{align*}
& Q \bigg(\frac{-1}{\log T}\frac{\partial}{\partial \alpha}\bigg) Q \bigg(\frac{-1}{\log T}\frac{\partial}{\partial \beta}\bigg) \frac{N^{\alpha x + \beta y}-T^{-\alpha-\beta}N^{-\beta x - \alpha y}}{\alpha+\beta} \bigg|_{\alpha=\beta=-R/L} \\
& = \log(T^{\theta(x+y)+1}) \int_0^1 Q(\theta t(x+y)-\theta y+t)Q(\theta t(x+y)-\theta x+t) e^{R[\theta t(x+y)-\theta y+t]} e^{R[\theta t(x+y)-\theta x+t]} dt.
\end{align*}
Finally, the main term coming from this subdivision of $I$ is given by
\begin{align*}
I_{1} &:= Q \bigg(\frac{-1}{\log T}\frac{\partial}{\partial \alpha}\bigg) Q \bigg(\frac{-1}{\log T}\frac{\partial}{\partial \beta}\bigg) I_{1}(\alpha,\beta) \bigg|_{\alpha=\beta=-R/L},
\end{align*}
and putting all these things back together produces
\begin{align*}
I_{1} = T\widehat{\Phi}(0) \frac{d^2}{dxdy} \frac{\theta(x+y)+1}{\theta}  
& \int_0^{1}  \int_0^1  (1-u)^2 P_{1} (x+u) P_{2} (y+u) e^{R[\theta t(x+y)-\theta y+t]} e^{R[\theta t(x+y)-\theta x+t]} \\
& \times   Q(\theta t(x+y)-\theta y+t)Q(\theta t(x+y)-\theta x+t)  \bigg|_{x=y=0}  du dt + O(T/L).
\end{align*}

Next, for $I_{2}$ we had
\begin{align*}
I_{2}(\alpha,\beta) &= I_{2,1}(\alpha,\beta)+I_{2,2}(\alpha,\beta) = I_{2,1}(\alpha,\beta)+T^{-\alpha-\beta}I_{2,1}(-\beta,-\alpha) +O(T/L) \\
& = T\frac{\widehat{\Phi}(0)}{\log N} \frac{1-T^{-\alpha-\beta}}{\alpha+\beta} \int_0^{1} P_{1}(u)P_{2}(u)du + O(T/L) + O(T/L) \\
%& = \frac{\widehat{\Phi}(0)}{\log N} \log T \int_0^1 T^{-t(\alpha+\beta)}dt \int_0^{1} P_{1}(u)P_{2}(u)du + O(T/L) \\
& = T\frac{\widehat{\Phi}(0)}{\theta} \int_0^1 \int_0^1 T^{-t(\alpha+\beta)} P_{1}(u)P_{2}(u) dt du + O(T/L).
\end{align*}
Moreover, 
\begin{align*}
Q \bigg(\frac{-1}{\log T}\frac{\partial}{\partial \alpha}\bigg) Q \bigg(\frac{-1}{\log T}\frac{\partial}{\partial \beta}\bigg) T^{-t\alpha-t\beta} \bigg|_{\alpha=\beta=-R/L} = Q(t)^2 e^{2Rt},
\end{align*}
and therefore
\begin{align*}
I_{2}= T\frac{\widehat{\Phi}(0)}{\theta} \int_0^1 \int_0^1 Q(t)^2 e^{2Rt} P_{1}(u)P_{2}(u) dt du + O(T/L).
\end{align*}

We move on $I_{3}$ following the same strategy as before. For this term
\begin{align*}
\frac{N^{\alpha x} - T^{-\alpha-\beta}N^{-\beta x}}{\alpha+\beta} = N^{\alpha x}\frac{1 - (TN^x)^{-\alpha-\beta}}{\alpha+\beta} = N^{\alpha x} \log(T N^x) \int_0^1 (TN^x)^{-t (\alpha+\beta)}dt.
\end{align*}
Also
\begin{align*}
 Q \bigg(\frac{-1}{\log T}\frac{\partial}{\partial \alpha}\bigg) & Q \bigg(\frac{-1}{\log T}\frac{\partial}{\partial \beta}\bigg)  (T^t N^{xt})^{-\alpha} (N^{-x} T^t N^{xt})^{-\beta}  \bigg|_{\alpha=\beta=-R/L} \\
&= Q(t+\theta x t)Q(-\theta x+t+\theta x t) e^{R[t+\theta xt]}e^{R[-\theta x + t + \theta x t]}.
\end{align*}
Therefore
\begin{align*}
I_{3} = -T\widehat{\Phi}(0) \frac{1+\theta x}{\theta} \frac{d}{dx} &\int_0^1\int_0^1 (1-u)P_1(x+u)P_2(u)e^{R[t+\theta xt]}e^{R[-\theta x + t + \theta x t]} \\
& \times Q(t+\theta x t)Q(-\theta x+t+\theta x t) dtdu \bigg|_{x=0} + O(T/L).
\end{align*}

By symmetry
\begin{align*}
I_{4} = -T\widehat{\Phi}(0) \frac{1+\theta y}{\theta} \frac{d}{dy} &\int_0^1\int_0^1 (1-u)P_1(u)P_2(y+u)e^{R[t+\theta yt]}e^{R[-\theta y + t + \theta y t]} \\
& \times Q(t+\theta y t)Q(-\theta y+t+\theta y t) dtdu \bigg|_{y=0} + O(T/L).
\end{align*}

Finally, we come to the last term that was of smaller order of magnitude in $\log N$. This term is similar to $I_{1}$ except for the presence of $(1-u)^2$ in the integrand and the power of $\log N$ in the denominator. The sum over primes is less than $\zeta(2)$ and
\[
\frac{d^2}{dxdy} N^{\alpha x + \beta y} \int_0^1 P_1(x+u)P_2(y+u)du  \ll 1  
\]
at $T \to \infty$, and since $N = T^{\theta}$ and $\alpha, \beta \ll L^{-1}$
\[
\frac{d^2}{dxdy} N^{\alpha x + \beta y}\bigg|_{x=y=0} = (\log N)(\alpha+\beta)N^{\alpha x+\beta y}\bigg|_{x=y=0} = (\log N)(\alpha +\beta) \ll 1.
\]
Thus for the integral in the term
\[
I_{5,1}(\alpha, \beta) = -\frac{T \widehat{\Phi}(0)}{\log^3 N} A_{\alpha, \beta}^{(1,1)}(0,0;\beta, \alpha) \frac{1}{\alpha+\beta} \frac{d^2}{dxdy} N^{\alpha x + \beta y} \int_0^1 P_1(x+u)P_2(y+u)du +O(T/L)
\]
we will get a term of size $\tfrac{T}{\log^3 N}\log T \ll \tfrac{T}{\log^2 T}$, and hence $I_{5,1}(\alpha,\beta) + T^{-\alpha-\beta}I_{5,1}(-\beta,-\alpha)$ is an error term of size $T/L^2 + T/L \ll T/L$.

Alternatively, to deal with $I_{5}$ we set $n=2$ in \cite[Entry 2.5]{ober} so that
\begin{align*}
\frac{1}{(a+x)^2} = -\frac{1}{2 \pi i} \int_{(c)} \pi a^{s-2} (s-1) \csc(\pi s)x^{-s} ds \quad \textnormal{with} \quad 0 < c < 2.
\end{align*}
Using $c=\frac{3}{2}$, $a=-1$ and $x=p^{1+\alpha+\beta}$ we can turn the denominator in the sum into an integral%.%yields
%\begin{align*}
%\frac{1}{(p^{1+\alpha+\beta}-1)^2} = -\frac{1}{2 \pi i} \int_{(c)} \pi (-1)^{s} (s-1) \csc(\pi s)(p^{1+\alpha+\beta})^{-s} ds.
%\end{align*}
%This implies that
\begin{align*}
- \sum_p \frac{\log^2 p}{(p^{1+\alpha+\beta}-1)^2} = \pi \sum_p (\log^2 p) \bigg(\frac{1}{2 \pi i} \int_{(\frac{3}{2})} (-1)^s (s-1) \csc(\pi s) p^{-s} (p^s)^{-\alpha}(p^s)^{-\beta} ds \bigg).
\end{align*}
%Recall that $I_5(\alpha,\beta)=I_{5,1}(\alpha,\beta)+I_{5,2}(\alpha,\beta)$. 
For this last part we have
\begin{align*}
I_{5,1}(\alpha,\beta) &+ I_{5,2}(\alpha,\beta) = I_{5,1}(\alpha,\beta)+T^{-\alpha-\beta} I_{5,1}(-\beta,-\alpha) + O(T/L) \\
& = \frac{\pi T \widehat{\Phi}(0)}{\log^3 N} \frac{d^2}{dxdy} \sum_p (\log^2 p) \bigg(\frac{1}{2 \pi i} \int_{(\frac{3}{2})}  \frac{N^{\alpha x + \beta y} (p^s)^{-\alpha}(p^s)^{-\beta}-T^{\alpha-\beta}N^{-\beta x - \alpha y} (p^s)^{\beta}(p^s)^{\alpha}}{\alpha+\beta} \\
& \quad \times (-1)^s (s-1) \csc(\pi s) p^{-s}ds \int_0^1 P_{1}(x+u)P_{2}(y+u)du\bigg) \bigg|_{x=y=0} + O(T/L).
\end{align*}
We examine the term containing $\alpha$ and $\beta$ and observe that
\begin{align*}
& \frac{N^{\alpha x + \beta y} (p^s)^{-\alpha}(p^s)^{-\beta}-T^{-\alpha-\beta}N^{-\beta x - \alpha y} (p^s)^{\beta}(p^s)^{\alpha}}{\alpha+\beta} \\
%& = N^{\alpha x + \beta y} (p^s)^{-\alpha}(p^s)^{-\beta}\frac{1- N^{-\alpha x - \beta y} (p^s)^{\alpha}(p^s)^{\beta}   T^{-\alpha-\beta}N^{-\beta x - \alpha y} (p^s)^{\beta}(p^s)^{\alpha}}{\alpha+\beta} \\
%& = N^{\alpha x + \beta y} (p^s)^{-\alpha-\beta}\frac{1- (N^{x +y} (p^{-2s})T)^{-\alpha-\beta}}{\alpha+\beta} \\
& = N^{\alpha x + \beta y} (p^s)^{-\alpha-\beta} \log(N^{x +y} (p^{-2s})T)\int_0^1 (N^{x +y} (p^{-2s})T)^{-t(\alpha+\beta)}dt.
\end{align*}
Bring in the differential operators $Q$ to get
\begin{align*}
& Q \bigg(\frac{-1}{\log T}\frac{\partial}{\partial \alpha}\bigg) Q \bigg(\frac{-1}{\log T}\frac{\partial}{\partial \beta}\bigg) \frac{N^{\alpha x + \beta y} (p^s)^{-\alpha}(p^s)^{-\beta}-T^{-\alpha-\beta}N^{-\beta x - \alpha y} (p^s)^{\beta}(p^s)^{\alpha}}{\alpha+\beta} \bigg|_{\alpha=\beta=-R/L} \\
& = \bigg(\theta(x +y)+1+\frac{-2s \log p}{\log T} \bigg) (\log T)  e^{R[-\theta y + \theta t(x+y) + t + \frac{(s-2st)\log p}{\log T}]} e^{R[-\theta x + \theta t(x+y) + t + \frac{(s-2st)\log p}{\log T}]} \\
& \quad \times Q\bigg(-\theta x+t\theta(x+y)+t+\frac{(s-2st) \log p}{\log T} \bigg)
Q\bigg(-\theta y + t\theta(x+y) + t + \frac{(s-2st)\log p}{\log T}\bigg).
\end{align*}
Thus we are left with
\begin{align*}
I_{5} &= \frac{\pi T \widehat{\Phi}(0)\log T}{\theta^3 \log^3 T} \frac{d^2}{dxdy} \sum_p (\log^2 p)  \frac{1}{2 \pi i} \int_{(\frac{3}{2})} \int_0^1 \int_0^1 \bigg(\theta(x +y)+1+\frac{-2s \log p}{\log T} \bigg)  \\
& \quad \times Q\bigg(-\theta x+t\theta(x+y)+t+\frac{(s-2st) \log p}{\log T} \bigg)
Q\bigg(-\theta y + t\theta(x+y) + t + \frac{(s-2st)\log p}{\log T}\bigg)\\
& \quad \times e^{R[-\theta y + \theta t(x+y) + t + \frac{(s-2st)\log p}{\log T}]} e^{R[-\theta x + \theta t(x+y) + t + \frac{(s-2st)\log p}{\log T}]} \\
& \quad \times (-1)^s (s-1) \csc(\pi s) p^{-s} P_{1}(x+u)P_{2}(y+u)dtduds \bigg|_{x=y=0} + O(T/L)\\
& \ll T/L.
\end{align*}
Hence the term associated to $A_{\alpha,\beta}^{(1,1)}(0,0;\beta,\alpha)$ is an error term. 

We also remark that
\begin{align*}
\frac{\zeta''}{\zeta}(s) - \bigg(\frac{\zeta'}{\zeta}(s)\bigg)^2 = \sum_p \frac{\log^2 p}{p^s} + O \bigg(\frac{\log^2 p}{p^{2-\varepsilon}}\bigg) 
\end{align*}
could prove useful in other scenarios to turn the sum into zeta functions.

\subsubsection{The case $\ell_1=\ell_2=2$}

Suppose we now move on to $\ell_1 = \ell_2 = 2$. We know that $A_{\alpha,\beta}(0,0,0,0;\beta,\alpha) = 1$ and a direct calculation from \eqref{logderivativeA} shows that
\begin{align*}
\frac{\partial}{\partial z_1} \log A_{\alpha,\beta}(z_1,z_2,w_1,w_2;\beta,\alpha)\bigg|_{z_1=z_2=w_1=w_2=0} &= 0, \\
\frac{\partial}{\partial z_2} \log A_{\alpha,\beta}(z_1,z_2,w_1,w_2;\beta,\alpha)\bigg|_{z_1=z_2=w_1=w_2=0} &= 0, \\
\frac{\partial}{\partial w_1} \log A_{\alpha,\beta}(z_1,z_2,w_1,w_2;\beta,\alpha)\bigg|_{z_1=z_2=w_1=w_2=0} &= 0, \\
\frac{\partial}{\partial w_2} \log A_{\alpha,\beta}(z_1,z_2,w_1,w_2;\beta,\alpha)\bigg|_{z_1=z_2=w_1=w_2=0} &= 0.
\end{align*}
Next, for the second derivatives
\begin{align*}
\frac{\partial^2}{\partial z_1 \partial z_2} \log A_{\alpha,\beta}(z_1,z_2,w_1,w_2;\beta,\alpha)\bigg|_{z_1=z_2=w_1=w_2=0} &= 0, \\
\frac{\partial^2}{\partial w_1 \partial w_2} \log A_{\alpha,\beta}(z_1,z_2,w_1,w_2;\beta,\alpha)\bigg|_{z_1=z_2=w_1=w_2=0} &= 0, \\
\frac{\partial^2}{\partial z_1 \partial w_1} \log A_{\alpha,\beta}(z_1,z_2,w_1,w_2;\beta,\alpha)\bigg|_{z_1=z_2=w_1=w_2=0} &= - \sum_p \bigg(\frac{\log p}{p^{1+\alpha+\beta}-1}\bigg)^2, \\
\frac{\partial^2}{\partial z_2 \partial w_2} \log A_{\alpha,\beta}(z_1,z_2,w_1,w_2;\beta,\alpha)\bigg|_{z_1=z_2=w_1=w_2=0} &= - \sum_p \bigg(\frac{\log p}{p^{1+\alpha+\beta}-1}\bigg)^2, \\
\frac{\partial^2}{\partial z_2 \partial w_1} \log A_{\alpha,\beta}(z_1,z_2,w_1,w_2;\beta,\alpha)\bigg|_{z_1=z_2=w_1=w_2=0} &= - \sum_p \bigg(\frac{\log p}{p^{1+\alpha+\beta}-1}\bigg)^2, \\
\frac{\partial^2}{\partial z_1 \partial w_2} \log A_{\alpha,\beta}(z_1,z_2,w_1,w_2;\beta,\alpha)\bigg|_{z_1=z_2=w_1=w_2=0} &= - \sum_p \bigg(\frac{\log p}{p^{1+\alpha+\beta}-1}\bigg)^2.
\end{align*}
For the triple and fourth derivatives
\begin{align*}
\frac{\partial^3}{\partial z_1 \partial z_2 \partial w_1} \log A_{\alpha,\beta}(z_1,z_2,w_1,w_2;\beta,\alpha)\bigg|_{z_1=z_2=w_1=w_2=0} &= 0, \\
\frac{\partial^3}{\partial z_1 \partial z_2 \partial w_2} \log A_{\alpha,\beta}(z_1,z_2,w_1,w_2;\beta,\alpha)\bigg|_{z_1=z_2=w_1=w_2=0} &= 0, \\
\frac{\partial^3}{\partial z_1 \partial w_1 \partial w_2} \log A_{\alpha,\beta}(z_1,z_2,w_1,w_2;\beta,\alpha)\bigg|_{z_1=z_2=w_1=w_2=0} &= 0, \\
\frac{\partial^3}{\partial z_2 \partial w_1 \partial w_2} \log A_{\alpha,\beta}(z_1,z_2,w_1,w_2;\beta,\alpha)\bigg|_{z_1=z_2=w_1=w_2=0} &= 0, \\
\frac{\partial^4}{\partial z_1 \partial z_2 \partial w_1 \partial w_2} \log A_{\alpha,\beta}(z_1,z_2,w_1,w_2;\beta,\alpha)\bigg|_{z_1=z_2=w_1=w_2=0} &= 0.
\end{align*}
The first four equations imply that
\begin{align*}
\frac{\partial}{\partial z_1} A_{\alpha,\beta}(z_1,z_2,w_1,w_2;\beta,\alpha)\bigg|_{z_1=z_2=w_1=w_2=0} &= 0, \\
\frac{\partial}{\partial z_2} A_{\alpha,\beta}(z_1,z_2,w_1,w_2;\beta,\alpha)\bigg|_{z_1=z_2=w_1=w_2=0} &= 0, \\
\frac{\partial}{\partial w_1} A_{\alpha,\beta}(z_1,z_2,w_1,w_2;\beta,\alpha)\bigg|_{z_1=z_2=w_1=w_2=0} &= 0, \\
\frac{\partial}{\partial w_2} A_{\alpha,\beta}(z_1,z_2,w_1,w_2;\beta,\alpha)\bigg|_{z_1=z_2=w_1=w_2=0} &= 0
\end{align*}
We now use \eqref{faadibruno} to see that
\begin{align*}
0 &= \frac{\partial^2}{\partial z_1 \partial z_2} \log A_{\alpha,\beta}(z_1, z_2, w_1, w_2; \beta, \alpha) \bigg|_{z_1=z_2=w_1=w_2=0} \\
& = - \frac{A_{\alpha,\beta}^{(1,0,0,0)}(0, 0, 0, 0; \beta, \alpha)A_{\alpha,\beta}^{(0,1,0,0)}(0, 0, 0, 0; \beta, \alpha)}{A_{\alpha,\beta}(0, 0, 0, 0; \beta, \alpha)^2} + \frac{A_{\alpha,\beta}^{(1,1,0,0)}(0, 0, 0, 0; \beta, \alpha)}{A_{\alpha,\beta}(0, 0, 0, 0; \beta, \alpha)} \\
& = A_{\alpha,\beta}^{(1,1,0,0)}(0, 0, 0, 0; \beta, \alpha).
\end{align*}
By symmetry
\begin{align*}
A_{\alpha,\beta}^{(0,0,1,1)}(0, 0, 0, 0; \beta, \alpha) = 0.
\end{align*}
Likewise
\begin{align*}
- \sum_p \bigg( \frac{\log p}{p^{1+\alpha+\beta}-1} \bigg)^2 &= \frac{\partial^2}{\partial z_1 \partial w_1} \log A_{\alpha,\beta}(z_1, z_2, w_1, w_2; \beta, \alpha) \bigg|_{z_1=z_2=w_1=w_2=0} \\
& = - \frac{A_{\alpha,\beta}^{(1,0,0,0)}(0, 0, 0, 0; \beta, \alpha)A_{\alpha,\beta}^{(0,0,1,0)}(0, 0, 0, 0; \beta, \alpha)}{A_{\alpha,\beta}(0, 0, 0, 0; \beta, \alpha)^2} + \frac{A_{\alpha,\beta}^{(1,0,1,0)}(0, 0, 0, 0; \beta, \alpha)}{A_{\alpha,\beta}(0, 0, 0, 0; \beta, \alpha)} \\
& = A_{\alpha,\beta}^{(1,0,1,0)}(0, 0, 0, 0; \beta, \alpha).
\end{align*}
By symmetry
\begin{align*}
A_{\alpha,\beta}^{(1,0,0,1)}(0, 0, 0, 0; \beta, \alpha) = A_{\alpha,\beta}^{(0,1,0,1)}(0, 0, 0, 0; \beta, \alpha) = A_{\alpha,\beta}^{(0,1,1,0)}(0, 0, 0, 0; \beta, \alpha) = - \sum_p \bigg( \frac{\log p}{p^{1+\alpha+\beta}-1} \bigg)^2.
\end{align*}
For the triple derivatives we get
\begin{align*}
0 & = \frac{\partial^3}{\partial z_1 \partial z_2 \partial w_1} \log A_{\alpha,\beta}(z_1, z_2, w_1, w_2; \beta, \alpha) \bigg|_{z_1=z_2=w_1=w_2} \\
& = \frac{2 A_{\alpha,\beta}^{(1,0,0,0)}(0, 0, 0, 0; \beta, \alpha) A_{\alpha,\beta}^{(0,1,0,0)}(0, 0, 0, 0; \beta, \alpha)A_{\alpha,\beta}^{(0,0,1,0)}(0, 0, 0, 0; \beta, \alpha)}{A_{\alpha,\beta}(0, 0, 0, 0; \beta, \alpha)^3} \\
& - \frac{A_{\alpha,\beta}^{(0,1,1,0)}(0, 0, 0, 0; \beta, \alpha) A_{\alpha,\beta}^{(1,0,0,0)}(0, 0, 0, 0; \beta, \alpha)}{A_{\alpha,\beta}(0, 0, 0, 0; \beta, \alpha)^2} \\
& - \frac{A_{\alpha,\beta}^{(0,1,0,0)}(0, 0, 0, 0; \beta, \alpha) A_{\alpha,\beta}^{(1,0,1,0)}(0, 0, 0, 0; \beta, \alpha)}{A_{\alpha,\beta}(0, 0, 0, 0; \beta, \alpha)^2} \\
& + \frac{A_{\alpha,\beta}^{(1,1,1,0)}(0, 0, 0, 0; \beta, \alpha)}{A_{\alpha,\beta}(0, 0, 0, 0; \beta, \alpha)} = A_{\alpha,\beta}^{(1,1,1,0)}(0, 0, 0, 0; \beta, \alpha).
\end{align*}
By symmetry
\begin{align*}
A_{\alpha,\beta}^{(1,1,0,1)}(0, 0, 0, 0; \beta, \alpha) = A_{\alpha,\beta}^{(1,0,1,1)}(0, 0, 0, 0; \beta, \alpha) = A_{\alpha,\beta}^{(0,1,1,1)}(0, 0, 0, 0; \beta, \alpha) = 0.
\end{align*}
Finally, (presenting only the surviving terms to save space)
\begin{align*}
0 & = \frac{\partial^4}{\partial z_1 \partial z_2 \partial w_1 \partial w_2} \log A_{\alpha,\beta}(z_1, z_2, w_1, w_2; \beta, \alpha) \bigg|_{z_1=z_2=w_1=w_2} \\
& = - \frac{A_{\alpha,\beta}^{(0,1,1,0)}(0,0,0,0;\beta,\alpha)A_{\alpha,\beta}^{(1,0,0,1)}(0,0,0,0;\beta,\alpha)}{A_{\alpha,\beta}(0,0,0,0;\beta,\alpha)^2}  - \frac{A_{\alpha,\beta}^{(0,1,0,1)}(0,0,0,0;\beta,\alpha)A_{\alpha,\beta}^{(1,0,1,0)}(0,0,0,0;\beta,\alpha)}{A_{\alpha,\beta}(0,0,0,0;\beta,\alpha)^2} \\
& + \frac{A_{\alpha,\beta}^{(1,1,1,1)}(0,0,0,0;\beta,\alpha)}{A_{\alpha,\beta}(0,0,0,0;\beta,\alpha)} = - 2\bigg(\sum_p \bigg(\frac{\log p}{p^{1+\alpha+\beta}-1}\bigg)^2\bigg)^2 + A_{\alpha, \beta}^{(1,1,1,1)}(0,0,0,0;\beta,\alpha).
\end{align*}
In other words
\begin{align*}
A_{\alpha,\beta}^{(1,1,1,1)}(0,0,0,0;\beta,\alpha) = 2\bigg(\sum_p \bigg(\frac{\log p}{p^{1+\alpha+\beta}-1}\bigg)^2\bigg)^2.
\end{align*}

Computing the contour integrals directly leads us to
\begin{align*}
\frac{1}{(2\pi i)^4} \oint \oint & \oint \oint \widetilde{\mathfrak{M}}_{\alpha,\beta}(z_1, z_2, w_1, w_2; s,u) \frac{dz_1}{z_1^2}\frac{dz_2}{z_2^2}\frac{dw_1}{w_1^2}\frac{dw_2}{w_2^2} = \frac{\zeta(1+s+u)\zeta(1+\alpha+\beta)}{\zeta(1+\alpha+s)\zeta(1+\beta+u)} \\
& \times \bigg[ 2 \bigg(\frac{\zeta''}{\zeta}(1+s+u)\bigg)^2 - \bigg(\frac{\zeta'}{\zeta}(1+s+u)\bigg)^4 + 2 \bigg(\frac{\zeta'}{\zeta}(1+s+u)\bigg)^3 \frac{\zeta'}{\zeta}(1+\alpha+s) \\
& \quad + 2 \bigg(\frac{\zeta'}{\zeta}(1+s+u)\bigg)^3 \frac{\zeta'}{\zeta}(1+\beta+u) + \bigg(\frac{\zeta'}{\zeta}(1+s+u)\bigg)^2 \bigg(\frac{\zeta'}{\zeta}(1+\alpha+s)\bigg)^2 \\
& \quad  + \bigg(\frac{\zeta'}{\zeta}(1+s+u)\bigg)^2 \bigg(\frac{\zeta'}{\zeta}(1+\beta+u)\bigg)^2 - 2 \bigg(\frac{\zeta'}{\zeta}(1+\beta+u)\bigg)^2 \frac{\zeta'}{\zeta}(1+\alpha+s) \frac{\zeta'}{\zeta}(1+s+u) \\
& \quad  - 2 \bigg(\frac{\zeta'}{\zeta}(1+\alpha+s)\bigg)^2 \frac{\zeta'}{\zeta}(1+\beta+u) \frac{\zeta'}{\zeta}(1+s+u) + \bigg(\frac{\zeta'}{\zeta}(1+\alpha+s)\bigg)^2 \bigg(\frac{\zeta'}{\zeta}(1+\beta+u)\bigg)^2 \\
& \quad  + 4 \frac{\zeta''}{\zeta}(1+s+u)\frac{\zeta'}{\zeta}(1+\alpha+s)\frac{\zeta'}{\zeta}(1+\beta+u) - 4 \frac{\zeta''}{\zeta}(1+s+u) \frac{\zeta'}{\zeta}(1+s+u)\frac{\zeta'}{\zeta}(1+\beta+u) \\
& \quad   - 4 \frac{\zeta''}{\zeta}(1+s+u) \frac{\zeta'}{\zeta}(1+s+u)\frac{\zeta'}{\zeta}(1+\alpha+s) \bigg].
\end{align*}
Here $\widetilde{\mathfrak{M}}_{\alpha,\beta}(z_1, z_2, w_1, w_2; s,u)$ is the same as $\mathfrak{M}_{\alpha,\beta}(z_1, z_2, w_1, w_2; s,u)$ but with $A_{\alpha,\beta}(z_1, z_2, w_1, w_2; s,u)$ replaced by $1$. The reason behind this change is that most of the terms of 
\begin{align*}
\frac{1}{(2\pi i)^4} \oint \oint \oint \oint \mathfrak{M}_{\alpha,\beta}(z_1, z_2, w_1, w_2; s,u) \frac{dz_1}{z_1^2}\frac{dz_2}{z_2^2}\frac{dw_1}{w_1^2}\frac{dw_2}{w_2^2}
\end{align*}
involve partial derivatives of $A$ that are zero and the ones that are not zero will eventually become error terms of size $O(T/L)$. Thus, for practical purposes we only really need to retain the terms that are multiplied by $A_{\alpha,\beta}(0,0,0,0;s,u)$. It is now clear that the process to achieve the main terms from here is the same as in the case $\ell_1 = \ell_2 = 1$, only that there are lot more of them.

\section{The general case $d \ge 0$}

We now specialize the coefficients to 
\begin{align} \label{specialgeneral1}
\mathfrak{a}_n := \frac{(\mu \star \Lambda_1^{\star\ell_1} \star \Lambda_2^{\star \ell_2} \star \cdots \star \Lambda_d^{\star \ell_d})(n)}{(\log N)^{\sum_{r=1}^d r \ell_r}},
\end{align}
as well as 
\begin{align} \label{specialgeneral2}
\mathfrak{b}_n := \frac{(\mu \star \Lambda_1^{*\barell_1} \star \Lambda_2^{\star \barell_2} \star \cdots \star \Lambda_d^{\star \barell_d})(n)}{(\log N)^{\sum_{\barr=1}^d \barr \barell_\barr}}.
\end{align}
The task ahead is to compute the arithmetical sum $\mathcal{S}_d = \sum_{hm=kn}$ appearing in \eqref{presumS}. This is precisely where the autocorrelation of ratios of $\zeta$ technique plays again a fundamental role and where we appreciate why it is unnecessary to discriminate square-free numbers $n$ as discussed in $\mathsection$3. To this end, define
\begin{align*}
\mathcal{S}_d :&= \sumtwo_{1 \le d,e \le \infty} \frac{\mathfrak{a}_d\overline{\mathfrak{b}_e}}{[d,e]} \frac{(d,e)^{\alpha+\beta}}{d^{\alpha+s}e^{\beta+u}} \\
&= \sum_{d,e} \frac{(\mu^{\star L_d + 1} \star (\log^{\ell_1}) \star (\log^2)^{\ell_2} \star \cdots \star (\log^d)^{\ell_d})(d)(\mu^{\star \barL_d + 1} \star (\log^{\barell_1}) \star (\log^2)^{\barell_2} \star \cdots \star (\log^d)^{\barell_d})(e)(d,e)^{\alpha+\beta}}{[d,e]d^{\alpha+s}e^{\beta+u} (\log N)^{\sum_{r=1}^d r \ell_r} (\log N)^{\sum_{\barr=1}^d \barr \barell_\barr}},
\end{align*}
by writing the generalized von Mangoldt function as $\Lambda_q = \mu \star \log^q$ and setting $L_d := \ell_1 + \ell_2 + \cdots + \ell_d$.

We have the rather unfortunate situation where it becomes difficult to recycle the alphabet but it should be clear that only the $d$'s in $(d,e), [d,e]$ and $d^{\alpha+s}$ are being summed. The $d$'s involved in the logs are a parameter of our choice.

We now write 
\[
d = d_0 (\prod_{1 \le i \le \ell_1} d_{1,i}) (\prod_{1 \le i \le \ell_2} d_{2,i}) \cdots (\prod_{1 \le i \le \ell_d} d_{d,i}),
\] 
and 
\[
e = e_0 (\prod_{1 \le j \le \barell_1} e_{1,j}) (\prod_{1 \le j \le \barell_2} e_{2,j}) \cdots (\prod_{1 \le j \le \barell_d} e_{d,j}).
\]
so that $\mathcal{S}_d$ becomes
\begin{align*}
\mathcal{S}_d &= \sumdots_{\substack{d_0, d_{1,1}, \cdots, d_{1,\ell_1}, d_{2,1}, \cdots, d_{2,\ell_1} \cdots d_{d,1} \cdots d_{d,\ell_1} \\ e_0, e_{1,1}, \cdots, e_{1,\ell_2}, e_{2,1}, \cdots, e_{2,\ell_2} \cdots e_{d,1} \cdots e_{d,\ell_2}}} \mu^{\star L_d+1}(d_0)\mu^{\star \barL_d+1}(e_0) \\
& \times \frac{(d_0 (\prod_{1 \le i \le \ell_1} d_{1,i}) (\prod_{1 \le i \le \ell_2} d_{2,i}) \cdots (\prod_{1 \le i \le \ell_d} d_{d,i}),e_0 (\prod_{1 \le j \le \barell_1} e_{1,j}) (\prod_{1 \le j \le \barell_2} e_{2,j}) \cdots (\prod_{1 \le j \le \barell_d} e_{d,j}))^{\alpha+\beta}}{[d_0 (\prod_{1 \le i \le \ell_1} d_{1,i}) (\prod_{1 \le i \le \ell_2} d_{2,i}) \cdots (\prod_{1 \le i \le \ell_d} d_{d,i}),e_0 (\prod_{1 \le j \le \barell_1} e_{1,j}) (\prod_{1 \le j \le \barell_2} e_{2,j}) \cdots (\prod_{1 \le j \le \barell_d} e_{d,j})]} \\
& \times \frac{(\prod_{1 \le i \le \ell_1} \log d_{1,i}) (\prod_{1 \le i \le \ell_2} \log d_{2,i}) \cdots (\prod_{1 \le i \le \ell_d} \log d_{d,i})}{(d_0 (\prod_{1 \le i \le \ell_1} d_{1,i}) (\prod_{1 \le i \le \ell_2} d_{2,i}) \cdots (\prod_{1 \le i \le \ell_d} d_{d,i}))^{\alpha+s} (\log N)^{\sum_{r=1}^d r \ell_r}} \\
& \times \frac{(\prod_{1 \le j \le \barell_1} \log e_{1,j}) (\prod_{1 \le j \le \barell_2} \log e_{2,j}) \cdots (\prod_{1 \le j \le \barell_d} \log e_{d,j})}{(e_0 (\prod_{1 \le j \le \barell_1} e_{1,i}) (\prod_{1 \le j \le \barell_2} e_{2,j}) \cdots (\prod_{1 \le j \le \barell_d} e_{d,j}))^{\beta+u} (\log N)^{\sum_{\barr=1}^d \barr \barell_\barr}}.
\end{align*}
To remove the various powers of $\log$, we now employ Cauchy's integral formula
\begin{align} \label{cauchydgeneral}
\log^q x = (-1)^q \frac{\partial^q}{\partial \gamma^q} \frac{1}{x^\gamma} \bigg|_{\gamma = 0} = (-1)^q  \frac{q!}{2 \pi i} \oint \frac{1}{x^q} \frac{dz}{z^{q+1}}, \quad q = 1,2,\cdots, 
\end{align}
where the contour of integration is a small circle around the origin. This further transforms $\mathcal{S}_d$ into
\begin{align*}
\mathcal{S}_d &= \frac{1}{(2 \pi i)^{\ell_1}}\ointdots \frac{1}{(2 \pi i)^{\ell_2}}\ointdots \cdots \frac{1}{(2 \pi i)^{\ell_d}}\ointdots (-1)^{1 \times \ell_1}(-1)^{2 \times \ell_2}\cdots(-1)^{d \times \ell_d} \\
& \times \frac{1}{(2 \pi i)^{\barell_1}}\ointdots \frac{1}{(2 \pi i)^{\barell_2}}\ointdots \cdots \frac{1}{(2 \pi i)^{\barell_d}}\ointdots (-1)^{1 \times \barell_1}(-1)^{2 \times \barell_2}\cdots(-1)^{d \times \barell_d} \\
& \times \sumdots_{\substack{d_0, d_{1,1}, \cdots, d_{1,\ell_1}, d_{2,1}, \cdots, d_{2,\ell_1} \cdots d_{d,1} \cdots d_{d,\ell_1} \\ e_0, e_{1,1}, \cdots, e_{1,\ell_2}, e_{2,1}, \cdots, e_{2,\ell_2} \cdots e_{d,1} \cdots e_{d,\ell_2}}} \mu^{\star L_d+1}(d_0)\mu^{\star \barL_d+1}(e_0) \\
& \times \frac{(d_0 (\prod_{1 \le i \le \ell_1} d_{1,i}) (\prod_{1 \le i \le \ell_2} d_{2,i}) \cdots (\prod_{1 \le i \le \ell_d} d_{d,i}),e_0 (\prod_{1 \le j \le \barell_1} e_{1,j}) (\prod_{1 \le j \le \barell_2} e_{2,j}) \cdots (\prod_{1 \le j \le \barell_d} e_{d,j}))^{\alpha+\beta}}{[d_0 (\prod_{1 \le i \le \ell_1} d_{1,i}) (\prod_{1 \le i \le \ell_2} d_{2,i}) \cdots (\prod_{1 \le i \le \ell_d} d_{d,i}),e_0 (\prod_{1 \le j \le \barell_1} e_{1,j}) (\prod_{1 \le j \le \barell_2} e_{2,j}) \cdots (\prod_{1 \le j \le \barell_d} e_{d,j})]} \\
& \times \frac{1}{d_0^{\alpha+s} (\prod_{1 \le i \le \ell_1} d_{1,i}^{\alpha+s+z_{1,i}}) (\prod_{1 \le i \le \ell_2} d_{2,i}^{\alpha+s+z_{2,i}}) \cdots (\prod_{1 \le i \le \ell_d} d_{d,i}^{\alpha+s+z_{d,i}}) (\log N)^{\sum_{r=1}^d r \ell_r}} \\
& \times \frac{1}{e_0^{\beta+u} (\prod_{1 \le j \le \barell_1} e_{1,i}^{\beta+u+w_{1,j}}) (\prod_{1 \le j \le \barell_2} e_{2,j}^{\beta+u+w_{2,j}}) \cdots (\prod_{1 \le j \le \barell_d} e_{d,j}^{\beta+u+w_{d,j}}) (\log N)^{\sum_{\barr=1}^d \barr \barell_\barr}} \\
& \times \frac{dz_{1,1}}{z_{1,1}^{1+1}} \cdots \frac{dz_{1,\ell_1}}{z_{1,\ell_1}^{1+1}} \frac{dz_{2,1}}{z_{2,1}^{2+1}} \cdots \frac{dz_{2,\ell_2}}{z_{2,\ell_2}^{2+1}} \cdots \frac{dz_{d,1}}{z_{d,1}^{d+1}} \cdots \frac{dz_{d,\ell_d}}{z_{d,\ell_d}^{d+1}}
\frac{dw_{1,1}}{w_{1,1}^{1+1}} \cdots \frac{dw_{1,\barell_1}}{w_{1,\barell_1}^{1+1}} \frac{dw_{2,1}}{w_{2,1}^{2+1}} \cdots \frac{dw_{2,\barell_2}}{w_{2,\barell_2}^{2+1}} \cdots \frac{dw_{d,1}}{w_{d,1}^{d+1}} \cdots \frac{dz_{w,\barell_d}}{z_{w,\barell_d}^{d+1}}.
\end{align*}
The multiplicativity of the above expression is now conducive for a representation as an Euler product format. One has that
\begin{align*}
\mathcal{S}_d &= \frac{1}{(2 \pi i)^{\ell_1}}\ointdots \frac{1}{(2 \pi i)^{\ell_2}}\ointdots \cdots \frac{1}{(2 \pi i)^{\ell_d}}\ointdots (-1)^{1 \times \ell_1}(-1)^{2 \times \ell_2}\cdots(-1)^{d \times \ell_d} \\
& \times \frac{1}{(2 \pi i)^{\barell_1}}\ointdots \frac{1}{(2 \pi i)^{\barell_2}}\ointdots \cdots \frac{1}{(2 \pi i)^{\barell_d}}\ointdots (-1)^{1 \times \barell_1}(-1)^{2 \times \barell_2}\cdots(-1)^{d \times \barell_d} \\
& \times \prod_p \sumdots_{\substack{p^{d_0}, p^{d_{1,1}}, \cdots, p^{d_{1,\ell_1}}, p^{d_{2,1}}, \cdots, p^{d_{2,\ell_1}} \cdots p^{d_{d,1}} \cdots p^{d_{d,\ell_1}} \\ p^{e_0}, p^{e_{1,1}}, \cdots, p^{e_{1,\ell_2}}, p^{e_{2,1}}, \cdots, p^{e_{2,\ell_2}}, \cdots, p^{e_{d,1}}, \cdots, p^{e_{d,\ell_2}}}} \mu^{\star L_d+1}(p^{d_0})\mu^{\star \barL_d+1}(p^{e_0}) \\
& \times \frac{\gcd(p^{d_0} (\prod_{1 \le i \le \ell_1} p^{d_{1,i}}) (\prod_{1 \le i \le \ell_2} p^{d_{2,i}}) \cdots (\prod_{1 \le i \le \ell_d} p^{d_{d,i}}),}{\operatorname{lcm}[p^{d_0} (\prod_{1 \le i \le \ell_1} p^{d_{1,i}}) (\prod_{1 \le i \le \ell_2} p^{d_{2,i}}) \cdots (\prod_{1 \le i \le \ell_d} p^{d_{d,i}}), } \\
& \quad \quad \quad \quad \frac{p^{e_0} (\prod_{1 \le j \le \barell_1} p^{e_{1,j}}) (\prod_{1 \le j \le \barell_2} p^{e_{2,j}}) \times \cdots \times (\prod_{1 \le j \le \barell_d} p^{e_{d,j}}))^{\alpha+\beta}}{p^{e_0} (\prod_{1 \le j \le \barell_1} p^{e_{1,j}}) (\prod_{1 \le j \le \barell_2} p^{e_{2,j}}) \times \cdots \times (\prod_{1 \le j \le \barell_d} p^{e_{d,j}})]}\\
& \times \frac{1}{(p^{d_0})^{\alpha+s} (\prod_{1 \le i \le \ell_1} (p^{d_{1,i}})^{\alpha+s+z_{1,i}}) (\prod_{1 \le i \le \ell_2} (p^{d_{2,i}})^{\alpha+s+z_{2,i}}) \cdots (\prod_{1 \le i \le \ell_d} (p^{d_{d,i}})^{\alpha+s+z_{d,i}}) (\log N)^{\sum_{r=1}^d r \ell_r}} \\
& \times \frac{1}{(p^{e_0})^{\beta+u} (\prod_{1 \le j \le \barell_1} (p^{e_{1,i}})^{\beta+u+w_{1,j}}) (\prod_{1 \le j \le \barell_2} (p^{e_{2,j}})^{\beta+u+w_{2,j}}) \cdots (\prod_{1 \le j \le \barell_d} (p^{e_{d,j}})^{\beta+u+w_{d,j}}) (\log N)^{\sum_{\barr=1}^d \barr \barell_\barr}} \\
& \times \frac{dz_{1,1}}{z_{1,1}^{1+1}} \cdots \frac{dz_{1,\ell_1}}{z_{1,\ell_1}^{1+1}} \frac{dz_{2,1}}{z_{2,1}^{2+1}} \cdots \frac{dz_{2,\ell_2}}{z_{2,\ell_2}^{2+1}} \cdots \frac{dz_{d,1}}{z_{d,1}^{d+1}} \cdots \frac{dz_{d,\ell_d}}{z_{d,\ell_d}^{d+1}} 
\frac{dw_{1,1}}{w_{1,1}^{1+1}} \cdots \frac{dw_{1,\barell_1}}{w_{1,\barell_1}^{1+1}} \frac{dw_{2,1}}{w_{2,1}^{2+1}} \cdots \frac{dw_{2,\barell_2}}{w_{2,\barell_2}^{2+1}} \cdots \frac{dw_{d,1}}{w_{d,1}^{d+1}} \cdots \frac{dz_{w,\barell_d}}{z_{w,\barell_d}^{d+1}}.
\end{align*}
To recover the main terms, we find the constant ($p^0$) and linear ($p^1$) terms of the above Euler product. To do so, we must constrain the choices of the $d$'s and the $e$'s to nine possibilities, much like in the $d=1$ case. While the computation is not complicated, it is so cumbersome that we provide it below for the benefit of the reader. The enumeration will be as follows. Label the choices by $\{d_0, d_{q,i}, e_0, e_{\barq,j} \}$ for some $q,\barq \in \{1,2,\cdots,d\}$, some $i \in \{1,2,\cdots,\ell_1\}$ and some $j \in \{1,2,\cdots,\ell_2\}$.
\begin{enumerate}
\item[(1)] For $\{0,0,0,0\}$ we naturally get $1$ in the Euler product. This will be the constant term.
\item[(2)] The case $\{1,0,0,0\}$ yields 
\begin{align*}
\frac{\mu^{\star L_d+1}(p)}{[p,1]} \frac{(p,1)^{\alpha+\beta}}{p^{\alpha+s}} = - \frac{L_d+1}{p^{1+\alpha+s}}.
\end{align*}
\item[(3)] By symmetry the case $\{0,0,1,0\}$ yields
\begin{align*}
 - \frac{\barL_d+1}{p^{1+\beta+u}}.
\end{align*}
\item[(4)] From the case $\{1,0,1,0\}$ we obtain
\begin{align*}
\frac{\mu^{\star L_d+1}(p)\mu^{\star \barL_d+1}(p)}{[p,p]} \frac{(p,p)^{\alpha+\beta}}{p^{\alpha+s}p^{\beta+u}} = \frac{(L_d+1)(\barL_d+1)}{p^{1+s+u}}.
\end{align*}
\item[(5)] The case $\{1,0,0,1\}$ is 
\begin{align*}
\frac{\mu^{\star L_d+1}(p)}{[p,p]} \frac{(p,p)^{\alpha+\beta}}{p^{\alpha+s}p^{\beta+u+w_{\barq,j}}} = -\frac{L_d+1}{p^{1+s+u+w_{\barq,j}}}.
\end{align*}
\item[(6)] The symmetric case to (5), which is $\{0,1,1,0\}$, is given by
\begin{align*}
- \frac{\barL_d+1}{p^{1+s+u+z_{q,i}}}.
\end{align*}
\item[(7)] The case $\{0,1,0,0\}$ yields
\begin{align*}
\frac{1}{[p,1]} \frac{(p,1)^{\alpha+\beta}}{p^{\alpha+s+z_{q,i}}} = \frac{1}{p^{1+\alpha+s+z_{q,i}}}.
\end{align*}
\item[(8)] Immediately, we see that for $\{0,0,0,1\}$ we get
\begin{align*}
\frac{1}{p^{1+\beta+u+w_{\barq,j}}}.
\end{align*}
\item[(9)] The last case is the most unpleasant one because it mixes $z$'s and $w$'s. For $\{0,1,0,1\}$ we arrive at
\begin{align*}
\frac{1}{[p,p]} \frac{(p,p)^{\alpha+\beta}}{p^{\beta+u+w_{\barq,j}} p^{\alpha+s+z_{q,i}}} = \frac{1}{p^{1+s+u+z_{q,i}+w_{\barq,j}}}.
\end{align*}
\end{enumerate}
We now insert these terms back into $\mathcal{S}_d$ and write
\begin{align*}
\mathcal{S}_d &= \frac{1}{(2 \pi i)^{\ell_1}}\ointdots \frac{1}{(2 \pi i)^{\ell_2}}\ointdots \cdots \frac{1}{(2 \pi i)^{\ell_d}}\ointdots (-1)^{1 \times \ell_1}(-1)^{2 \times \ell_2}\cdots(-1)^{d \times \ell_d} \\
& \quad \times \frac{1}{(2 \pi i)^{\barell_1}}\ointdots \frac{1}{(2 \pi i)^{\barell_2}}\ointdots \cdots \frac{1}{(2 \pi i)^{\barell_d}}\ointdots (-1)^{1 \times \barell_1}(-1)^{2 \times \barell_2}\cdots(-1)^{d \times \barell_d} \\
& \quad \times \prod_p \bigg(1 - \frac{L_d+1}{p^{1+\alpha+s}} - \frac{\barL_d+1}{p^{1+\beta+u}} + \frac{(L_d+1)(\barL_d+1)}{p^{1+s+u}} - \sum_{\barq=1}^d \sum_{j=1}^{\barell_\barq} \frac{L_d+1}{p^{1+s+u+w_{\barq,j}}} - \sum_{q=1}^d \sum_{i=1}^{\ell_q} \frac{\barL_d+1}{p^{1+s+u+z_{q,i}}} \\
& + \sum_{q=1}^d \sum_{i=1}^{\ell_q} \frac{1}{p^{1+\alpha+s+z_{q,i}}} + \sum_{\barq=1}^d \sum_{j=1}^{\barell_\barq} \frac{1}{p^{1+\beta+u+w_{\barq,j}}}+ \sum_{q=1}^d \sum_{i=1}^{\ell_q} \sum_{\barq=1}^d \sum_{j=1}^{\barell_\barq} \frac{1}{p^{1+s+u+z_{q,i}+w_{\barq,j}}} \bigg) \frac{A_{\alpha,\beta}(\{z_{q,i}\}, \{w_{\barq,j}\}; s, u)}{(\log N)^{\sum r \ell_r} (\log N)^{\sum \barr \barell_\barr}}\\
& \quad \times \frac{dz_{1,1}}{z^{1+1}} \cdots \frac{dz_{1,\ell_1}}{z^{1+1}} \frac{dz_{2,1}}{z^{2+1}} \cdots \frac{dz_{2,\ell_2}}{z^{2+1}} \cdots \frac{dz_{d,1}}{z^{d+1}} \cdots \frac{dz_{d,\ell_d}}{z^{d+1}} \frac{dw_{1,1}}{w^{1+1}} \cdots \frac{dw_{1,\barell_1}}{w^{1+1}} \frac{dw_{2,1}}{w^{2+1}} \cdots \frac{dw_{2,\barell_2}}{w^{2+1}} \cdots \frac{dw_{d,1}}{w^{d+1}} \cdots \frac{dz_{w,\barell_d}}{w^{d+1}}.
\end{align*}
The last step is to write this in as an autocorrelation ratio of zeta functions. This naturally leads us to
\begin{align} \label{generaltermsauto}
\mathcal{S}_d &= \frac{1}{(2 \pi i)^{\ell_1}}\ointdots \frac{1}{(2 \pi i)^{\ell_2}}\ointdots \cdots \frac{1}{(2 \pi i)^{\ell_d}}\ointdots (-1)^{1 \times \ell_1}(-1)^{2 \times \ell_2}\cdots(-1)^{d \times \ell_d} \nonumber \\
& \quad  \times \frac{1}{(2 \pi i)^{\barell_1}}\ointdots \frac{1}{(2 \pi i)^{\barell_2}}\ointdots \cdots \frac{1}{(2 \pi i)^{\barell_d}}\ointdots (-1)^{1 \times \barell_1}(-1)^{2 \times \barell_2}\cdots(-1)^{d \times \barell_d} \nonumber \\
& \quad  \times \frac{\zeta(1+s+u)^{(L_d+1)(\barL_d+1)}}{\zeta(1+\alpha+s)^{L_d+1}\zeta(1+\beta+u)^{\barL_d+1}} A_{\alpha,\beta}(\{z_{q,i}\}, \{w_{\barq,j}\}; s, u) \nonumber \\ 
& \quad  \times \bigg(\prod_{q=1}^d \prod_{i=1}^{\ell_q} \prod_{\barq=1}^d \prod_{j=1}^{\barell_\barq} \zeta(1+s+u+z_{q,i}+w_{\barq,j}) \bigg) \nonumber \\
& \quad  \times \bigg(\prod_{q=1}^d \prod_{i=1}^{\ell_q} \frac{\zeta(1+\alpha+s+z_{q,i})}{\zeta(1+s+u+z_{q,i})^{\barL_d+1}}\bigg) \bigg( \prod_{\barq=1}^d \prod_{j=1}^{\barell_\barq} \frac{\zeta(1+\beta+u+z_{\barq,j})}{\zeta(1+s+u+w_{\barq,j})^{L_d+1}} \bigg) \nonumber \\
& \quad  \times \frac{dz_{1,1}}{z_{1,1}^{1+1}} \cdots \frac{dz_{1,\ell_1}}{z_{1,\ell_1}^{1+1}} \frac{dz_{2,1}}{z_{2,1}^{2+1}} \cdots \frac{dz_{2,\ell_2}}{z_{2,\ell_2}^{2+1}} \cdots \frac{dz_{d,1}}{z_{d,1}^{d+1}} \cdots \frac{dz_{d,\ell_d}}{z_{d,\ell_d}^{d+1}} \nonumber \\
& \quad  \times \frac{dw_{1,1}}{w_{1,1}^{1+1}} \cdots \frac{dw_{1,\barell_1}}{w_{1,\barell_1}^{1+1}} \frac{dw_{2,1}}{w_{2,1}^{2+1}} \cdots \frac{dw_{2,\barell_2}}{w_{2,\barell_2}^{2+1}} \cdots \frac{dw_{d,1}}{w_{d,1}^{d+1}} \cdots \frac{dz_{w,\barell_d}}{z_{w,\barell_d}^{d+1}} \frac{1}{(\log N)^{\sum r \ell_r} (\log N)^{\sum \barr \barell_\barr} }.
\end{align}
Here $A=A_{\alpha,\beta}(\{z_{q,i}\}, \{w_{\barq,j}\}; s, u) = A_{\alpha,\beta}(\mathbf{z},\mathbf{w};s,u)$ is again an arithmetical factor that is absolutely convergent in some half-plane containing the origin. We can re-construct this arithmetical term by assembling back the nine cases above. The term $A_{\alpha, \beta}$ is given by
\begin{align} \label{Arithemticcleangeneral}
A_{\alpha,\beta}(\mathbf{z},\mathbf{w};s,u) &= \prod_p \bigg\{ \frac{(1-\tfrac{1}{p^{1+s+u}})^{(L_d+1)(\barL_d+1)}}{(1- \tfrac{1}{p^{1+\alpha+s}})^{(L_d+1)}(1- \tfrac{1}{p^{1+\beta+u}})^{(\barL_d+1)}}  \bigg( \prod_{q=1}^d \prod_{i=1}^{\ell_q} \prod_{\barq=1}^{d} \prod_{j=1}^{\barell_\barq} \bigg( 1-\frac{1}{p^{1+s+u+z_{q,i}+w_{\barq,j}}}\bigg) \bigg) \nonumber \\
& \times \bigg(\prod_{q=1}^d \prod_{i=1}^{\ell_q} \frac{1-\tfrac{1}{p^{1+\alpha+s+z_{q,i}}}}{(1-\tfrac{1}{p^{1+s+u+z_{q,i}}})^{(\barL_d+1)}}\bigg) \bigg(\prod_{\barq=1}^d \prod_{j=1}^{\barell_\barq} \frac{1-\tfrac{1}{p^{1+\beta+u+w_{\barq,j}}}}{(1-\tfrac{1}{p^{1+s+u+w_{\barq,j}}})^{(L_d+1)}} \bigg) \nonumber \\
& \times \bigg[ 1-\frac{L_d+1}{p^{1+\alpha+s}} -\frac{\barL_d+1}{p^{1+\beta+u}} + \frac{(L_d+1)(\barL_d+1)}{p^{1+s+u}} - \sum_{\barq=1}^d \sum_{j=1}^{\barell_\barq} \frac{L_d+1}{p^{1+s+u+w_{\barq,j}}} - \sum_{q=1}^d \sum_{i=1}^{\ell_q} \frac{\barL_d+1}{p^{1+s+u+z_{q,i}}} \nonumber \\
&  \quad \quad \, + \sum_{q=1}^d \sum_{i=1}^{\ell_q} \frac{1}{p^{1+\alpha+s+z_{q,i}}} + \sum_{\barq=1}^d \sum_{j=1}^{\barell_\barq} \frac{1}{p^{1+\beta+u+w_{\barq,j}}} + \sum_{q=1}^d \sum_{i=1}^{\ell_q} \sum_{\barq=1}^d \sum_{j=1}^{\barell_\barq} \frac{1}{p^{1+s+u+z_{q,i}+w_{\barq,j}}} \bigg] \bigg\}.
\end{align}
We remark that
\begin{align*}
A_{\alpha,\beta}(\mathbf{z},\mathbf{w};\beta,\alpha) &= \prod_p \bigg\{ \bigg( 1-\frac{1}{p^{1+\alpha+\beta}}\bigg)^{L_d \barL_d - 1} \bigg( \prod_{q=1}^d \prod_{i=1}^{\ell_q} \prod_{\barq=1}^{d} \prod_{j=1}^{\barell_\barq} \bigg( 1-\frac{1}{p^{1+\alpha+\beta+z_{q,i}+w_{\barq,j}}}\bigg) \bigg) \nonumber \\
& \times \bigg(\prod_{q=1}^d \prod_{i=1}^{\ell_q} \frac{1}{(1-\tfrac{1}{p^{1+\alpha+\beta+z_{q,i}}})^{\barL_d}}\bigg) \bigg(\prod_{\barq=1}^d \prod_{j=1}^{\barell_\barq} \frac{1}{(1-\tfrac{1}{p^{1+\alpha+\beta+w_{\barq,j}}})^{L_d}} \bigg) \nonumber \\
& \times \bigg[ 1-\frac{L_d+1}{p^{1+\beta+\alpha}} -\frac{\barL_d+1}{p^{1+\alpha+\beta}} + \frac{(L_d+1)(\barL_d+1)}{p^{1+\alpha+\beta}} - \sum_{\barq=1}^d \sum_{j=1}^{\barell_\barq} \frac{L_d+1}{p^{1+\alpha+\beta+w_{\barq,j}}} - \sum_{q=1}^d \sum_{i=1}^{\ell_q} \frac{\barL_d+1}{p^{1+\alpha+\beta+z_{q,i}}} \nonumber \\
&  \quad \quad \, + \sum_{q=1}^d \sum_{i=1}^{\ell_q} \frac{1}{p^{1+\beta+\alpha+z_{q,i}}} + \sum_{\barq=1}^d \sum_{j=1}^{\barell_\barq} \frac{1}{p^{1+\alpha+\beta+w_{\barq,j}}} + \sum_{q=1}^d \sum_{i=1}^{\ell_q} \sum_{\barq=1}^d \sum_{j=1}^{\barell_\barq} \frac{1}{p^{1+\alpha+\beta+z_{q,i}+w_{\barq,j}}} \bigg] \bigg\}.
\end{align*}
Hence at $\mathbf{z}=\mathbf{w}=\mathbf{0}$ we obtain
\begin{align*} 
A_{\alpha,\beta}(\mathbf{0},\mathbf{0};\beta,\alpha) %&= \prod_p \bigg\{ \bigg( 1-\frac{1}{p^{1+\alpha+\beta}}\bigg)^{L_d \barL_d - 1} \bigg( \prod_{q=1}^d \prod_{i=1}^{\ell_q} \prod_{\barq=1}^{d} \prod_{j=1}^{\barell_\barq} \bigg( 1-\frac{1}{p^{1+\alpha+\beta}}\bigg) \bigg) \nonumber \\
%& \times \bigg(\prod_{q=1}^d \prod_{i=1}^{\ell_q} \frac{1}{(1-\tfrac{1}{p^{1+\alpha+\beta}})^{\barL_d}}\bigg) \bigg(\prod_{\barq=1}^d \prod_{j=1}^{\barell_\barq} \frac{1}{(1-\tfrac{1}{p^{1+\alpha+\beta}})^{L_d}} \bigg) \nonumber \\
%& \times \bigg[ 1-\frac{L_d+1}{p^{1+\beta+\alpha}} -\frac{\barL_d+1}{p^{1+\alpha+\beta}} + \frac{(L_d+1)(\barL_d+1)}{p^{1+\alpha+\beta}} - \sum_{\barq=1}^d \sum_{j=1}^{\barell_\barq} \frac{L_d+1}{p^{1+\alpha+\beta}} - \sum_{q=1}^d \sum_{i=1}^{\ell_q} \frac{\barL_d+1}{p^{1+\alpha+\beta}} \nonumber \\
%&  \quad \quad \, + \sum_{q=1}^d \sum_{i=1}^{\ell_q} \frac{1}{p^{1+\beta+\alpha}} + \sum_{\barq=1}^d \sum_{j=1}^{\barell_\barq} \frac{1}{p^{1+\alpha+\beta}} + \sum_{q=1}^d \sum_{i=1}^{\ell_q} \sum_{\barq=1}^d \sum_{j=1}^{\barell_\barq} \frac{1}{p^{1+\alpha+\beta}} \bigg] \bigg\} \nonumber \\
& = \prod_p \bigg\{ \bigg(1-\frac{1}{p^{1+\alpha+\beta}} \bigg)^{-1} \bigg[1-\frac{1}{p^{1+\alpha+\beta}} \bigg] \bigg\} = 1.
\end{align*}
Finally, we may define $\zeta(1+\alpha+\beta) \times (\operatorname{integrand})$ to be
\begin{align} \label{generalintegrandM}
\mathfrak{M}_{\alpha,\beta}(\mathbf{z},\mathbf{w}; s, u) &:= \frac{\zeta(1+s+u)^{(L_d+1)(\barL_d+1)}\zeta(1+\alpha+\beta)}{\zeta(1+\alpha+s)^{L_d+1}\zeta(1+\beta+u)^{\barL_d+1}} A_{\alpha,\beta}(\mathbf{z},\mathbf{w}; s, u) \nonumber \\ 
& \quad  \times \bigg(\prod_{q=1}^d \prod_{i=1}^{\ell_q} \prod_{\barq=1}^d \prod_{j=1}^{\barell_\barq} \zeta(1+s+u+z_{q,i}+w_{\barq,j}) \bigg) \nonumber \\
& \quad  \times \bigg(\prod_{q=1}^d \prod_{i=1}^{\ell_q} \frac{\zeta(1+\alpha+s+z_{q,i})}{\zeta(1+s+u+z_{q,i})^{\barL_d+1}}\bigg) \bigg( \prod_{\barq=1}^d \prod_{j=1}^{\barell_\barq} \frac{\zeta(1+\beta+u+w_{\barq,j})}{\zeta(1+s+u+w_{\barq,j})^{L_d+1}} \bigg),
\end{align}
and observe that
\begin{align*}
\mathfrak{M}_{\alpha,\beta}(\mathbf{z},\mathbf{w}; \beta, \alpha) &:= \zeta(1+\alpha+\beta)^{L_d \barL_d} A_{\alpha,\beta}(\mathbf{z},\mathbf{w}; \beta, \alpha) \nonumber \\ 
& \quad  \times \bigg(\prod_{q=1}^d \prod_{i=1}^{\ell_q} \prod_{\barq=1}^d \prod_{j=1}^{\barell_\barq} \zeta(1+\alpha+\beta+z_{q,i}+w_{\barq,j}) \bigg) \nonumber \\
& \quad  \times \bigg(\prod_{q=1}^d \prod_{i=1}^{\ell_q} \frac{1}{\zeta(1+\alpha+\beta+z_{q,i})^{\barL_d}}\bigg) \bigg( \prod_{\barq=1}^d \prod_{j=1}^{\barell_\barq} \frac{1}{\zeta(1+\alpha+\beta+w_{\barq,j})^{L_d}} \bigg),
\end{align*}
%\begin{align*}
%\mathfrak{M}_{\alpha,\beta}(\mathbf{0},\mathbf{0}; \beta, \alpha) &= \zeta(1+\alpha+\beta)^{L_d \barL_d} A_{\alpha,\beta}(\mathbf{0},\mathbf{0}; \beta, \alpha)  \bigg(\prod_{q=1}^d \prod_{i=1}^{\ell_q} \prod_{\barq=1}^d \prod_{j=1}^{\barell_\barq} \zeta(1+\alpha+\beta) \bigg) \nonumber \\
%& \quad  \times \bigg(\prod_{q=1}^d \prod_{i=1}^{\ell_q} \frac{1}{\zeta(1+\alpha+\beta)^{\barL_d}}\bigg) \bigg( \prod_{\barq=1}^d \prod_{j=1}^{\barell_\barq} \frac{}{\zeta(1+\alpha+\beta)^{L_d}} \bigg) = 1.
%\end{align*}
so that $\mathfrak{M}_{\alpha,\beta}(\mathbf{0},\mathbf{0}; \beta, \alpha)=1$. Indeed the formulas
\begin{align} \label{0pointvalue1}
A_{\alpha,\beta}(\mathbf{0},\mathbf{0};\beta,\alpha) = \mathfrak{M}_{\alpha,\beta}(\mathbf{0},\mathbf{0}; \beta, \alpha) = 1
\end{align}
are key properties that will be used frequently. We summarize this result below.

\begin{theorem} \label{theoremmaintermerror47}
Let $\mathfrak{a}_d$ and $\mathfrak{b}_e$ be given by \eqref{specialgeneral1} and \eqref{specialgeneral2}, respectively and set
\begin{align*}
\mathscr{D}X = \mathscr{D}X_{d,\ell_d} := \frac{dx_{1,1}}{x_{1,1}^{1+1}} \cdots \frac{dx_{1,\ell_1}}{x_{1,\ell_1}^{1+1}} \frac{dx_{2,1}}{x_{2,1}^{2+1}} \cdots \frac{dx_{2,\ell_2}}{x_{2,\ell_2}^{2+1}} \cdots \frac{dx_{d,1}}{x_{d,1}^{d+1}} \cdots \frac{dx_{d,\ell_d}}{x_{d,\ell_d}^{d+1}}.
\end{align*}
One has that
\begin{align*}
& \sumtwo_{1 \le d,e \le \infty} \frac{\mathfrak{a}_d \overline{\mathfrak{b}_e}}{[d,e]} \frac{(d,e)^{\alpha+\beta}}{d^{\alpha+s} e^{\beta+u}} P_{d,\ell_d}\bigg( \frac{\log(N/d)}{\log N} \bigg) P_{d,\barell_d}\bigg( \frac{\log(N/e)}{\log N} \bigg)\nonumber \\
& \quad \quad \quad \quad \times \int_{-\infty}^\infty \bigg(\zeta(1+\alpha+\beta)+\zeta(1-\alpha-\beta)\bigg(\frac{2\pi de}{t(d,e)^2}\bigg)^{\alpha+\beta} \bigg)\Phi\bigg(\frac{t}{T}\bigg)dt + O(\mathcal{E}_3) \\
& =  \sumtwo_{i,j} \frac{p_{d,\ell_d,i}p_{d,{\bar \ell}_d,j}i!j!}{(\log N)^{i+j}}\frac{1}{(2\pi i)^2}\int_{(1)}\int_{(1)} \frac{1}{(\log N)^{\sum r\ell_r} (\log N)^{\sum \barr\barell_\barr}} \\
& \times  \frac{1}{(2\pi i)^{L_d+\barL_d}}\bigg( \oint \oint \mathfrak{M}_{\alpha,\beta}(\mathbf{z},\mathbf{w};s,u) \mathscr{D}Z \mathscr{D}W  \\
& \quad \quad \quad \quad - \oint \oint T^{-\alpha-\beta}\mathfrak{M}_{-\beta,-\alpha}(\mathbf{z},\mathbf{w};s,u) \mathscr{D}Z \mathscr{D}W \bigg) \frac{ds}{s^{i+1}}\frac{du}{u^{j+1}} + O(T^{1-\varepsilon}),
\end{align*}
where $\mathfrak{M}_{\alpha,\beta}$ is given by \eqref{generalintegrandM}.
\end{theorem}

Again, this is as far as one can go before things get out of hand and in fact it is bad enough as it is already. The procedure to obtain the main terms is exactly the same as it was in the case $d=1$. We give the general recipe below before we proceed.
\begin{enumerate}
\item[(1)] Choose the degree $d$ of the polynomial $Q$. The higher the degree, the more taxing the calculation.
\item[(2)] Choose a truncation $\ell, \barell$. Once more, the higher the truncation, the more taxing the calculation.
\item[(3)] Use the logarithmic derivative of \eqref{Arithemticcleangeneral} in conjunction with Fa\`{a} di Bruno's formula \eqref{faadibruno} and with \eqref{0pointvalue1} to isolate the derivatives of $A_{\alpha,\beta}(\{z_{q,i}\}, \{w_{\barq,j}\}; s, u)$ at $s=\alpha$, $u=\beta$ and $\{z_{q,i}\}=\{w_{\barq,j}\}=0$. Most of the derivatives will be $0$ and the remaining ones will go into an error term of size $\ll T/L$. In other words, only the terms that have a coefficient $A_{\alpha,\beta}(\mathbf{0},\mathbf{0};s,u)$ will survive.
\item[(4)] Insert the results back into $I_1$ (recall that $I_2$ is basically the same as $I_1$ up to symmetries) and deform the paths of integration to $\real(s)=\real(u)=\delta$ with $\delta>0$ then take $\delta \asymp L^{-1}$ and bound trivially. This allows us to change $A_{\alpha,\beta}(\mathbf{0},\mathbf{0};s,u)$ into $A_{\alpha,\beta}(\mathbf{0},\mathbf{0};\beta,\alpha)$ with an acceptable error term.
\item[(5)] Use Dirichlet convolutions of appropriate $\mathbf{1} \star \Lambda_1^{k_1} \star \Lambda_2^{k_2} \star \cdots \star \Lambda_d^{k_d}$ to separate the complex variables $s$ and $u$ along with knowledge of \eqref{0pointvalue1}. The truncation of the $n$-sum will be at $N$.
\item[(6)] Divide and conquer strategy: separate the resulting terms and apply Lemma \ref{contourlemma} below to each one of them.
\item[(7)] Sum over $i$ and $j$ to recover the polynomials $P_{d,\ell}$ and $P_{d,\barell}$.
\item[(7)] Use Euler-Maclaurin result (Lemma \ref{eulermaclaurinlemma}) below to turn the remaining sums involving arithmetical functions into integrals.
\end{enumerate}

\begin{remark}
If we look at the nine cases that we considered earlier in this section to get to the autocorrelation ratio of $\mathcal{S}_d$, it is clear  that each of these cases would remain unaffected by the presence of $\mu^2(d)$ and $\mu^2(e)$. Indeed each case would be multiplied by either $\mu^2(p^0)=\mu^2(1)=1$ or $\mu^2(p^1)=\mu^2(p)=(-1)^2=1$, where $p$ is always a prime. We will be explicit about this important point. If we had taken
\begin{align} \label{alternativemu2}
\mathfrak{a}_n = \frac{\mu^2(n)(\mu \star \Lambda_1^{\star \ell_1} \star \Lambda_2^{\star \ell_2} \star \cdots \star \Lambda_d^{\star \ell_d})(n)}{(\log N)^{\sum r \ell_r}},
\end{align}
instead of \eqref{specialgeneral1}, then the affected possibilities (once symmetries are accounted for) in the Euler product are
\begin{enumerate}
\item[(1)] The case $\{0,0,0,0\}$ yields
\begin{align*}
\frac{\mu^2(1)\mu^2(1)}{1 \times 1} = 1.
\end{align*}
\item[(2)] The case $\{1,0,0,0\}$ yields 
\begin{align*}
\frac{\mu^2(p) \mu^{\star L_d+1}(p) \mu^2(1)}{[p,1]} \frac{(p,1)^{\alpha+\beta}}{p^{\alpha+s}} = - \frac{L_d+1}{p^{1+\alpha+s}}.
\end{align*}
\item[(4)] The case $\{1,0,1,0\}$ yields
\begin{align*}
\frac{\mu^2(p)\mu^{\star L_d+1}(p)\mu^2(p)\mu^{\star \barL_d+1}(p)}{[p,p]} \frac{(p,p)^{\alpha+\beta}}{p^{\alpha+s}p^{\beta+u}} = \frac{(L_d+1)(\barL_d+1)}{p^{1+s+u}}.
\end{align*}
\item[(5)] The case $\{1,0,0,1\}$ yields 
\begin{align*}
\frac{\mu^2(p)\mu^{\star L_d+1}(p)\mu^2(1)}{[p,p]} \frac{(p,p)^{\alpha+\beta}}{p^{\alpha+s}p^{\beta+u+w_{\barq,j}}} = -\frac{L_d+1}{p^{1+s+u+w_{\barq,j}}}.
\end{align*}
\item[(7)] The case $\{0,1,0,0\}$ yields
\begin{align*}
\frac{\mu^2(1)\mu^2(1)}{[p,1]} \frac{(p,1)^{\alpha+\beta}}{p^{\alpha+s+z_{q,i}}} = \frac{1}{p^{1+\alpha+s+z_{q,i}}}.
\end{align*}
\item[(9)] The case $\{0,1,0,1\}$ yields
\begin{align*}
\frac{\mu^2(1)\mu^2(1)}{[p,p]} \frac{(p,p)^{\alpha+\beta}}{p^{\beta+u+w_{\barq,j}} p^{\alpha+s+z_{q,i}}} = \frac{1}{p^{1+s+u+z_{q,i}+w_{\barq,j}}}.
\end{align*}
\end{enumerate}
Moreover, the arithmetical factor that would arise from \eqref{alternativemu2} also behaves nicely.\\

It is likely that the present authors would have not noticed this had they not used the autocorrelation of ratios approach. Indeed this shows that the powerful autocorrelation of ratios technique is really the `way to do things'.\\

In any case, Feng's approach is certainly commendable for his main terms are accurate and his outstanding intuition about the size of the mollifier was spot on. 
\end{remark}

\subsection{Example: the quadratic case $d=2$}

We have illustrated the mechanism of this twisted moment in its easiest manifestation, that is when $d=1$ and when the truncations are small, i.e. $\ell_1, \ell_2 \le 2$. Now that we have given the general procedure, it is advisable to see what happens when the degree of $Q$ is increased from $1$ to $2$. Indeed, this will make the above presentation more agreeable and in fact these terms have never before appeared in the literature. 

With $d=2$, we have $L_d = \ell_1 + \ell_2$ and $\barL_d = \barell_1 + \barell_2$. Let us set $\ell_1 = \ell_2 = \barell_1 = \barell_2 = 1$ for simplicity and use $\widetilde{\mathfrak{M}}_{\alpha,\beta}(z_{1,1},z_{2,1},w_{1,1},w_{2,1};s,u)$ to denote $\mathfrak{M}_{\alpha,\beta}(z_{1,1},z_{2,1},w_{1,1},w_{2,1};s,u)$ with the arithmetical factor $A_{\alpha,\beta}$ replaced by $1$. 

We start with the arithmetical product. We can compute the logarithm from \eqref{Arithemticcleangeneral} of $A_{\alpha,\beta}$ so that
\begin{align} \label{logarithmgeneralarithmetic}
\log A_{\alpha,\beta}(\mathbf{z},\mathbf{w};s,u) = \sum_p & \bigg\{ \log \frac{(1-\tfrac{1}{p^{1+s+u}})^{(L_d+1)(\barL_d+1)}}{(1- \tfrac{1}{p^{1+\alpha+s}})^{(L_d+1)}(1- \tfrac{1}{p^{1+\beta+u}})^{(\barL_d+1)}} \nonumber \\
& +  \sum_{q=1}^d \sum_{i=1}^{\ell_q} \sum_{\barq=1}^{d} \sum_{j=1}^{\barell_\barq} \log \bigg( 1-\frac{1}{p^{1+s+u+z_{q,i}+w_{\barq,j}}}\bigg) \nonumber \\
& + \sum_{q=1}^d \sum_{i=1}^{\ell_q} \frac{1-\tfrac{1}{p^{1+\alpha+s+z_{q,i}}}}{(1-\tfrac{1}{p^{1+s+u+z_{q,i}}})^{(\barL_d+1)}} + \sum_{\barq=1}^d \sum_{j=1}^{\barell_\barq} \frac{1-\tfrac{1}{p^{1+\beta+u+w_{\barq,j}}}}{(1-\tfrac{1}{p^{1+s+u+w_{\barq,j}}})^{(L_d+1)}} \nonumber \\
& + \log \bigg[ 1-\frac{L_d+1}{p^{1+\alpha+s}} -\frac{\barL_d+1}{p^{1+\beta+u}} + \frac{(L_d+1)(\barL_d+1)}{p^{1+s+u}} \nonumber \\
&  \quad \quad \ - \sum_{\barq=1}^d \sum_{j=1}^{\barell_\barq} \frac{L_d+1}{p^{1+s+u+w_{\barq,j}}} - \sum_{q=1}^d \sum_{i=1}^{\ell_q} \frac{\barL_d+1}{p^{1+s+u+z_{q,i}}} \nonumber \\
&  \quad \quad \, + \sum_{q=1}^d \sum_{i=1}^{\ell_q} \frac{1}{p^{1+\alpha+s+z_{q,i}}} + \sum_{\barq=1}^d \sum_{j=1}^{\barell_\barq} \frac{1}{p^{1+\beta+u+w_{\barq,j}}} \nonumber \\
&  \quad \quad \ + \sum_{q=1}^d \sum_{i=1}^{\ell_q} \sum_{\barq=1}^d \sum_{j=1}^{\barell_\barq} \frac{1}{p^{1+s+u+z_{q,i}+w_{\barq,j}}} \bigg] \bigg\}.
\end{align}

Using \eqref{logarithmgeneralarithmetic} it is straightforward to show that
\begin{align*}
\frac{\partial}{\partial z_{1,1}} \log A_{\alpha,\beta}(\mathbf{z},\mathbf{w};\beta,\alpha) \bigg|_{\mathbf{z}=\mathbf{w}=\mathbf{0}} &= 0, \\
\frac{\partial^2}{\partial z_{2,1}^2} \log A_{\alpha,\beta}(\mathbf{z},\mathbf{w};\beta,\alpha) \bigg|_{\mathbf{z}=\mathbf{w}=\mathbf{0}} &= 0, \\
\frac{\partial^3}{\partial z_{1,1} \partial z_{2,1}^2} \log A_{\alpha,\beta}(\mathbf{z},\mathbf{w};\beta,\alpha) \bigg|_{\mathbf{z}=\mathbf{w}=\mathbf{0}} &= 0, \\
\frac{\partial^2}{\partial z_{1,1}\partial w_{1,1}} \log A_{\alpha,\beta}(\mathbf{z},\mathbf{w};\beta,\alpha) \bigg|_{\mathbf{z}=\mathbf{w}=\mathbf{0}} &= \sum_p \bigg(\frac{\log^2 p}{p^{1+\alpha+\beta}}-\frac{p^{1+\alpha+\beta} \log^2 p}{(p^{1+\alpha+\beta}-1)^2}\bigg), \\
\frac{\partial^3}{\partial z_{1,1}\partial w_{2,1}^2} \log A_{\alpha,\beta}(\mathbf{z},\mathbf{w};\beta,\alpha) \bigg|_{\mathbf{z}=\mathbf{w}=\mathbf{0}} &= \sum_p \bigg(\frac{p^{1+\alpha+\beta}(1+p^{1+\alpha+\beta})\log^3 p}{(p^{1+\alpha+\beta}-1)^3} - \frac{\log^3 p}{p^{1+\alpha+\beta}}\bigg), \\
\frac{\partial^4}{\partial z_{2,1}^2\partial w_{2,1}^2} \log A_{\alpha,\beta}(\mathbf{z},\mathbf{w};\beta,\alpha) \bigg|_{\mathbf{z}=\mathbf{w}=\mathbf{0}} &= \sum_p \bigg(\frac{\log^4 p}{p^{1+\alpha+\beta}} - \frac{p^{1+\alpha+\beta}(1+p^{1+\alpha+\beta}(4+p^{1+\alpha+\beta})) \log^4 p}{(p^{1+\alpha+\beta}-1)^4} \bigg), \\
\frac{\partial^4}{\partial z_{1,1}\partial z_{2,1}^2 \partial w_{1,1}} \log A_{\alpha,\beta}(\mathbf{z},\mathbf{w};\beta,\alpha) \bigg|_{\mathbf{z}=\mathbf{w}=\mathbf{0}} &= 0, \\
\frac{\partial^5}{\partial z_{1,1}\partial z_{2,1}^2 \partial w_{1,1} \partial w_{2,1}} \log A_{\alpha,\beta}(\mathbf{z},\mathbf{w};\beta,\alpha) \bigg|_{\mathbf{z}=\mathbf{w}=\mathbf{0}} &= 0, \\
\frac{\partial^6}{\partial z_{1,1}\partial z_{2,1}^2 \partial w_{1,1} \partial w_{2,1}^2} \log A_{\alpha,\beta}(\mathbf{z},\mathbf{w};\beta,\alpha) \bigg|_{\mathbf{z}=\mathbf{w}=\mathbf{0}} &= 0 ,
\end{align*}
etc... The pattern as to what terms will survive becomes clearer (some of the `balanced' mixed derivatives). However, it does not really matter because these surviving terms will get absorbed in an error term of size $O(T/L)$ much like in the case $d=1$ that we handled previously.

A lengthy residue calculus computation, again we are paying the toll for having used \eqref{cauchydgeneral}, shows that (\texttt{Mathematica 11.2} was used to produce this)
\begin{align*}
& \oint \oint \oint \oint \widetilde{\mathfrak{M}}_{\alpha,\beta}(z_{1,1},z_{2,1},w_{1,1},w_{2,1};s,u) \frac{dz_{1,1}}{z_{1,1}^2}\frac{dz_{2,1}}{z_{2,1}^3}\frac{dw_{1,1}}{w_{1,1}^2}\frac{dw_{2,1}}{w_{2,1}^3} = \frac{\zeta(1+s+u)\zeta(1+\alpha+\beta)}{\zeta (1+\alpha+s) \zeta (1+\beta+u)}  \\
& \times \bigg\{ 12 \frac{\zeta '(1+s+u)^6}{\zeta
   (1+s+u)^6}+12 \frac{\zeta '(1+\alpha+s) \zeta '(1+s+u)^5}{\zeta (1+\alpha+s)\zeta (1+s+u)^5} \\
& \quad +12\frac{ \zeta '(1+\beta+u) \zeta
   '(1+s+u)^5}{\zeta (1+\beta+u) \zeta (1+s+u)^5}-8\frac{
   \zeta '(1+\alpha+s)^2 \zeta '(1+s+u)^4}{\zeta (1+\alpha+s)^2 \zeta
   (1+s+u)^4} \\
& \quad -8 \frac{\zeta '(1+\beta+u)^2 \zeta '(1+s+u)^4}{
   \zeta (1+\beta+u)^2 \zeta (1+s+u)^4}-32 \frac{\zeta '(1+\alpha+s) \zeta
   '(1+\beta+u) \zeta '(1+s+u)^4}{\zeta (1+\alpha+s) \zeta (1+\beta+u) \zeta
   (1+s+u)^4} \\
& \quad -4 \frac{\zeta ''(1+\alpha+s) \zeta '(1+s+u)^4}{\zeta (1+\alpha+s)^2
   \zeta (1+\beta+u) \zeta (1+s+u)^3}-4 \frac{\zeta ''(1+\beta+u) \zeta
   '(1+s+u)^4}{\zeta (1+\beta+u) \zeta (1+s+u)^4} \\
& \quad -48\frac{
   \zeta ''(1+s+u) \zeta '(1+s+u)^4}{\zeta
   (1+s+u)^5}+8 \frac{\zeta '(1+\alpha+s) \zeta '(1+\beta+u)^2 \zeta
   '(1+s+u)^3}{\zeta (1+\alpha+s) \zeta (1+\beta+u)^2 \zeta (1+s+u)^3} \\
& \quad +8\frac{
   \zeta '(1+\alpha+s)^2 \zeta '(1+\beta+u) \zeta '(1+s+u)^3}{\zeta (1+\alpha+s)^2
   \zeta (1+\beta+u) \zeta (1+s+u)^3}-2 \frac{\zeta '(1+\alpha+s) \zeta
   ''(1+\alpha+s) \zeta '(1+s+u)^3}{\zeta (1+\alpha+s)^2 \zeta
   (1+s+u)^3} \\
& \quad +4 \frac{\zeta '(1+\beta+u) \zeta ''(1+\alpha+s) \zeta
   '(1+s+u)^3}{\zeta (1+\alpha+s) \zeta (1+\beta+u) \zeta (1+s+u)^3}+4\frac{
   \zeta '(1+\alpha+s) \zeta ''(1+\beta+u) \zeta '(1+s+u)^3}{\zeta (1+\alpha+s) \zeta
   (1+\beta+u) \zeta (1+s+u)^3} \\
& \quad -2 \frac{\zeta '(1+\beta+u) \zeta ''(1+\beta+u)
   \zeta '(1+s+u)^3}{\zeta (1+\beta+u)^2 \zeta
   (1+s+u)^3}-2 \frac{\zeta '(1+\alpha+s) \zeta ''(1+s+u) \zeta
   '(1+s+u)^3}{\zeta (1+\alpha+s) \zeta (1+s+u)^4} \\
& \quad -2\frac{
   \zeta '(1+\beta+u) \zeta ''(1+s+u) \zeta '(1+s+u)^3}{\zeta
   (1+\beta+u) \zeta (1+s+u)^4}+8 \frac{\zeta ^{(3)}(1+s+u) \zeta
   '(1+s+u)^3}{\zeta (1+s+u)^4} \\
& \quad +43\frac{
   \zeta ''(1+s+u)^2 \zeta '(1+s+u)^2}{\zeta
   (1+s+u)^4}+4 \frac{\zeta '(1+\alpha+s) \zeta '(1+\beta+u) \zeta ''(1+\alpha+s) \zeta
   '(1+s+u)^2}{\zeta (1+\alpha+s)^2 \zeta (1+\beta+u) \zeta (1+s+u)^2} \\
& \quad +4\frac{
   \zeta '(1+\alpha+s) \zeta '(1+\beta+u) \zeta ''(1+\beta+u) \zeta '(1+s+u)^2}{\zeta
   (1+\alpha+s) \zeta (1+\beta+u)^2 \zeta (1+s+u)^2} \\
& \quad +14 \frac{\zeta '(1+\alpha+s)^2
   \zeta ''(1+s+u) \zeta '(1+s+u)^2}{\zeta (1+\alpha+s)^2 \zeta
   (1+s+u)^3} \\
& \quad +14\frac{ \zeta '(1+\beta+u)^2 \zeta ''(1+s+u) \zeta
   '(1+s+u)^2}{\zeta (1+\beta+u)^2 \zeta (1+s+u)^3} \\
& \quad +40\frac{
   \zeta '(1+\alpha+s) \zeta '(1+\beta+u) \zeta ''(1+s+u) \zeta '(1+s+u)^2}{\zeta
   (1+\alpha+s) \zeta (1+\beta+u) \zeta (1+s+u)^3} \\
& \quad +7 \frac{\zeta ''(1+\alpha+s)
   \zeta ''(1+s+u) \zeta '(1+s+u)^2}{\zeta (1+\alpha+s) \zeta
   (1+s+u)^3}+7 \frac{\zeta ''(1+\beta+u) \zeta ''(1+s+u) \zeta
   '(1+s+u)^2}{\zeta (1+\beta+u) \zeta (1+s+u)^3} \\
& \quad + 8 \frac{
   \zeta '(1+\alpha+s) \zeta ^{(3)}(1+s+u) \zeta '(1+s+u)^2}{\zeta (1+\alpha+s)
   \zeta (1+s+u)^3}+8 \frac{\zeta '(1+\beta+u) \zeta
   ^{(3)}(1+s+u) \zeta '(1+s+u)^2}{\zeta (1+\beta+u) \zeta
   (1+s+u)^3} \\
& \quad -25 \frac{\zeta '(1+\alpha+s) \zeta ''(1+s+u)^2 \zeta
   '(1+s+u)}{\zeta (1+\alpha+s) \zeta (1+s+u)^3}-25 \frac{
   \zeta '(1+\beta+u) \zeta ''(1+s+u)^2 \zeta '(1+s+u)}{\zeta
   (1+\beta+u) \zeta (1+s+u)^3} \\
& \quad -2 \frac{\zeta '(1+\alpha+s) \zeta '(1+\beta+u)^2
   \zeta ''(1+\alpha+s) \zeta '(1+s+u)}{\zeta (1+\alpha+s)^2 \zeta
   (1+\beta+u)^2 \zeta(1+s+u)} \\
& \quad -2 \frac{\zeta '(1+\alpha+s)^2 \zeta '(1+\beta+u) \zeta ''(1+\beta+u)
   \zeta '(1+s+u)}{\zeta (1+\alpha+s)^2 \zeta (1+\beta+u)^2 \zeta(1+s+u)} \\
& \quad -\frac{\zeta '(1+\alpha+s)
   \zeta ''(1+\alpha+s) \zeta ''(1+\beta+u) \zeta '(1+s+u)}{\zeta (1+\alpha+s)^2 \zeta
   (1+\beta+u) \zeta(1+s+u)} \\
& \quad -\frac{\zeta '(1+\beta+u) \zeta ''(1+\alpha+s) \zeta ''(1+\beta+u) \zeta
   '(1+s+u)}{\zeta (1+\alpha+s) \zeta (1+\beta+u)^2 \zeta(1+s+u)} \\
& \quad -16 \frac{\zeta '(1+\alpha+s)
   \zeta '(1+\beta+u)^2 \zeta ''(1+s+u) \zeta '(1+s+u)}{\zeta (1+\alpha+s) \zeta
   (1+\beta+u)^2 \zeta (1+s+u)^2} \\
& \quad -16 \frac{\zeta '(1+\alpha+s)^2 \zeta '(1+\beta+u)
   \zeta ''(1+s+u) \zeta '(1+s+u)}{\zeta (1+\alpha+s)^2 \zeta (1+\beta+u) \zeta
   (1+s+u)^2} \\
& \quad +\frac{\zeta '(1+\alpha+s) \zeta ''(1+\alpha+s) \zeta ''(1+s+u) \zeta
   '(1+s+u)}{\zeta (1+\alpha+s)^2 \zeta (1+s+u)^2} \\
& \quad -8 \frac{\zeta
   '(1+\beta+u) \zeta ''(1+\alpha+s) \zeta ''(1+s+u) \zeta '(1+s+u)}{\zeta
   (1+\alpha+s) \zeta (1+\beta+u) \zeta (1+s+u)^2} \\
& \quad -8 \frac{\zeta '(1+\alpha+s) \zeta
   ''(1+\beta+u) \zeta ''(1+s+u) \zeta '(1+s+u)}{\zeta (1+\alpha+s) \zeta
   (1+\beta+u) \zeta (1+s+u)^2} \\
& \quad +\frac{\zeta '(1+\beta+u) \zeta ''(1+\beta+u) \zeta
   ''(1+s+u) \zeta '(1+s+u)}{\zeta (1+\beta+u)^2 \zeta
   (1+s+u)^2} \\
& \quad -2 \frac{\zeta '(1+\alpha+s)^2 \zeta ^{(3)}(1+s+u) \zeta
   '(1+s+u)}{\zeta (1+\alpha+s)^2 \zeta (1+s+u)^2}-2 \frac{\zeta
   '(1+\beta+u)^2 \zeta ^{(3)}(1+s+u) \zeta '(1+s+u)}{\zeta
   (1+\beta+u)^2 \zeta (1+s+u)^2} \\
& \quad -16 \frac{\zeta '(1+\alpha+s) \zeta '(1+\beta+u) \zeta
   ^{(3)}(1+s+u) \zeta '(1+s+u)}{\zeta (1+\alpha+s) \zeta (1+\beta+u) \zeta
   (1+s+u)^2} \\
& \quad -\frac{\zeta ''(1+\alpha+s) \zeta ^{(3)}(1+s+u) \zeta
   '(1+s+u)}{\zeta (1+\alpha+s) \zeta (1+s+u)^2} \\
& \quad -\frac{\zeta
   ''(1+\beta+u) \zeta ^{(3)}(1+s+u) \zeta '(1+s+u)}{\zeta
   (1+\beta+u) \zeta (1+s+u)^2}-20 \frac{\zeta ''(1+s+u) \zeta ^{(3)}(1+s+u)
   \zeta '(1+s+u)}{\zeta
   (1+s+u)^3} \\
& \quad -\frac{\zeta '(1+\alpha+s) \zeta ^{(4)}(1+s+u) \zeta
   '(1+s+u)}{\zeta (1+\alpha+s) \zeta (1+s+u)^2}-\frac{\zeta
   '(1+\beta+u) \zeta ^{(4)}(1+s+u) \zeta '(1+s+u)}{\zeta
   (1+\beta+u) \zeta (1+s+u)^2} \\
& \quad +4 \frac{\zeta ''(1+s+u)^3}{\zeta (1+s+u)^3}-\frac{2 \zeta '(1+\alpha+s)^2 \zeta
   ''(1+s+u)^2}{\zeta (1+\alpha+s)^2 \zeta (1+s+u)^2} \\
& \quad -2\frac{
   \zeta '(1+\beta+u)^2 \zeta ''(1+s+u)^2}{\zeta (1+\beta+u)^2
   \zeta (1+s+u)^2}+16 \frac{\zeta '(1+\alpha+s) \zeta '(1+\beta+u) \zeta
   ''(1+s+u)^2}{\zeta (1+\alpha+s) \zeta (1+\beta+u) \zeta
   (1+s+u)^2} \\
& \quad -\frac{\zeta ''(1+\alpha+s) \zeta ''(1+s+u)^2}{\zeta (1+\alpha+s)
   \zeta (1+s+u)^2}-\frac{\zeta ''(1+\beta+u) \zeta
   ''(1+s+u)^2}{\zeta (1+\beta+u) \zeta (1+s+u)^2} \\
& \quad +\frac{\zeta
   ^{(3)}(1+s+u)^2}{\zeta
   (1+s+u)^2}+\frac{\zeta (1+s+u) \zeta '(1+\alpha+s) \zeta '(1+\beta+u) \zeta
   ''(1+\alpha+s) \zeta ''(1+\beta+u)}{\zeta (1+\alpha+s)^2 \zeta (1+\beta+u)^2 \zeta(1+s+u)} \\
& \quad + 4 \frac{
   \zeta '(1+\alpha+s)^2 \zeta '(1+\beta+u)^2 \zeta ''(1+s+u)}{\zeta (1+\alpha+s)^2
   \zeta (1+\beta+u)^2 \zeta(1+s+u)}+2 \frac{\zeta '(1+\beta+u)^2 \zeta ''(1+\alpha+s) \zeta
   ''(1+s+u)}{\zeta (1+\alpha+s) \zeta (1+\beta+u)^2 \zeta(1+s+u)} \\
& \quad -\frac{\zeta '(1+\alpha+s) \zeta
   '(1+\beta+u) \zeta ''(1+\alpha+s) \zeta ''(1+s+u)}{\zeta (1+\alpha+s)^2 \zeta
   (1+\beta+u) \zeta(1+s+u)}+2 \frac{\zeta '(1+\alpha+s)^2 \zeta ''(1+\beta+u) \zeta
   ''(1+s+u)}{\zeta (1+\alpha+s)^2 \zeta (1+\beta+u) \zeta(1+s+u)} \\
& \quad -\frac{\zeta '(1+\alpha+s) \zeta
   '(1+\beta+u) \zeta ''(1+\beta+u) \zeta ''(1+s+u)}{\zeta (1+\alpha+s) \zeta
   (1+\beta+u)^2 \zeta(1+s+u)}+\frac{\zeta ''(1+\alpha+s) \zeta ''(1+\beta+u) \zeta
   ''(1+s+u)}{\zeta (1+\alpha+s) \zeta (1+\beta+u) \zeta(1+s+u)} \\
& \quad +2 \frac{\zeta '(1+\alpha+s)
   \zeta '(1+\beta+u)^2 \zeta ^{(3)}(1+s+u)}{\zeta (1+\alpha+s) \zeta
   (1+\beta+u)^2 \zeta(1+s+u)}+2 \frac{\zeta '(1+\alpha+s)^2 \zeta '(1+\beta+u) \zeta
   ^{(3)}(1+s+u)}{\zeta (1+\alpha+s)^2 \zeta (1+\beta+u) \zeta(1+s+u)} \\
& \quad +\frac{\zeta '(1+\beta+u)
   \zeta ''(1+\alpha+s) \zeta ^{(3)}(1+s+u)}{\zeta (1+\alpha+s) \zeta
   (1+\beta+u) \zeta(1+s+u)}+\frac{\zeta '(1+\alpha+s) \zeta ''(1+\beta+u) \zeta
   ^{(3)}(1+s+u)}{\zeta (1+\alpha+s) \zeta (1+\beta+u) \zeta(1+s+u)} \\
& \quad +5 \frac{\zeta '(1+\alpha+s)
   \zeta ''(1+s+u) \zeta ^{(3)}(1+s+u)}{\zeta (1+\alpha+s)
   \zeta (1+s+u)^2}+5 \frac{\zeta '(1+\beta+u) \zeta ''(1+s+u) \zeta
   ^{(3)}(1+s+u)}{\zeta (1+\beta+u) \zeta
   (1+s+u)^2} \\
& \quad +\frac{\zeta '(1+\alpha+s) \zeta '(1+\beta+u) \zeta
   ^{(4)}(1+s+u)}{\zeta (1+\alpha+s) \zeta (1+\beta+u) \zeta(1+s+u)}+\frac{\zeta ''(1+s+u)
   \zeta ^{(4)}(1+s+u)}{\zeta (1+s+u)^2} \bigg\}.
\end{align*}
Each of these cases is now treated as in $\mathsection$5.1 with the assistance of Lemma \ref{contourlemma} and Lemma \ref{eulermaclaurinlemma}, see below. A way to automate this process would be most welcome.

\subsection{Resuming the general case $d \ge 0$}

We shall finish our delineation of the main terms. An inspection shows that a general term of (using the compactified notation of Theorem \ref{theoremmaintermerror47})
\begin{align*}
\frac{1}{(2 \pi i)^{L_d+\barL_d}} \oint \oint \mathfrak{M}_{\alpha,\beta}(\mathbf{w},\mathbf{z};s,u)  \mathscr{D} Z \mathscr{D}W
\end{align*}
is of the form 
\begin{align} \label{generalterm}
& \frac{\zeta(1+s+u)\zeta(1+\alpha+\beta)}{\zeta(1+\alpha+s) \zeta(1+\beta+u)} \Psi(k_1, k_2, \cdots, k_{d}; l_1, l_2, \cdots, l_d; m_1, m_2, \cdots, m_{d}) \nonumber \\
& \times \bigg(\frac{\zeta'}{\zeta}(1+s+u)\bigg)^{k_1}\bigg(\frac{\zeta''}{\zeta}(1+s+u)\bigg)^{k_2} \times \cdots \times \bigg(\frac{\zeta^{(d+d)}}{\zeta}(1+s+u)\bigg)^{k_{d}} \nonumber \\
& \times \bigg(\frac{\zeta'}{\zeta}(1+\alpha+s)\bigg)^{l_1}\bigg(\frac{\zeta''}{\zeta}(1+\alpha+s)\bigg)^{l_2} \times \cdots \times \bigg(\frac{\zeta^{(d)}}{\zeta}(1+\alpha+s)\bigg)^{l_d} \nonumber \\
& \times \bigg(\frac{\zeta'}{\zeta}(1+\beta+u)\bigg)^{m_1}\bigg(\frac{\zeta''}{\zeta}(1+\beta+u)\bigg)^{m_2} \times \cdots \times \bigg(\frac{\zeta^{(d)}}{\zeta}(1+\beta+u)\bigg)^{m_{d}},
\end{align}
where $\Psi(k_1, k_2, \cdots, k_{d}; l_1, l_2, \cdots, l_d; m_1, m_2, \cdots, m_{d})$ is a constant independent of the complex variables $\alpha, \beta, s$ and $u$. The combinatorial meaning of 
\[
\mathbf{k} = \{k_1, k_2, \cdots, k_{d}\}, \quad \mathbf{l} = \{l_1, l_2, \cdots, l_d\} \quad \textnormal{and} \quad \mathbf{m} = \{m_1, m_2, \cdots, m_d\}
\]
follows from the residue calculus process and it involves partitioning the structure of $d$ and $L_d, \barL_d$. What is critically important about this is not so much the powers and coefficients, which are achievable in an elementary if ugly way as we just showed, but its structure in terms of the variables $\alpha, \beta, s$ and $u$. In $\mathsection$9, we raise the question of whether some of the combinatorial tools such as \cite[Lemma 2.5.1]{cfkrs} could be helpful in elucidating the meaning of these contour integrals. In any case, in what follows, we will need to decouple $s$ and $u$ and to do this, we shall use the Dirichlet convolution
\begin{align} \label{dconvolvedLambdas}
(\mathbf{1} \star \Lambda_1^{\star k_1} \star \Lambda_2^{\star k_2} \star \cdots \star \Lambda_{d+d}^{\star k_{d}})(n).
\end{align}
This should not be surprising as it somewhat tallies up with the structure of the coefficients in Theorem \ref{meanvalueintegral}.

Going back to \eqref{presumS} with these new coefficients and considering only the first part of the integral (i.e. the one involving $\zeta(1+\alpha+\beta)$) yields
\begin{align*}
I_{1,d}(\alpha,\beta) = \frac{T\widehat{\Phi}(0)}{(\log N)^{\sum r \ell_r} (\log N)^{\sum \barr \barell_\barr}} \sumtwo_{i,j} \frac{p_{d,\ell_d,i}p_{d,\barell_d,j}i!j!}{(\log N)^{i+j}} J_{1,d} + O(\mathcal{E}_3),
\end{align*}
where
\begin{align*}
J_{1,d} = \frac{1}{(2 \pi i)^2} \int_{(1)}\int_{(1)}  \frac{1}{(2 \pi i)^{L_d+\barL_d}} \oint \oint \mathfrak{M}_{\alpha,\beta}(\mathbf{w},\mathbf{z};s,u)  \mathscr{D} Z \mathscr{D}W N^{s+u}\frac{ds}{s^{i+1}}\frac{du}{u^{j+1}}.
\end{align*}
Deforming the path of integration of the $s$, $u$-integrals to $\real(s)=\real(u)=\delta$ with $\delta>0$, we come to see that
\begin{align*}
J_{1,d} = \frac{1}{(2 \pi i)^2} \int_{(\delta)} & \int_{(\delta)} \frac{\zeta(1+s+u)\zeta(1+\alpha+\beta)}{\zeta(1+\alpha+s) \zeta(1+\beta+u)} A_{\alpha, \beta}^{(m,n)}(\mathbf{0},\mathbf{0};s,u) \sum_{\mathbf{k},\mathbf{l},\mathbf{m}}\Psi(\mathbf{k}, \mathbf{l}, \mathbf{m}) \nonumber \\
& \times \bigg(\frac{\zeta'}{\zeta}(1+s+u)\bigg)^{k_1}\bigg(\frac{\zeta''}{\zeta}(1+s+u)\bigg)^{k_2} \times \cdots \times \bigg(\frac{\zeta^{(d+d)}}{\zeta}(1+s+u)\bigg)^{k_{d}} \nonumber \\
& \times \bigg(\frac{\zeta'}{\zeta}(1+\alpha+s)\bigg)^{l_1}\bigg(\frac{\zeta''}{\zeta}(1+\alpha+s)\bigg)^{l_2} \times \cdots \times \bigg(\frac{\zeta^{(d)}}{\zeta}(1+\alpha+s)\bigg)^{l_d} \nonumber \\
& \times \bigg(\frac{\zeta'}{\zeta}(1+\beta+u)\bigg)^{m_1}\bigg(\frac{\zeta''}{\zeta}(1+\beta+u)\bigg)^{m_2} \times \cdots \times \bigg(\frac{\zeta^{(d)}}{\zeta}(1+\beta+u)\bigg)^{m_{d}} N^{s+u} \frac{ds}{s^{i+1}}\frac{du}{u^{j+1}}.
\end{align*}
Following the recipe take $\delta \asymp L^{-1}$ and bound trivially to get $J_{1,d} \ll L^{i+j+\sum r \ell_r +\sum {\bar r} \ell_{\bar r}}$. Next, use a Taylor expansion $A_{\alpha,\beta}^{(m,n)}(\mathbf{0},\mathbf{0};s,u)=A_{\alpha, \beta}^{(m,n)}(\mathbf{0},\mathbf{0};\beta, \alpha)+ O(|s-\beta|+|u-\alpha|)$. This puts us in the following situation
%\begin{align*}
%J_{1,d} = \frac{1}{\alpha+\beta} & \frac{1}{(2 \pi i)^2} \int_{\asymp(L^{-1})}\int_{\asymp(L^{-1})} \frac{\zeta(1+s+u)}{\zeta(1+\alpha+s) \zeta(1+\beta+u)} \sum_{\mathbf{k},\mathbf{l},\mathbf{m}} \Psi(\mathbf{k},\mathbf{l},\mathbf{m}) \nonumber \\
%& \times \bigg(\frac{\zeta'}{\zeta}(1+s+u)\bigg)^{k_1}\bigg(\frac{\zeta''}{\zeta}(1+s+u)\bigg)^{k_2} \times \cdots \times \bigg(\frac{\zeta^{(d+d)}}{\zeta}(1+s+u)\bigg)^{k_{d}} \nonumber \\
%& \times \bigg(\frac{\zeta'}{\zeta}(1+\alpha+s)\bigg)^{l_1}\bigg(\frac{\zeta''}{\zeta}(1+\alpha+s)\bigg)^{l_2} \times \cdots \times \bigg(\frac{\zeta^{(d)}}{\zeta}(1+\alpha+s)\bigg)^{l_d} \nonumber \\
%& \times \bigg(\frac{\zeta'}{\zeta}(1+\beta+u)\bigg)^{m_1}\bigg(\frac{\zeta''}{\zeta}(1+\beta+u)\bigg)^{m_2} \times \cdots \times \bigg(\frac{\zeta^{(d)}}{\zeta}(1+\beta+u)\bigg)^{m_{d}} N^{s+u}\frac{ds}{s^{i+1}}\frac{du}{u^{j+1}}.
%\end{align*}
\begin{align*}
J_{1,d} = \frac{1}{\alpha+\beta}& \sum_{n \le N}  \frac{(\mathbf{1} \star \Lambda_1^{\star k_1} \star \Lambda_2^{\star k_2} \star \cdots \star \Lambda_{d+d}^{\star k_{d}})(n)}{n} \sum_{\mathbf{k},\mathbf{l},\mathbf{m}} \Psi(\mathbf{k},\mathbf{l},\mathbf{m}) L_d (\mathbf{l};\alpha,i,n) L_d (\mathbf{m};\beta,j,n),
\end{align*}
by the aid of \eqref{dconvolvedLambdas} and where
\begin{align*}
L_d (\mathbf{l};\alpha,i,n) = \frac{1}{2 \pi i} &\int_{\asymp(L^{-1})} \frac{1}{\zeta(1+\alpha+s)} \\
&\times \bigg(\frac{\zeta'}{\zeta}(1+\alpha+s)\bigg)^{l_1}\bigg(\frac{\zeta''}{\zeta}(1+\alpha+s)\bigg)^{l_2} \times \cdots \times \bigg(\frac{\zeta^{(d)}}{\zeta}(1+\alpha+s)\bigg)^{l_d} \bigg(\frac{N}{n}\bigg)^s \frac{ds}{s^{i+1}},
\end{align*}
as well as
\begin{align*}
L_d (\mathbf{m};\beta,j,n) = \frac{1}{2 \pi i} &\int_{\asymp(L^{-1})} \frac{1}{\zeta(1+\beta+u)} \\
& \times \bigg(\frac{\zeta'}{\zeta}(1+\beta+u)\bigg)^{m_1}\bigg(\frac{\zeta''}{\zeta}(1+\beta+u)\bigg)^{m_2} \times \cdots \times \bigg(\frac{\zeta^{(d)}}{\zeta}(1+\beta+u)\bigg)^{m_{d}} \bigg(\frac{N}{n}\bigg)^u \frac{du}{u^{j+1}}.
\end{align*}

We shall need one further tool from complex analysis before proceeding.

\begin{lemma} \label{contourlemma}
Suppose $0< \delta \asymp L^{-1}$, $i\geq 1$ and $\alpha \ll L^{-1}$. 
Let $N \ge n \ge 0$ and  $l_r=0,1,2,\cdots$ for $r=1,\cdots,d$. 
One then has for some  $\nu \asymp \log\log N$ that
\begin{align*}
 \frac{1}{2 \pi i} \int_{(\delta)} \bigg(\frac{N}{n}\bigg)^s \frac{1}{\zeta(1+\alpha+s)} & \bigg(\frac{\zeta'}{\zeta}(1+\alpha+s)\bigg)^{l_1}\bigg(\frac{\zeta''}{\zeta}(1+\alpha+s)\bigg)^{l_2} \times \cdots \times \bigg(\frac{\zeta^{(d)}}{\zeta}(1+\alpha+s)\bigg)^{l_d} \frac{ds}{s^{i+1}} \\
& = \Upsilon(d,\mathbf{l}) 
%+ \textcolor{red}{O(E_1)+O(E_2)},
+ O({L^{{\sum_{r=1}^{d}r l_r-2+i}}}) + O\bigg( {{{\bigg( {\frac{{{N}}}{n}} \bigg)}^{ - \nu }}{\log(N)^\varepsilon }} \bigg),
\end{align*}
where
\begin{align*}
\Upsilon(d,\mathbf{l}) := (1!(-1)^{1})^{l_1}(2!(-1)^2)^{l_2} \cdots (d!(-1)^d)^{l_d} \frac{1}{2 \pi i} \oint \frac{1}{(\alpha+s)^{1 \times l_1 + 2 \times l_2 + \cdots + d \times l_d - 1}} \bigg(\frac{N}{n}\bigg)^s \frac{ds}{s^{i+1}},
\end{align*}
%\[
%\Upsilon(d,g) = (r!(-1)^r)^{l_r} \frac{1}{2 \pi i} \oint \frac{1}{(\alpha+s)^{rl_r-d}} \bigg(\frac{N}{n}\bigg)^s \frac{ds}{s^{i+1}},
%\]
with the contour of integration being a circle of radius one centered at the origin and enclosing $-\alpha$.
\end{lemma}
\begin{proof}
The idea behind this proof is to use the standard zero-free region of $\zeta$, see for instance \cite[Lemma 6.1]{bcy}. 
Let $U$ be a large parameter with $U\to\infty$ and $U = o(T)$ as $T\to\infty$, which will be chosen at the end of the proof. 
The above integral is, by Cauchy's theorem, equal to the sum of the residues at $s=0$ and at $s=-\alpha$ plus the sum over line integrals over the segments $\gamma_1 = \{s=it \,:\, t \in \R, \; |t| \ge U\}$, $\gamma_2 = \{ s = \sigma \pm iU \, : \, -c/\log U \le \sigma \le 0\}$, and $\gamma_3 = \{ s = -c/\log U + it \, : \, |t| \le U\}$, where $c$ is some fixed positive constant such that $\zeta(1+\alpha+s)$ has no zeros in the region on the right-hand side of the contour determined by the $\gamma_i$'s. Another two requirements on $c$ are that the estimate 
\begin{align*}
\frac{1}{\zeta(\sigma+it)} \ll \log (2+|t|)
\end{align*}
holds in this region, and that
\begin{align*}
\frac{\zeta^{(j)}}{\zeta}(\sigma+it) \ll \log^j(4+|t|), \quad j=1,2,\cdots,
\end{align*}
see \cite[Theorem 6.7]{montvaug}. The below diagram \ref{gammacontours} illustrates the contour of integration.
\begin{figure}[H] \label{gammacontours}
\centering
\includegraphics[width=.35\textwidth]{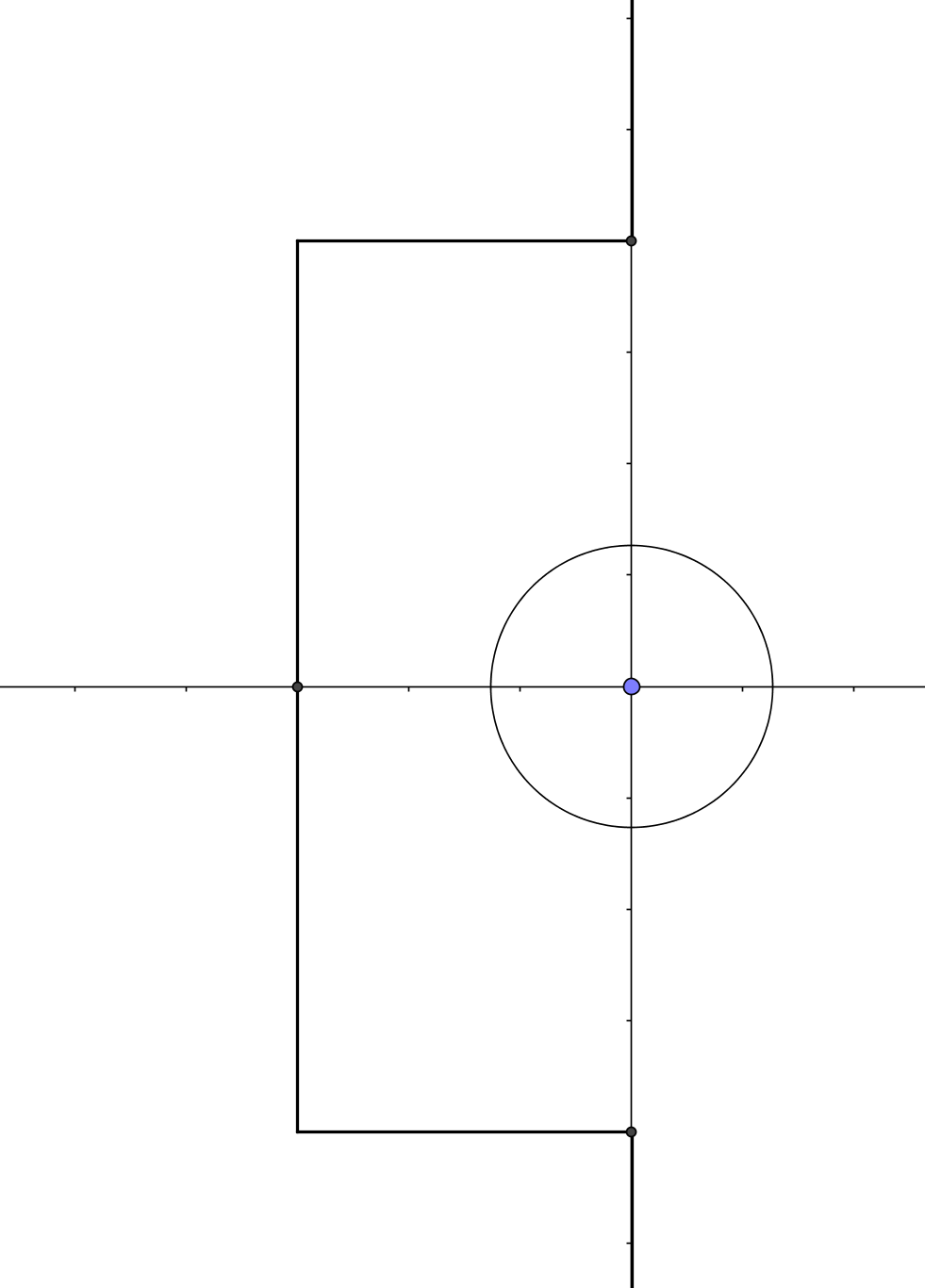}
   \put(-61,92){\mbox{$0$}}
%   \put(-93,27){\mbox{$\gamma_1$}}
   \put(-47,188){\mbox{$iU$}}
    \put(-47,30){\mbox{$-iU$}}
   \put(-149,118){\mbox{$-\frac{c}{\log U}$}}
%  \put(-23,82){\mbox{\scriptsize$\varphi$}}
%  \put(-57.8,79){\mbox{\scriptsize$\alpha$}}
 \caption{Curve $\gamma$ in the proof of Lemma~\ref{contourlemma}.}	
\end{figure}
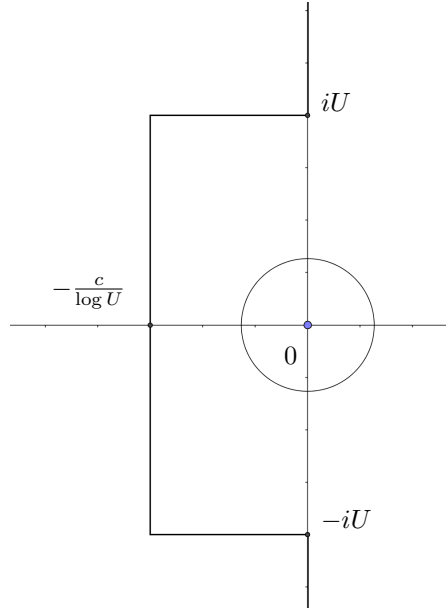
%----------------
Then, one has
\[
\int_{\gamma_1} \ll {\int_U^\infty \frac{(\log t )^{1+\sum_{r=1}^{d}r l_r}}{t^{i+1}}\, dt} \ll
{\frac{(\log U )^{1+\sum_{r=1}^{d}r l_r}}{U^{i}}},
\]
since $i \ge 1$. Moreover, since $n\leq N$
\[
\int_{\gamma_2} \ll {\int_{-c/\log U}^0 (\log U )^{1+\sum_{r=1}^{d}r l_r} \left(\frac{N}{n}\right)^x \frac{1}{U^{i+1}} \, dx} 
\ll {\frac{ (\log U )^{\sum_{r=1}^{d}r l_r}}{U^{i+1}}} ,
\]
and finally
\begin{align*}
\int_{\gamma_3} &\ll {\int_{-U}^{U}  (\log(4+|t|))^{1+\sum_{r=1}^{d}r l_r} \frac{(N/n)^{-c/\log U}}{(c^2/\log^2 U +t^2)^{(i+1)/2}} \, dt}\\
&\ll 
(N/n)^{-c/\log U}\bigg( \int_{0}^{c/\log U} + \int_{c/\log U}^1+ \int_{1}^{U} \bigg)
 \frac{(\log(4+|t|))^{1+\sum_{r=1}^{d}r l_r}}{(c^2/\log^2 U +t^2)^{(i+1)/2}} \, dt\\
&\ll
 (N/n)^{-c/\log U} \int_{0}^{c/\log U}   \frac{(\log 5 )^{1+\sum_{r=1}^{d}r l_r}}{|c/\log U|^{(i+1)}}dt
 +
 (N/n)^{-c/\log U}\int_{c/\log U}^1    \frac{(\log 5 )^{1+\sum_{r=1}^{d}r l_r}}{t^{(i+1)}}dt\\
 & \quad +
 (N/n)^{-c/\log U} \int_{1}^U   \frac{(\log(4+U))^{1+\sum_{r=1}^{d}r l_r}}{t^{(i+1)}}dt\\
&\ll
(N/n)^{-c/\log U} ( \log^{i} U + (\log U )^{1+\sum_{r=1}^{d}r l_r}).
%(\log U )^{i+\sum_{r=1}^{d}r l_r}(N/n)^{-c/\log U} \\
%\ll
\end{align*}

Appropriately choosing $U \asymp \log N$ gives an error of size $O((\log \log N)^{i+\sum_{r=1}^{d}r l_r}) = O(\log N)$.
%-----------
The last thing we need to do is collect the residues, which we write as the contour integral 
\begin{align*}
 \frac{1}{2 \pi i} \oint_{\Omega} \bigg(\frac{N}{n}\bigg)^s \frac{1}{\zeta(1+\alpha+s)} & \bigg(\frac{\zeta'}{\zeta}(1+\alpha+s)\bigg)^{l_1}\bigg(\frac{\zeta''}{\zeta}(1+\alpha+s)\bigg)^{l_2} \times \cdots \times \bigg(\frac{\zeta^{(d)}}{\zeta}(1+\alpha+s)\bigg)^{l_d} \frac{ds}{s^{i+1}},
\end{align*}
where the contour is now a small circle $\Omega$ of radius $\asymp L^{-1}$ centered around the origin such that $-\alpha \in \Omega$. Since the radius of the circle tends to zero as $T \to \infty$, we may use the Laurent expansions around $s = -\alpha$ 
\[
\frac{1}{\zeta(1+\alpha+s)} = (\alpha+s) \left(1+ O(\alpha+s)\right),
\]
where, we recall, $C_0$ is the Euler constant, as well as
\[
\bigg(\frac{\zeta^{(r)}}{\zeta}(1+\alpha+s)\bigg)^{l_r} 
= 
\bigg(\frac{r!(-1)^r}{(\alpha+s)^r}\bigg)^{l_r} \left(1+ O(\alpha+s)\right),
\]
to finally arrive at
\begin{align*}
&\frac{1}{2 \pi i}  \oint_{\Omega} \bigg(\frac{N}{n}\bigg)^s \frac{1}{\zeta(1+\alpha+s)} \bigg(\frac{\zeta'}{\zeta}(1+\alpha+s)\bigg)^{l_1}\bigg(\frac{\zeta''}{\zeta}(1+\alpha+s)\bigg)^{l_2} \times \cdots \times \bigg(\frac{\zeta^{(d)}}{\zeta}(1+\alpha+s)\bigg)^{l_d} \frac{ds}{s^{i+1}} \\
& = \frac{1}{2 \pi i} \oint \bigg(\frac{N}{n}\bigg)^s (\alpha+s) \bigg(\frac{1!(-1)^1}{(\alpha+s)^1}\bigg)^{l_1}  \bigg(\frac{2!(-1)^2}{(\alpha+s)^2}\bigg)^{l_2} \times \cdots  \times \bigg(\frac{d!(-1)^d}{(\alpha+s)^d}\bigg)^{l_d} \left(1+ O(\alpha+s)\right) \frac{ds}{s^{i+1}}.
\end{align*}
Using a direct estimation yields that the right-hand side above is equal to 
\begin{align*}
(1!(-1)^{1})^{l_1}(2!(-1)^2)^{l_2} \cdots (d!(-1)^d)^{l_d} \frac{1}{2 \pi i} \oint \frac{1}{(\alpha+s)^{1 \times l_1 + 2 \times l_2 + \cdots + d \times l_d - 1}} \bigg(\frac{N}{n}\bigg)^s \frac{ds}{s^{i+1}},
\end{align*}
which is precisely the definition of $\Upsilon(d,\mathbf{l})$ given in the statement of the lemma.
\end{proof}

The resulting contour integral leads to three very different cases which we now need to distinguish. Set 
\begin{align} \label{defomega}
\omega(d,\mathbf{l}) := 1 \times l_1 + 2 \times l_2 + \cdots + d \times l_d - 1
\end{align}
\subsubsection{Case $A$: Terms for which $\omega(d,\mathbf{l}) = -1$} In this case we have
\begin{align*}
\Upsilon_A(d,\mathbf{l}) &= \mathcal{U}(d,\mathbf{l}) \frac{1}{2 \pi i} \oint (\alpha+s) \bigg(\frac{N}{n}\bigg)^s \frac{ds}{s^{i+1}} = \mathcal{U}(d,\mathbf{l}) \frac{d}{dx} e^{\alpha x} \frac{1}{2 \pi i} \oint \bigg(\frac{N e^x}{n}\bigg)^s \frac{ds}{s^{i+1}} \bigg|_{x=0} \\ & = \mathcal{U}(d,\mathbf{l})\frac{1}{i!}\frac{d}{dx}e^{\alpha x}\bigg(x + \log \frac{N}{n} \bigg)^i \bigg|_{x=0} = \mathcal{U}(d,\mathbf{l})\frac{\log^i N}{i!\log N}\frac{d}{dx}N^{\alpha x}\bigg(x  + \frac{\log(N/n)}{\log N} \bigg)^i \bigg|_{x=0},
\end{align*}
by the change of variable $x \to x \log N$ and where we have adopted the shorter notation
\begin{align} \label{defU}
\mathcal{U}(d,\mathbf{l}) = \mathbf{1}\{\omega(d,\mathbf{l}) = -1\} (1!(-1)^{1})^{l_1}(2!(-1)^2)^{l_2} \cdots (d!(-1)^d)^{l_d} ,
\end{align}
with
\begin{equation}
\mathbf{1}\{\omega(d,\mathbf{l}) = -1\} = \begin{cases}
1, & \mbox{ if $\omega(d,\mathbf{l}) = -1$}, \nonumber \\
0, & \mbox{ otherwise}. \nonumber 
\end{cases}
\end{equation}
We take the chance now to sum over $i$ (the index coming from the polynomial decomposition of $P_{d,\ell}$), so that
\[
\sum_i \frac{p_{d,\ell_d,i}i!}{\log^i N} \mathcal{U}(d,\mathbf{l})\frac{\log^i N}{i!\log N}\frac{d}{dx}N^{\alpha x}\bigg(x  + \frac{\log(N/n)}{\log N} \bigg)^i \bigg|_{x=0} = \frac{\mathcal{U}(d,\mathbf{l})}{\log N}\frac{d}{dx}N^{\alpha x} P_{d,\ell}\bigg(x  + \frac{\log(N/n)}{\log N} \bigg)\bigg|_{x=0}. 
\]
\subsubsection{Case $B$: Terms for which $\omega(d,\mathbf{l}) = 0$} We now deal with
\begin{align*}
\Upsilon_B(d,\mathbf{l}) &= \mathcal{V}(d,\mathbf{l}) \frac{1}{2 \pi i} \oint \bigg(\frac{N}{n}\bigg)^s \frac{ds}{s^{i+1}} = -\mathcal{V}(d,\mathbf{l}) \frac{\log^i N}{i!} \bigg(\frac{\log(N/n)}{\log N}\bigg)^i,
\end{align*}
by the use of Cauchy's integral theorem and where 
\begin{align} \label{defV}
\mathcal{V}(d,\mathbf{l}) = \mathbf{1}\{\omega(d,\mathbf{l}) = 0\}(1!(-1)^{1})^{l_1}(2!(-1)^2)^{l_2} \cdots (d!(-1)^d)^{l_d}.
\end{align}
We perform the sum over $i$ so that
\[
\sum_i \frac{p_{d,\ell_d,i}i!}{\log^i N} \mathcal{V}(d,\mathbf{l}) \frac{\log^i N}{i!} \bigg(\frac{\log(N/n)}{\log N}\bigg)^i = -\mathcal{V}(d,\mathbf{l}) P_{d,\ell}\bigg(\frac{\log(N/n)}{\log N}\bigg).
\]
\subsubsection{Case $C$: Terms for which $\omega(d,\mathbf{l}) > 0$} 
This is the most complicated case. We have
\[
\Upsilon_C(d,\mathbf{l}) = \mathcal{W}(d,\mathbf{l}) \frac{1}{2 \pi i} \oint \frac{1}{(\alpha+s)^{\omega(d,\mathbf{l})}} \bigg(\frac{N}{n}\bigg)^s \frac{ds}{s^{i+1}},
\]
where
\begin{align} \label{defW}
\mathcal{W}(d,\mathbf{l}) = \mathbf{1}\{\omega(d,\mathbf{l}) > 0\}(1!(-1)^{1})^{l_1}(2!(-1)^2)^{l_2} \cdots (d!(-1)^d)^{l_d}.
\end{align}
For the third case we employ
\begin{align} \label{tauidentity}
\int_{1/q}^1 t^{\alpha+s-1} \log^\tau t dt = \frac{(-1)^\tau \tau!}{(\alpha+s)^{\tau+1}} - \frac{q^{-\alpha-s}}{(\alpha+s)^{\tau+1}}P(s,\alpha,\log q),
\end{align}
for $\tau=0,1,2,\cdots$. Here $P$ is a polynomial in $\log q$ of degree $\tau$, see \cite[p. 56]{bcy} and \cite[p. 285]{rrz01}. Only the first term of the right-hand side above contributes when we insert this expression into $\Upsilon(d,\mathbf{l})$. This is because the contribution from the second term vanishes by taking the contour to be arbitrary large. The explicit calculation is as follows (temporarily set $\tau = \omega(d,\mathbf{l})-1$ and $q=\frac{N}{n}$):
\begin{align} 
\Upsilon_C(d,\mathbf{l}) &= \mathcal{W}(d,\mathbf{l})\frac{1}{2 \pi i} \oint (\alpha+s)^{-\omega(d,\mathbf{l})}\bigg(\frac{N}{n}\bigg)^s \frac{ds}{s^{i+1}} = \mathcal{W}(d,\mathbf{l})\frac{1}{2 \pi i} \oint \frac{1}{(\alpha+s)^{\tau+1}} \bigg(\frac{N}{n}\bigg)^s \frac{ds}{s^{i+1}} \nonumber \\
& = \mathcal{W}(d,\mathbf{l})\frac{1}{2 \pi i} \oint \bigg(\frac{N}{n}\bigg)^s \int_{1/q}^1 \frac{1}{(-1)^\tau \tau!} t^{\alpha+s-1} \log^\tau t dt \frac{ds}{s^{i+1}} \nonumber \\
& = \mathcal{W}(d,\mathbf{l})\frac{1}{(-1)^{\tau}\tau!} \int_{1/q}^1 \bigg(\frac{1}{2 \pi i} \oint (qt)^s \frac{ds}{s^{i+1}}\bigg) t^{\alpha-1} \log^\tau t dt \nonumber \\
& = \mathcal{W}(d,\mathbf{l})\frac{1}{(-1)^{\omega(d,\mathbf{l})-1}(\omega(d,\mathbf{l})-1)!} \frac{1}{i!} \int_{1/q}^1 (\log qt)^i t^{\alpha-1} (\log t)^{\omega(d,\mathbf{l})-1} dt \nonumber \\
& = \mathcal{W}(d,\mathbf{l})\frac{(-1)^{1-\omega(d,\mathbf{l})}}{i!(\omega(d,\mathbf{l})-1)!} \bigg(\log \frac{N}{n}\bigg)^{i+\omega(d,\mathbf{l})} \int_0^1 (1-a)^i a^{\omega(d,\mathbf{l})-1} \bigg(\frac{N}{n}\bigg)^{-\alpha a}da \nonumber \\
& = \mathcal{W}(d,\mathbf{l})\frac{(-1)^{1-\omega(d,\mathbf{l})}}{i!(\omega(d,\mathbf{l})-1)!} (\log N)^{\omega(d,\mathbf{l})} \bigg(\frac{\log (N/n)}{\log N}\bigg)^{\omega(d,\mathbf{l})} \bigg(\log \frac{N}{n}\bigg)^{i} \int_0^1 (1-a)^i a^{\omega(d,\mathbf{l})-1} \bigg(\frac{N}{n}\bigg)^{-\alpha a}da, \nonumber
\end{align}
by the use of the change of variables $t=q^{-a}$. Lastly we do the sum over $i$ and obtain
\begin{align*}
&\sum_i \frac{p_{d,\ell_d,i}i!}{\log^i N}\mathcal{W}(d,\mathbf{l})\frac{(-1)^{1-\omega(d,\mathbf{l})}}{i!(\omega(d,\mathbf{l})-1)!} (\log N)^{\omega(d,\mathbf{l})} \bigg(\frac{\log (N/n)}{\log N}\bigg)^{\omega(d,\mathbf{l})} \bigg(\log \frac{N}{n}\bigg)^{i} \int_0^1 (1-a)^i a^{\omega(d,\mathbf{l})-1} \bigg(\frac{N}{n}\bigg)^{-\alpha a}da \\
&= \mathcal{W}(d,\mathbf{l})\frac{(-1)^{1-\omega(d,\mathbf{l})}}{(\omega(d,\mathbf{l})-1)!} (\log N)^{\omega(d,\mathbf{l})}\bigg(\frac{\log (N/n)}{\log N}\bigg)^{\omega(d,\mathbf{l})}  \int_0^1 P_{d,\ell}\bigg((1-a)\frac{\log (N/n)}{\log N}\bigg) a^{\omega(d,\mathbf{l})-1} \bigg(\frac{N}{n}\bigg)^{-\alpha a}da.
\end{align*}

Let us now recap. We had $I_d(\alpha,\beta) = I_{1,d}(\alpha,\beta) + I_{2,d}(\alpha,\beta) = I_{1,d}(\alpha,\beta) + T^{-\alpha-\beta}I_{1,d}(-\beta,-\alpha)+O(T/L)$, where
\begin{align*}
I_{1,d}(\alpha,\beta) = \frac{T\widehat{\Phi}(0)}{(\log N)^{\sum_{r=1}^d r \ell_r} (\log N)^{\sum_{\barr=1}^d \barr \barell_\barr}} & \frac{1}{\alpha+\beta} \sum_{\mathbf{k}, \mathbf{l}, \mathbf{m}} \Psi(\mathbf{k},\mathbf{l},\mathbf{m})\\ 
& \times \sum_{n \le N} \frac{(\mathbf{1} \star \Lambda_1^{\star k_1} \star \Lambda_2^{\star k_2} \star \cdots \star \Lambda_{d+d}^{\star k_d})(n)}{n} F_d(\mathbf{l},\alpha,n) F_d(\mathbf{k},\beta,n).
\end{align*}
The terms $F_d$ are given by the three different cases
\begin{align*}
F_d(\mathbf{l},\beta,n) = \begin{cases}
\frac{\mathcal{U}(d,\mathbf{l})}{\log N}\frac{d}{dx}N^{\alpha x} P_{d,\ell}(x  + \frac{\log(N/n)}{\log N} )|_{x=0}, \\
\mathcal{V}(d,\mathbf{l}) P_{d,\ell}(\frac{\log(N/n)}{\log N}), \\
\mathcal{W}(d,\mathbf{l})\frac{(-1)^{1-\omega(d,\mathbf{l})}}{(\omega(d,\mathbf{l})-1)!} (\log N)^{\omega(d,\mathbf{l})}(\frac{\log (N/n)}{\log N})^{\omega(d,\mathbf{l})}  \int_0^1 P_{d,\ell}((1-a)\frac{\log (N/n)}{\log N}) a^{\omega(d,\mathbf{l})-1} (\frac{N}{n})^{-\alpha a}da ,
\end{cases}
\end{align*}
depending on whether $\omega(d,\mathbf{l}) = -1$, $\omega(d,\mathbf{l}) = 0$ and $\omega(d,\mathbf{l}) > 0$, respectively. Similarly
\begin{align*}
F_d(\mathbf{k},\alpha,n) = \begin{cases}
\frac{\mathcal{U}(d,\mathbf{k})}{\log N}\frac{d}{dy}N^{\beta y} P_{d,\ell}(y + \frac{\log(N/n)}{\log N} )|_{y=0}, \\
\mathcal{V}(d,\mathbf{k}) P_{d,\ell}(\frac{\log(N/n)}{\log N}), \\
\mathcal{W}(d,\mathbf{k})\frac{(-1)^{1-\omega(d,\mathbf{k})}}{(\omega(d,\mathbf{k})-1)!} (\log N)^{\omega(d,\mathbf{k})}(\frac{\log (N/n)}{\log N})^{\omega(d,\mathbf{k})}  \int_0^1 P_{d,\ell}((1-b)\frac{\log (N/n)}{\log N}) b^{\omega(d,\mathbf{k})-1} (\frac{N}{n})^{-\beta b}db,
\end{cases}
\end{align*}
depending on whether $\omega(d,\mathbf{k}) = -1$, $\omega(d,\mathbf{k}) = 0$ and $\omega(d,\mathbf{k}) > 0$, respectively. Here $\mathcal{U}, \mathcal{V}$ and $\mathcal{W}$ are given by \eqref{defomega}, \eqref{defU}, \eqref{defV}, \eqref{defW}.

So we must now cross each of all possible $3 \times 3 = 9$ terms. In fact, due to symmetries, there will only be six different cases, much like in the Euler product. 

To deal with the sum over $n$ we will need the following Euler-Maclaurin lemma.

\begin{lemma} \label{eulermaclaurinlemma}
We have
\begin{align*}
& \sum_{n \le z} \frac{(d_k \star \Lambda_1^{\star k_1} \star \cdots \star \Lambda_m^{\star k_m} )(n)}{n^{1+s}} F\bigg( \frac{\log(x/n)}{\log x} \bigg) H\bigg( \frac{\log(z/n)}{\log z} \bigg) \\
& = \frac{1^{k_1} (2!)^{k_2} \cdots (m!)^{k_m} (\log z)^{k + 1\times k_1 + \cdots m \times k_m}}{z^s(k + 1\times k_1 + \cdots m \times k_m-1)!} \int_0^1 (1-u)^{k + 1\times k_1 + \cdots m \times k_m - 1} F\bigg(1-(1-u)\frac{\log z}{\log x} \bigg) H(u)z^{us} du \\
& \quad + O\left((\log z)^{k + 1\times k_1 + \cdots m \times k_m-1}\right),
\end{align*}
for smooth functions $F$ and $H$ in the interval $[0,1]$.
\end{lemma}
\begin{proof}
Let $g(n)$ be an arithmetic function with  
\begin{align}
\sum_{n \le z} g(n)= c_gz\log^{k_g-1}z + O(z\log^{k_g-2}z)
\label{eq:euler_mac1}
\end{align}
for some $c_g>0$ and $k_g\geq1$. Then 
\begin{align*}
\sum_{n \le z} \frac{g(n)}{n^{1+s}} F\bigg( \frac{\log(x/n)}{\log x} \bigg) H\bigg( \frac{\log(z/n)}{\log z} \bigg) 
&= \frac{c_g\log^{k_g}z}{z^s}  \int_0^1 (1-u)^{k_g - 1} F\bigg(1-(1-u)\frac{\log z}{\log x} \bigg) H(u)z^{us} du \\
 & \quad + O(\log^{k_g-1})
\end{align*}
for $F$, $H$ and $s$ as in the lemma. 
The proof of this statement is almost the same as the proof of \cite[Lemma~2.4]{krz01} and we thus omit it.
It is thus enough to show that 
\begin{align}
\sum_{n \le z} (d_k \star \Lambda_1^{\star k_1} \star \cdots \star \Lambda_m^{\star k_m} )(n)
&= \frac{1^{k_1} (2!)^{k_2} \cdots (m!)^{k_m}}{(k+1\times k_1+\ldots+m\times k_m -1)!} 
z (\log z)^{(k+1\cdot k_1+\ldots+m\cdot k_m)-1} \nonumber \\
& \quad \times \big(1+O(\log^{-1}z)\big)
\label{eq:euler_mac2}
\end{align}
To prove \eqref{eq:euler_mac2}, we need 
\begin{align}
\sum_{n \le z} g(n) \log (n)
&=
c_g z\log^{k_g}z + O(z\log^{k_g-1}z)
\label{eq:euler_mac3}\\
\sum_{n \le z} \frac{g(n) \log^\ell (n)}{n}
&=
\frac{c_g z\log^{k_g+\ell-1}z}{k_g+\ell} + O(z\log^{k_g+\ell-2}z) \nonumber \\
\sum_{n \le z} (g \star \Lambda)(n)
&=
\frac{c_g z\log^{k_g}z}{k_g} + O(z\log^{k_g-1}z)
\label{eq:euler_mac5}
\end{align}
with $g(n)$ as in \eqref{eq:euler_mac1}. 
The proof of these three equations use partial summation and we omit the details.
Finally, let $f(n)$ be an arithmetic function with  
\begin{align*}
\sum_{n \le z} f(n)= c_fz\log^{k_f-1}z + O(z\log^{k_f-2}z),
\end{align*}
then
\begin{align}
\sum_{n \le z} (g\star f)(n)= 
c_fc_g z(\log z)^{k_f+k_g-1} \frac{(k_f-1)!(k_g-1)!}{(k_f+k_g-1)!}+ O(z\log^{k_f+k_g-2}z)
\label{eq:euler_mac7}.
\end{align}
We have 
\begin{align*}
\sum_{n \le z} (g\star f)(n)
&= 
\sum_{a \le z} g(a) \sum_{z/a}f(b)
= 
\sum_{a \le z} g(a) \left(c_f\frac{z}{a}\log^{k_f-1}(z/a) + O(z\log^{k_f-2}z) \right)\\
&=
c_f z\bigg(\sum_{a \le z} \frac{g(a)}{a} \log^{k_f-1}(z/a)\bigg) + O(z\log^{k_f+k_g-2}z) \\
&=
c_f z \bigg(\sum_{j=0}^{k_f-1} \binom{k_f-1}{j}(-1)^j\sum_{a \le z} \frac{g(a)}{a} \log^{j}(a) \log^{k_f-1-j}(z)\bigg) + O(z\log^{k_f+k_g-2}z)\\
&=
%c_f z \bigg(\sum_{j=0}^{k_f-1} \binom{k_f-1}{j}(-1)^j \frac{c_g}{k_g+j} \log^{k_f+k_g-1}(z)\bigg) + O(z\log^{k_f+k_g-2}z)\\
%&=
c_fc_g z \bigg(\sum_{j=0}^{k_f-1} \binom{k_f-1}{j}(-1)^j \frac{1}{k_g+j}\bigg) \log^{k_f+k_g-1}(z) + O(z\log^{k_f+k_g-2}z).
\end{align*}
We now have 
\begin{align*}
\sum_{j=0}^{k_f-1} \binom{k_f-1}{j}(-1)^j \frac{1}{k_g+j} &= \sum_{j=0}^{k_f-1} \binom{k_f-1}{j}(-1)^j  \int_0^1 u^{k_g+j-1}du \\
& =
\int_0^1 u^{k_g-1}(1-u)^{k_f-1}du
=
\frac{(k_f-1)!(k_g-1)!}{(k_f+k_g-1)!}.
\end{align*}
The last equality can be proved for instance by induction over $k_f$, which completes the proof of \eqref{eq:euler_mac7}.
We now can show that we have for each $k\geq 1$
\begin{align*}
\sum_{n\leq z} \Lambda_k(n) = k z \log^{k-1}z + O(z \log^{k-2}z ).
\end{align*}
We prove this equation by induction. 
We have $\Lambda_1= \Lambda$ and thus the case $k=1$ is trivial. 
We thus assume the statement holds for a given $k$. 
We use \eqref{arithmeticLambdaK} together with \eqref{eq:euler_mac3} and \eqref{eq:euler_mac5} and get
\begin{align*}
\sum_{n\leq z} \Lambda_{k+1}(n) 
&= 
\sum_{n\leq z} \Lambda_{k}(n) \log n
+
\sum_{n\leq z} (\Lambda_{k}\star \Lambda)(n)
=
k z \log^{k}z
+
k \frac{z \log^{k}z}{k} + O(z \log^{k-1}z) \\
&=
(k+1) z \log^{k}z + O(z \log^{k-1}z ).
\end{align*}
In a similar way, one can show that 
\begin{align*}
\sum_{n\leq z} d_k(n) =  \frac{z \log^{k-1}z}{(k-1)!} + O(z \log^{k-2}z ).
\end{align*}
We now can prove \eqref{eq:euler_mac2} and thus complete the proof of the lemma. 
We argue by induction. The case $k_1=k_2=\ldots=k_m=0$ follows immediately from \eqref{eq:euler_mac7}.
Suppose now that \eqref{eq:euler_mac2} for given $k_1$, $k_2$, \ldots $k_m$.
We then get with \eqref{eq:euler_mac7}
\begin{align*}
\sum_{n \le z} (d_k \star \Lambda_1^{\star k_1} \star \cdots \star \Lambda_m^{\star {k_m+1}} )(n)
&=
\sum_{n \le z} ((d_k \star \Lambda_1^{\star k_1} \star \cdots \star \Lambda_m^{\star {k_m}})\star \Lambda_m )(n)\\
&=
z (\log z)^{(k+1\cdot k_1+\ldots+m\cdot k_m)+k_m-1}\left(\frac{1^{k_1} (2!)^{k_2} \cdots (m!)^{k_m}}{(k+1\cdot k_1+\ldots+m\cdot k_m-1)!}\right) k_m  \\
& \quad \times\frac{\big((k+1\cdot k_1+\ldots+m\cdot k_m)-1\big)!(k_m-1)!}{\big((k+1\cdot k_1+\ldots+m\cdot k_m)+k_m-1 \big)!} .
\end{align*}
This ends the proof.
\end{proof}

All but one of the $\log N$'s in $I_{1,d}$ will cancel with the $\log N$'s coming from the Euler-Maclaurin expression leaving us only with $\frac{1}{\log N}$ in $I_{1,d}$.

The last step is to apply the differential operators $Q$ on $I_{1,d}(\alpha,\beta)+T^{-\alpha-\beta}I_{1,d}(-\beta,-\alpha)$
\begin{align*}
I_d = Q \bigg(\frac{-1}{\log T} \frac{\partial}{\partial \alpha}\bigg) Q \bigg(\frac{-1}{\log T} \frac{\partial}{\partial \beta}\bigg) [I_{1,d}(\alpha,\beta)+T^{-\alpha-\beta}I_{1,d}(-\beta,-\alpha)] \bigg|_{\alpha=\beta=-R/L}.
\end{align*}

From the above procedure a term of the form $\log N$ (such as $\log (N^{x+y}T)$ for example) will come out. This term will be combined with the term $\frac{1}{\log N}$ above to produce an expression like $\frac{\log (N^{x+y}T)}{\log N} = \frac{\theta(x+y)+1}{\theta}$, thereby removing the $\log T$ dependence. For the terms involving derivatives of the arithmetical factor $A_{\alpha,\beta}$, it can be seen that the resulting expression after the application of the two $Q$'s will be of the form $\frac{\log N}{\log^2 N} \asymp \frac{1}{\log T}$, thereby producing error terms of size $\frac{T\widehat{\Phi}(0)}{\log T} \ll T/L$.

This ends our delineation of the final main terms.

\section{Numerical aspects}

The below numerical calculations are similar to the ones that appeared in the older version of Feng's paper on the \texttt{arXiv}. Feng obtains $\kappa \ge 0.417288$ which was rounded up to $0.4173$. Since we were never able to achieve a $\frac{4}{7}$ length in the past, we had to rely on our code of the main terms (which matched Feng's). However we are now able to reach $\frac{4}{7}$, we are only too happy to give credit to Feng and recover part of his numerical setup. Let us set $d=1$ and $K=3$ in \eqref{fengsmollifierlogp} as well as
\[
\theta = \frac{4}{7}-\varepsilon \quad \textnormal{and} \quad R=1.3036.
\]
Moreover, we slightly increase the number of terms of the polynomials $P$ of the mollifiers (Feng only used $\deg P_{1}=4$, $\deg P_{2}=2$ and $\deg P_{3}=1$) and take
\begin{align*}
P_{1}(x) &=   x + 0.261076 x (1 - x) - 
 1.071007 x (1 - x)^2 - 0.236840 x (1 - x)^3 + 
 0.260233 x (1 - x)^4, \\
P_{2}(x) &= 1.048274 x + 1.319912 x^2 - 
 0.940058 x^3, \\
P_{3}(x) &= 0.522811 x - 0.686510 x^2 - 
 0.049923 x^3,
\end{align*}
as well as
\begin{align*}
Q(x) = 0.490464 + 
 0.636851 (1 - 2 x) - 0.159327 (1 - 2 x)^3 + 
 0.032011(1 - 2 x)^5.
\end{align*}
Letting $\varepsilon \to 0$ yields $\kappa \ge 0.417293962$. Using only a linear polynomial in $Q$ yields the proportion of simple zeros on the critical line, see \cite{anderson, heathbrownsimple, levinsoncollected}. Taking 
\begin{align*}
P_{1}(x) &= x + 0.052703 x (1 - x) - 
 0.657999 x (1 - x)^2 - 
 0.003193 x (1 - x)^3 - 0.101832 x (1 - x)^4 \\
P_{2}(x) &= 1.049837 x  -0.097446 x^2 \\
P_{3}(x) &= 0.035113 x - 0.156465 x^2,
\end{align*}
as well as 
\begin{align*}
Q(x) = 0.483777 + (1 - 0.483777) (1 - 2 x);
\end{align*}
along with $R=1.1167$, $\theta=\frac{4}{7}$, $d=1$ and $K=3$ and letting $\varepsilon \to 0$ yields $\kappa^* \ge 0.407511457$.

\section{Further remarks and future work}

We end our paper with some questions and discussions on how the methodology we have presented could be used in future work.
\begin{enumerate}
\item[(1)] Naturally, the most pressing question is whether these ideas can be applied to other $L$-functions in the Selberg class. Along with this comes their associations with compact groups.
\item[(2)] Equally important would be the effect of these mollification refinements on other arithmetical objects such as discrete moments, logarithmic moments, $k$-th moments, and pair correlations as illustrated in e.g. \cite{cs}.
\item[(3)] In the proof of Theorem \ref{meanvalueintegral}, it was important to exploit the additivity of the logarithm in $a_n \sim (\mu \star \log)(n)$. This allowed us to use Vaughan and Heath-Brown identities \cite{heathbrown, vaughanidentity} and get $\theta = \frac{4}{7}-\varepsilon$. Moreover, in \cite{bcr} it is shown that the best size one can take for coefficients $a_n$ about which we know nothing (other than $a_n \ll_\varepsilon n^{\varepsilon}$) is $\theta = \frac{17}{33}-\varepsilon$. It is an interesting question to ask how one can increase the length of the Dirichlet polynomials without specializing the coefficients $a_n$ too much.
\item[(4)] As mentioned in $\mathsection$1.3, the mixing of \eqref{bcy} and the Feng mollifier ($d=1$) is technically difficult. It is possible to obtain a better result for $\kappa$ by working with a mollifier of the form
\begin{align*}
\psi(s) &= \sum_{\ell=0}^K (-1)^\ell \sum_{\ell_1 + \ell_2 + \cdots + \ell_d = \ell} \binom{\ell}{\ell_1, \ell_2, \cdots, \ell_d} \nonumber \\
& \quad \times \sum_{n \le y_d} \frac{n^{\sigma_0-1/2}}{n^s} \frac{(\mu \star \Lambda^{\star \ell_1} \star \Lambda_2^{\star \ell_2} \star \cdots \star \Lambda_d^{\star \ell_d}) (n)}{(\log y_d)^{\sum_{r=1}^d r \ell_r}} P_{d,\ell} \bigg( \frac{\log(y_d/n)}{\log y_d} \bigg) \\
& +  \chi(s + \tfrac{1}{s}-\sigma_0) \sum_{hk \le y_2} \frac{\mu_2(h)h^{\sigma_0-1/2}k^{1/2-\sigma_0}}{h^sk^{1-s}}P_2 \bigg(\frac{\log (y_2/hk)}{\log y_2}\bigg).
\end{align*}
If necessary, then further pieces from \cite{krz02, sono} may be brought in. This would qualify as a painful calculation. Moreover, it would also be interesting to set $d=1$ and analyze how a large truncation of $\ell$ could yield terms from a higher degree $d \ge 2$. What would be a good balance between the degree $d$ and the truncation $\ell$?
\item[(5)] Is there a way to find a better expression (if possible a comfortable one such as \cite[Lemma 2.5.1]{cfkrs}) for the contour integrals that yield the logarithmic derivatives of $\zeta$? See for instance the end of $\mathsection$6.2 and $\mathsection$7.1. Another way to re-write $\mathfrak{M}_{\alpha,\beta}$ in \eqref{generalintegrandM} is
\begin{align*}
\quad \quad \quad \mathfrak{M}_{\alpha,\beta}(\mathbf{z},\mathbf{w};s,u) & = (-1)^{1 \times \ell_1 + 2 \times \ell_2 + \cdots d \times \ell_d}(-1)^{1 \times \barell_1 + 2 \times \barell_2 + \cdots d \times \barell_d} \\
& \quad \times \frac{\zeta(1+s+u)^{(L_d+1)(\barL_d+1)}\zeta(1+\alpha+\beta)}{\zeta(1+\beta+u)^{\barL_d+1} \zeta(1+\alpha+s)^{L_d+1}} \\
& \quad \times \bigg(\prod_{q=1}^d \prod_{i=1}^{\ell_q} \frac{\partial^q}{\partial z_{q,i}^q}\bigg) \bigg(\prod_{\barq=1}^d \prod_{j=1}^{\barell_\barq} \frac{\partial^\barq}{\partial w_{\barq,j}^\barq}\bigg) \bigg\{ \bigg(\prod_{q=1}^d \prod_{i=1}^{\ell_q} \prod_{\barq = 1}^{d} \prod_{j=1}^{\barell_\barq} \zeta(1+s+u+w_{q,i}+z_{\barq,j})\bigg) \\
& \quad \times \bigg(\prod_{q=1}^d \prod_{i=1}^{\ell_q} \frac{\zeta(1+\alpha+s+z_{q,i})}{\zeta^{\barL_d + 1}(1+s+u+z_{q,i})}\bigg) \bigg(\prod_{\barq=1}^d \prod_{j=1}^{\barell_\barq} \frac{\zeta(1+\beta+u+w_{\barq,j})}{\zeta^{L_d+1}(1+s+u+w_{\barq,j})}\bigg)\\ 
& \quad \times  A_{\alpha,\beta}(\boldsymbol{z}, \boldsymbol{w}, s, u) \bigg\}\bigg|_{\boldsymbol{z}= \boldsymbol{w} = 0}.
\end{align*}
Some attempts to obtain a closed formula for this problem indicate that it is much too cumbersome and that the easiest way is to proceed with mathematical software.
\item[(6)] In \cite{fengwu}, Feng and Wu used the $d=1$ version of a close variant of $\psi_d$ to show that infinitely often consecutive non-trivial zeros of the Riemann zeta-function differ by at least 2.7327 times the average spacing, and infinitely often they differ by at most 0.5154 times the average spacing, under RH. In other words, if 
\begin{align*}
\lambda = \limsup (\gamma-\gamma')\frac{\log \gamma}{2\pi} \quad \textnormal{and}\quad \mu = \liminf (\gamma-\gamma')\frac{\log \gamma}{2\pi},
\end{align*}
then, under RH, they prove that $\lambda > 2.7327$ and $\mu < 0.5154$. See also \cite{bui2, buimilng} among others. Since $\psi_{d=1}$ is a special case of a wider general family, can one improve the results of Feng and Wu with a higher degree $d$? In a recent paper of Bui and Milinovich \cite{buimilinovich}, these bounds are improved. The question remains whether the technique could be useful. Very recently, Goldston and Turnage-Butterbaugh \cite{gtb} have assumed RH and improved results in this direction. In particular using the weights developed by Wu (which have their root in Feng's mollifier), they have shown that there are infinitely many zeros of the zeta function whose differences are smaller than $0.50412$ times the average space. 
\item[(7)] In \cite{cis1, cis2}, the asymptotic large sieve is used along with Levinson's method to obtain lower bounds for the proportion of simple zeros on the critical line of the twists by primitive Dirichlet characters of a fixed $L$-function of degree $1$, $2$, or $3$. For a certain family of Dirichlet $L$-functions, Conrey, Iwaniec and Soundararajan prove that at least $56\%$ of the zeros are on the critical line and are simple. For aesthetic reasons, it would be desirable to increase this proportion to more than $\frac{3}{5}$. This would necessitate an analysis of more complicated mollifiers. The asymptotic large sieve technology would not apply directly, and additional difficulties would arise.
\end{enumerate}

\section{Acknowledgments}
The authors would like to acknowledge Matthew Faust of the Illinois Geometry Lab project for helping them write the code that produced their results as well as Siegfried Baluyot for proofreading an earlier version of this manuscript. Moreover, the authors would like to thank Hung Bui and Arindam Roy for useful discussions.

For part of this work the first author was supported by NSF grant DMS-1501982.

The authors are extremely grateful to the anonymous referees for their meticulous checking, for thoroughly reporting countless typos and inaccuracies as well as for their valuable comments. These corrections and additions have made the manuscript clearer and more readable.
%%%%%%%%%%%%%%%%%%%%%%%%%%%%%%%%%%%%%%%%%%%%%%%%%%%%%%%%%%%%%%%%%%%%%

\end{document}